\UseRawInputEncoding
\pdfoutput=1
\documentclass[11pt]{article}
\usepackage{amsfonts}
\usepackage{CJK}
\usepackage{graphics}
\usepackage{indentfirst}
\usepackage{cite}
\usepackage{latexsym}
\usepackage{amsmath}
\usepackage{amsthm}
\usepackage{amssymb}
\usepackage[dvips]{epsfig}
\usepackage{amscd}

\usepackage{color}

\newtheorem{theorem}{Theorem}[section]
\newtheorem{remark}{Remark}[section]

\newtheorem{lemma}[theorem]{Lemma}
\newtheorem{pro}{Proposition}[section]
\newtheorem{cor}[theorem]{Corollary}

\newcommand{\ltwo}{_{L^2}^2}

\renewcommand{\div}{ {\rm div }  }

\newcommand{\na}{\nabla }

\newcommand{\pa}{\partial}
\renewcommand{\r}{\mathbb{R}}

\newcommand{\dis}{\displaystyle}
\newcommand{\ia}{\int_0^T}

\newcommand{\bt}{{\hat\theta}}
\newcommand{\bl}{\begin{lemma}}
\newcommand{\el}{\end{lemma}}
\newcommand{\et}{\end{theorem}}
\newcommand{\ga}{\gamma}

\newcommand{\curl}{{\rm curl} }
\newcommand{\te}{\theta}

\newcommand{\al}{\alpha}
\newcommand{\de}{\delta}
\newcommand{\ve}{\varepsilon}
\newcommand{\la}{\label}

\newcommand{\p}{p(\rho)  }

\newcommand{\ka}{\kappa}

\newcommand{\bn}{\begin{eqnarray}}
\newcommand{\en}{\end{eqnarray}}
\newcommand{\bnn}{\begin{eqnarray*}}
\newcommand{\enn}{\end{eqnarray*}}

\newcommand{\bnnn}{\begin{eqnarray*}}
\newcommand{\ennn}{\end{eqnarray*}}

\newcommand{\ba}{\begin{aligned}}
\newcommand{\ea}{\end{aligned}}
\newcommand{\be}{\begin{equation}}
\newcommand{\ee}{\end{equation}}
\def\O{\Omega}
\def\p{\partial}
\def\norm[#1]#2{\|#2\|_{#1}}

\def\o{\omega}

\newcommand{\ep}{\varepsilon}
\newcommand{\n}{\rho}
\newcommand{\si}{\sigma}

\def\la{\label}

\def\na{\nabla}
\def\on{\hat\rho}

\def\tn{1}

\def\xl{\left}
\def\xr{\right}

\makeatletter   
\@addtoreset{equation}{section}
\makeatother       

\title{
Global Classical Solutions to Full Compressible Magnetohydrodynamic System with Large Oscillations and Vacuum in 3D Exterior Domains}

\date{}
\begin{document}

\author{Yunkun Chen$^a$, Yi Peng$^b$,  Xue Wang$^c$\thanks{
		Email addresses: cyk2013@gznu.edu.cn (Y. K. Chen), pengyi16@mails.ucas.ac.cn (Y. Peng), xuewa@amss.ac.cn (X. Wang)} \\[3mm]
a. Anshun University, Anshun 561000,    P. R. China; \\
{\normalsize b. College of Mathematics and Physics,}\\
{\normalsize Beijing University of Chemical Technology,
Beijing 100029,  P. R. China;} \\{\normalsize c.   School of Mathematical Sciences,}\\
{\normalsize  University of Chinese Academy of Sciences, Beijing 100049, P. R. China}}

\maketitle

\begin{abstract}
  The full compressible magnetohydrodynamic system in three-dimensional exterior domains   is  investigated. For the initial-boundary-value problem of this system with slip boundary condition for the velocity, adiabatic one for the temperature, and perfect
 one for the magnetic field,   the global existence and uniqueness of classical solutions is established, under the condition that the initial data are of small energy but possibly large oscillations. In particular, the initial density and temperature are both allowed to vanish. Moreover,   the large-time behavior of the classical solutions is also obtained.
\end{abstract}


\section{Introduction}

The motion of a viscous, compressible, and heat conducting magnetohydrodynamic (MHD) flow in a three-dimensional(3D) spatial domain $\Omega\subset\r^{3}$ can be described by the full compressible MHD system (see \cite{hw2008}):
\begin{equation}\label{h0}
\begin{cases}
\rho_t+ \mathop{\mathrm{div}}\nolimits(\rho u)=0,\\
(\rho u)_t+\mathop{\mathrm{div}}\nolimits(\rho u\otimes u)+ \nabla P=\mu\Delta u+(\mu+\lambda)\nabla\mathop{\mathrm{div}}u+(\na\times H)\times H,\\
\frac{R}{\gamma-1}((\rho \theta)_t+\div(\n\theta u))+P\mathop{\mathrm{div}} u=\kappa\Delta\theta\!+\!\lambda (\mathop{\mathrm{div}} u)^2\!+\!2\mu |\mathbb{D}(u)|^2\!+\!\nu|\mathop{\rm curl} H|^2, \\
H_t -\nabla \times (u \times H)=-\nu \nabla \times (\nabla \times H),
\\
\mathop{\mathrm{div}}\nolimits H=0,
\end{cases}
\end{equation}
where  $t\ge 0$ is time, $x\in \Omega$ is the spatial coordinate, and $\n$, $u=\left(u^1,u^2,u^3\right)^{\rm tr},$    $\te $, $P=R\n\te(R>0),$  and $H=\left(H^1,H^2,H^3\right)^{\rm tr}$ represent respectively the fluid density, velocity, absolute temperature, pressure,   and magnetic field.
The viscosity coefficients $\mu$ and $\lambda$  are constants satisfying the physical restrictions:
\be\la{h3} \mu>0,\quad 2 \mu + 3\lambda\ge 0.\ee
The heat-conductivity coefficient $\ka$ and magnetic resistivity coefficient $\nu$ both are  positive constants. $\ga>1$  is the adiabatic constant.

Let $\Omega=\r^3-\bar{D}$ be the exterior of a simply connected bounded  smooth domain $D\subset\r^3$, we impose on the system \eqref{h0}   the following initial data:
\be \la{h2}
(\rho,\n u, \n\te,H)(x,{t=0})=(\rho_0,\n_0u_0,\n_0\te_0,H_0)(x), \quad x\in \Omega,
\ee
 the  boundary conditions:
\be \la{h1}\begin{cases}
u\cdot n=0,~\mbox{curl} u\times n=0, ~~\text{on}~\p\Omega \times(0,T),\\ \na \theta\cdot n=0,~H\cdot n=0,~\mbox{curl} H\times n=0~~\text{on}~\p\Omega \times(0,T),\end{cases}
\ee
and the far field behavior:
\be \la{ch2} (\n,u,\te,H)(x,t)\rightarrow (1,0,1,0)\,\,\text{as}\,\, |x|\rightarrow\infty, \ee
where $n =(n^1, n^2, n^3)^{\rm{tr}}$ is the unit outward normal vector on $\partial\Omega$. Indeed, for fixed $\O$, there exists a positive constant  $d$ such that $\bar{D}\subset B_d$, then one can extend the unit outer normal $n$ to $\Omega$ such that $n\in C^\infty(\bar{\Omega})$ and $n\equiv 0  \mbox{ on }\r^3\setminus  B_{2d}$. The method of extension is not unique, and we fix one through out the paper.

There are a lot of literatures on the well-posedness and dynamic behavior of the solutions to the  compressible MHD system  due to its physical importance and mathematical challenges.
The  one-dimensional problem has been studied extensively by many people, see \cite{cw2002,ko1982,Wang2003,hss2021} and the references therein.
For the multi-dimensional case, the local existence of strong or classical solutions with initial vacuum for isentropic or full compressible MHD system  are showed in \cite{fy2009,lh2016,tg2016,xh2017,lh2015}.
The global classical solutions to the 2D Cauchy problem of compressible MHD system were first proved by Kawashima \cite{k1984} for the initial data close to a non-vacuum equilibrium in $H^3$-norm. Chen-Tan \cite{ct} extended this result to the 3D Cauthy problem. See also the generalizations to the  Cauthy problem  \cite{pg}  or the exterior domains \cite{llz2021} for the full compressible MHD system.

On the other hand, when vacuum state appears, the problem becomes more complicated and difficult.  Hu-Wang \cite{hw2008,hw2010} gave the global existence of renormalized weak solutions  with large data in 3D bounded domains. See  the relevant results in \cite{sh2012,lyz2013,df2006} for isentropic case or non-isentropic case.
The global existence and uniqueness of classical solutions to the  Cauchy problem of 3D isentropic compressible MHD system was established by Li et al. \cite{lxz2013}, where the initial energy is small but oscillations maybe  large  and the vacuum state is allowed. Later, the result was generalized by Hong et al. \cite{hhpz2017} to the large initial data with $\gamma-1$ and $\nu^{-1}$ are suitably small. For the 2D case, Lv et al. \cite{lsx2016} obtained the global existence and uniqueness of classical solutions  and some better a priori decay rates. Recently, Chen et al. \cite{chs2020-mhd} derived the global classical solutions to isentropic compressible MHD system with slip boundary conditions in 3D bounded domains for the regular initial data with small energy but possibly large oscillations and vacuum. More recently, the global well-posedness of strong and weak solutions of the isentropic compressible MHD system in 2D bounded domains with large initial data and vacuum was investigated by Chen et al. \cite{chs2021-mhd}.
As for the full compressible MHD system, the global strong solutions to the 3D Cauchy problem was investigated by Liu-Zhong \cite{lz2020-mhd,lz2021-fmhd,lz2022-fmhd} and Hou et al. \cite{hjp2022} under some certain small conditions, referring  \cite{fl2020,ls2019,tw2018,ww2017,zhu2015,ls2021} for more results to  compressible non-resistive or inviscid MHD equations.
Considering the exterior domains problem with slip boundary conditions, Chen et al. \cite{chs} obtained the global existence of classical solutions with small energy but possibly large oscillations to baratropic compressible MHD system.
Then the main aim of this paper is to extend this result to full compressible MHD system \eqref{h0}, which in fact is based \cite{C-L-L,lll} for compressible Navier-Stokes equations with slip boundary conditions in exterior domains and  \cite{ccw} for full compressible MHD system with slip boundary conditions in  bounded domains.

Before stating the main result, we first introduce the notations and conventions used throughout this paper.
For $1\le p\le \infty $ and integer $k\ge 0,$  we adopt the following notations for the standard homogeneous and inhomogeneous  Sobolev spaces:
\be\ba\notag \begin {cases}
L^p=L^p(\Omega),\quad W^{k,p}=W^{k,p}(\Omega),\quad H^k=W^{k,2},\\
 D^{k,p}=\{\left.f\in L^1_{loc}(\O)\right|\na^{k}f \in L^p(\O)\},\quad D^k=D^{k,2},\\
 H_s^2= \left.\left\{f\in H^2 \right | f\cdot n=0,\,\curl f\times n=0 \rm \,\,{on}\,\, \p\O\right\}.\end{cases}\ea\ee

Define the initial energy  $C_0$  as follows:
\be\la{e}\ba  C_0 =\int_{\Omega}\left(\frac{1}{2}\rho_0|u_0|^2\!+\! R (1\!+\!\rho_0\log\rho_0\!-\!\rho_0)\!+\!\frac{R}{\gamma-1}\rho_0(\theta_0\!-\!\log\te_0\!-\!1)\!+\!\frac{1}{2}|H_0|^2 \right)dx.\ea\ee
Then the  main result in this paper can be stated  as follows:

\begin{theorem}\la{th1}
Let $\Omega=\r^3-\bar{D}$ be the exterior of a simply connected bounded smooth domain $D\subset\r^3$.
For  given numbers $M>0$ (not necessarily
small), $q\in (3,6),$
  $\on> 2,$ and $\bt>1,$
suppose that the initial data $(\n_0,u_0,\te_0,H_0)$ satisfies \be\ba
\la{co3} \n_0-1\in H^2&\cap W^{2,q}, \quad  (u_0,H_0) \in H^2_s, \quad
\te_0-1\in H^1,\\
& \div H_0=0,\quad \na \te_0\cdot n|_{\p\O}=0 , \ea\ee
\be \la{co4} 0\le\inf\rho_0\le\sup\rho_0<\hat{\rho},\quad 0  \le\inf\te_0\le\sup\te_0\le \bt, \quad \|\na u_0\|_{L^2}+\|\na H_0\|_{L^2}\le M,
   \ee
   and the compatibility condition
\be
\la{co2}-\mu \Delta u_0-(\mu+\lambda)\na\div u_0+R\na (\n_0\te_0)-(\na\times H_0)\times H_0
=\sqrt{\n_0} g \ee
with  $g\in L^2. $ Then there exists a positive constant $\ve$
depending only
 on $\mu,$ $\lambda,$ $\nu$, $\ka,$ $ R,$ $ \ga,$  $\on,$ $\bt$, $\O$, and $M$ such that if
 \be
 \la{co14} C_0\le\ve,
   \ee  the   problem  (\ref{h0})--(\ref{ch2})
admits a unique global classical solution $(\rho,u,\te,H)$ in
   $\Omega\times(0,\infty)$ satisfying
  \be\la{h8}
  0\le\rho(x,t)\le 2\hat{\rho},\quad \te(x,t)\ge 0,\quad x\in \Omega,\,~~ t\ge 0,
  \ee
 and \be
\la{h9}\begin{cases}
\rho-1\in C([0,T];H^2\cap W^{2,q}),\\
(u,H)\in C([0,T]; H^2),\, \te-1   \in  C((0,T];H^2),\\
  u \in    L^2(0,T;H^3)\cap L^\infty(\tau,T; H^3\cap W^{3,q}),\\
\te-1 \in   L^\infty(0,T; H^1)\cap L^2(0,T;H^2)\cap L^\infty(\tau,T; H^4),\\
H\in   L^2(0,T;H^3)\cap L^\infty(\tau,T; H^4),\\
  (u_t,\te_t,H_t)\in
L^{\infty}(\tau,T;H^2)\cap H^1(\tau,T;H^1), \end{cases} \ee for any $0<\tau<T<\infty.$
Moreover, the following large-time behavior holds:
 \be\la{h11}
  \lim_{t\rightarrow \infty}\left( \|\n(\cdot,t)-1\|_{L^p} +\|\na u(\cdot,t)\|_{L^r }+\|\na \te(\cdot,t)\|_{L^r }+\|\na H(\cdot,t)\|_{L^r }\right)=0,  \ee
  for any $p\in  (2 ,\infty)$ and  $r\in [2,6)$.
\end{theorem}

Next, as a direct application of \eqref{h11}, the following Corollary \ref{th3} indicates the  large-time behavior of the gradient of density  when the vacuum state appears  initially. The proof is analogue to that of     \cite[Theorem 1.2]{hulx}.

\begin{cor} \la{th3}
In addition to the conditions  of Theorem \ref{th1}, assume further
that there exists some point $x_0\in \Omega$ such that $\rho
_0(x_0)=0.$ Then  the unique global    classical solution
$(\rho,u,\te,H)$ to the   problem (\ref{h0})--(\ref{ch2})
obtained in Theorem \ref{th1} has to blow up as $t\rightarrow
\infty,$ in the sense that for any $r>3,$
\be\notag\lim\limits_{t\rightarrow \infty}\|\nabla \rho(\cdot,t)
\|_{L^r }=\infty.\ee
\end{cor}

A few remarks are in order:

\begin{remark}
One can deduce from (\ref{h9}) and Sobolev imbedding  theorem that for any $0<\tau<T<\infty,$
\be \la{hk1}
(\n-1,\, \na\n, u,H) \in C(\overline{\Omega} \times[0,T] ),\quad   \te-1 \in C(\overline{\Omega} \times[\tau,T] ),\ee
and
\be \ba\notag
&(\na u,\,\na^2 u) \in  C( [0,T];L^2 )\cap L^\infty(\tau,T;  W^{1,q})\hookrightarrow  C(\overline{\Omega} \times[\tau,T] ),\\
(\na \te,\,\na^2 \te&,\,\na H,\,\na^2 H,\,u_t,\,\te_t,H_t) \in  C( [\tau,T];L^2 )\cap L^\infty(\tau,T;  H^2)\hookrightarrow  C(\overline{\Omega} \times[\tau,T] ),
\ea \ee
which combined with  (\ref{h0})  and (\ref{hk1}) implies
\be\notag \n_t\in   C(\overline{\Omega}\times[\tau,T] ).\ee
Therefore, the solution $(\n,u,\te,H)$ obtained in Theorem \ref{th1} is a classical one to  the   problem (\ref{h0})--(\ref{ch2}) in $\Omega\times (0,\infty).$
\end{remark}

\begin{remark}
Theorem \ref{th1} extends the global existence result of the  barotropic compressible MHD system studied in  Chen-Huang-Shi \cite{chs} to the full compressible MHD system, which is the first result concerning the global existence of classical solutions with initial vacuum to problem (\ref{h0})--(\ref{ch2}) in  the exterior domains. Although it's energy is small, the oscillations could be arbitrarily large. 
 \end{remark}

We now comment on the analysis of this paper. 
First we extend the local classical solutions away from vacuum (see Lemma \ref{th0}) globally in all time provided that the initial energy is suitably small (see Proposition \ref{pro2}), then let the lower bound of the initial density tend to zero to obtain the global classical solutions with vacuum.  It turns out that the key issue in this paper is to derive both the uniform  (in time) upper bound for the density and the time-dependent higher-order estimates of the smooth solution $(\rho, u,\theta,H)$, which are all independent of the lower bound of density.

Compared to the isentropic case \cite{chs}, the first difficulty  in this paper lies in the fact that the basic energy $E(t)$ defined in \eqref{enet} is unavailable due to the slip boundary conditions (see \eqref{la2.7}).
To overcome this difficulty, making an   a priori assumption on $A_2(T)$ (see \eqref{3.q2}), we then derive a ``weaker'' basic energy estimate (see Lemma \ref{a13.1})
\be\notag E(t) \le CC_0^{1/4},\ee
which will  cause some essential difficulties in later calculations.

Next,  based on the elaborate analysis of the maximum of temperature and   the structure of \eqref{h0}, we succeed in re-establishing the basic energy in small time $[0,\si(T)]$ and then deriving  the energy-like estimate on $A_2(T)$, which includes the key bounds on the $L^2$-norm (in both time and space) of $(\na u, \na \te,\na H)$,  provided the initial energy is small (see Lemma \ref{le3}).  Here, we make use of some estimates on  $(\dot u, \dot \te,H_t)$ due to Hoff \cite{Hof1} (see also Huang-Li \cite{H-L}), which contributes to  deal with the a priori estimates  $A_3(T)$ as well. However, in this process slip boundary conditions lead to a difficulty once again, i.e. the boundary integral terms (see \eqref{bz8}). In fact, they can be solved by the ideas in \cite{C-L} (see \eqref{bz3}), i.e.  using the fact $u=u^\perp\times n$ on $\p\O$ with $u^\perp =-u\times n$,  the boundary integral terms can only  be bounded by the gradients of velocity and efficient viscous flux. On the other hand,   we adopt some  ideas in  Cai-Li-L\"u \cite{C-L-L} to re-establish the estimates on $(\nabla u,~\div u,~\curl u)$  and $(\na H,~\curl H)$, which are used repeatedly in the whole analysis.

Thus, with the help of the lower estimates , we can derive   the  time-independent upper bound of the density smoothly (see Lemma \ref{le7}) and then obtain the higher-order estimates (see Section 4) under the compatibility condition on the velocity. It's noted that all the a priori estimates are independent of the lower bound of the initial density, thus after a standard approximation procedure, we can arrive at the global existence of classical solutions with vacuum.

The rest of this paper is organized as follows: In Section 2,  some elementary facts and inequalities  will be exhibited. To extend the local solutions to all time, we establish the  lower-order a priori estimates on classical solutions in Section 3 and the higher-order estimates  in Section 4.
Finally, with all a priori estimates at hand, the main result Theorem \ref{th1} is proved in Section 5.

\section{Preliminaries}\la{se2}


We begin with the following well-known local existence theory with strictly positive initial density, which can be proved by the standard contraction mapping argument as in \cite{choe1,Tani,M1}.

\begin{lemma}   \la{th0} Let $\O$ be as in Theorem \ref{th1}. Assume  that
 $(\n_0,u_0,\te_0,H_0)$ satisfies \be \la{2.1}
\begin{cases}
(\n_0-1,u_0,\te_0-1,H_0)\in H^3 \quad \inf\limits_{x\in\Omega}\n_0(x) >0, \quad \inf\limits_{x\in\Omega}\te_0(x)> 0,\\
u_0\cdot n=0,\quad\curl u_0\times n=0,\quad \na \te_0\cdot n =0~~{on}~\p\Omega,\\
H_0\cdot n=0,\quad\curl H_0\times n=0~~{on}~\p\Omega,\quad \div H_0=0.
\end{cases}\ee
Then there exist  a small time
$T_0>0$ and a unique classical solution $(\rho , u,\te,H )$ to the problem  (\ref{h0})--(\ref{ch2}) on $\Omega\times(0,T_0]$ satisfying
\be\la{mn6}
  \inf\limits_{(x,t)\in\Omega\times (0,T_0]}\n(x,t)\ge \frac{1}{2}
 \inf\limits_{x\in\Omega}\n_0(x), \ee and
 \be\la{mn5}
 \begin{cases}
 ( \rho-1,u,\te-1,H) \in C([0,T_0];H^3),\quad
 \n_t\in C([0,T_0];H^2),\\   (u_t,\te_t,H_t)\in C([0,T_0];H^1),
 \quad (u,\te-1,H)\in L^2(0,T_0;H^4).\end{cases}\ee
 \end{lemma}

 \begin{remark} Applying the same arguments as in \cite[Lemma 2.1]{H-L} to  the classical solution $(\rho , u,\te,H )$  obtained in Lemma \ref{th0}, one gets
\be \notag\begin{cases} (t u_{t},  t \te_{t},t H_{t}) \in L^2( 0,T_0;H^3)  ,\quad (t u_{tt},  t\te_{tt},tH_{tt}) \in L^2(
0,T_0;H^1), \\  (t^2u_{tt},  t^2\te_{tt},t^2H_{tt}) \in L^2( 0,T_0;H^2), \quad
 (t^2u_{ttt},  t^2\te_{ttt},t^2H_{ttt}) \in L^2(0,T_0;L^2).
\end{cases}\ee
Moreover, for any    $(x,t)\in \Omega\times [0,T_0],$ the following estimate holds:
   \be\la{mn2}
\te(x,t)\ge
\inf\limits_{x\in\Omega}\te_0(x)\exp\left\{-(\ga-1)\int_0^{T_0}
 \|\div u\|_{L^\infty}dt\right\}.\ee
\end{remark}

The following classical Gagliardo-Nirenberg-Sobolev-type inequality (see \cite{fcpm}) will be used later frequently.

\begin{lemma}
 \la{11} Assume that $\Omega$ is the exterior of a simply connected Lipschitz domain $D$ in $\r^3$. For $r\in [2,6]$, $p\in(1,\infty)$, and $ q\in
(3,\infty),$ there exist  positive
 constants $C$ which may depend  on $r,\,p,$ and $q$, such that for $f\in H^1(\Omega)$,
 $g\in L^p(\Omega)\cap D^{1,q}(\Omega)$,  and $\varphi,\,\psi\in H^2(\Omega)$,
 \be
\la{g1}\|f\|_{L^r} \le C \| f\|_{L^2}^{{(6-r)}/{2r}}  \|\na f\|_{L^2}^{{(3r-6)}/{2r}},\ee
\be \la{g2}\|g\|_{C\left(\overline{\Omega}\right)}
       \le C \|g\|_{L^p}^{p(q-3)/(3q+p(q-3))}\|\na g\|_{L^q}^{3q/(3q+p(q-3))},\ee
\be\la{hs} \|\varphi\psi\|_{H^2}\le C \|\varphi\|_{H^2}\|\psi\|_{H^2}.\ee
\end{lemma}

Then, as a result of \eqref{g1}, one has the following Sobolev-type  inequality (see  \cite[Lemma 2.3]{lll}), which  provides an estimate on $\te-1$ (see \eqref{a2.17}).
\begin{lemma}
	\la{lll} Let the function $g(x)$ defined in $\Omega$ be non-negative and satisfy $g(\cdot)-1\in L^2(\Omega)$. Then there exists a universal positive constant C such that for $s\in[1,2]$ and any open set $\Sigma\subset\Omega$, the following estimate holds
	\be\la{eee8}\int_\Sigma|f|^sdx\leq C\int_\Sigma g|f|^sdx+C\|g-1\|_{L^2(\Omega)}^{(6-s)/3}\|\na f\|_{L^2(\Omega)}^s,\ee
	for all $f\in\{D^1(\Omega)~\big|~g|f|^s\in L^1(\Sigma)\}$.
\end{lemma}

Next,  to estimate $\|\nabla u\|_{L^\infty}$ for the further higher order estimates, we need the following Beale-Kato-Majda-type inequality \cite{bkm}, which in fact is a critical case of Sobolev inequality,
whose detailed proof is similar to that of the case of slip boundary condition  in \cite[Lemma 2.7]{C-L} (see also \cite{h1x}).
\begin{lemma}\la{le9}
 	Let  $\Omega=\r^3-\bar{D}$ be as in Theorem \ref{th1}. For $3<q<\infty$, assume that $u\in \{ f\in L^1_{loc}|\na f\in L^2(\O)\cap D^{1,q}(\O) \,and \,f\cdot n=0, \curl f\times n=0 \text{ on } \partial \Omega \}$,  then there is a constant  $C=C(q)$ such that
\bnn\ba
\|\na u\|_{L^\infty}\le C\left(\|{\rm div}u\|_{L^\infty}+\|\curl u\|_{L^\infty} \right)\ln(e+\|\na^2u\|_{L^q})+C\|\na u\|_{L^2} +C.
\ea\enn
\end{lemma}

 The div-curl-type estimates of $\na v$ with  boundary condition $v\cdot n=0$ or $v\times n=0$ on $\partial\Omega$ are obtained in the following Lemmas \ref{crle1}--\ref{crle5}, whose proof can be found in \cite[Theorem 3.2]{vww}, \cite[Theorem 5.1]{lhm}, and \cite[Lemma 2.9]{C-L-L}, respectively.

\begin{lemma}  \cite[Theorem 3.2]{vww} \la{crle1}
Let $\Omega=\r^3-\bar{D}$ be the exterior of a simply connected bounded domain $D\subset\r^3$  with $C^{1,1}$ boundary.
For $v\in D^{1,q}(\O)$ with $v\cdot n=0$ on $\partial\Omega$, it holds that
	\be\la{ljq01}\|\nabla v\|_{L^{q}}\leq C(q,\O)(\|\div v\|_{L^{q}}+\|\curl v\|_{L^{q}}),~~~\,\,\,for\, any\,\, 1<q<3,\ee
\be\la{ok}\|\nabla v\|_{L^q}\leq C(q,\O)(\|\div v\|_{L^q}+\|\curl v\|_{L^q}+\|\nabla v\|_{L^2}),~~~\,\,\,for\, any\,\, 3\leq q<+\infty.\ee
\end{lemma}

\begin{lemma} \cite[Theorem 5.1]{lhm}  \la{crle2}
	Let $\Omega$ be as in Lemma \ref{crle1}, for any $v\in W^{1,q}(\O)\,\,(1<q<+\infty)$ with $v\times n=0$ on $\partial\Omega$, it holds that
\be\la{oh}\|\nabla v\|_{L^q}\leq C(q,\O)(\|v\|_{L^q}+\|\div v\|_{L^q}+\|\curl v\|_{L^q}).\ee
\end{lemma}

\begin{lemma}  \cite[Lemma 2.9]{C-L-L} \la{crle5}
Let $\Omega$  be as in Theorem \ref{th1}.
For any $q\in[2,4] $,   if $v\in \{ D^{1,2}(\O)| v(x)\rightarrow 0  \mbox{ as } |x|\rightarrow\infty  \}$,   then
\be\la{eee0}\ba\|v\|_{L^q(\partial\Omega)}\leq C(q,\O)\|\na v\|_{L^2(\Omega)}.\ea\ee
Moreover, for $p\in[2,6] $, $k\geq 1,$   every $v\in \{f\in D^{1,2}(\O)|\na f\in W^{k,p}(\O),\, f(x)\rightarrow 0  \mbox{ as } |x|\rightarrow\infty  \}$  with $v\cdot n|_{\partial\Omega}=0$ or $v\times n|_{\partial\Omega}=0$ satisfies
	\be\la{uwkq}\ba\|\nabla v\|_{W^{k,p}}\leq C(k,p,\O)(\|\div v\|_{W^{k,p}}+\|\curl v\|_{W^{k,p}}+\|\nabla v\|_{L^2}).\ea\ee
\end{lemma}

Next, we state the following estimate  on $\na \dot u$ with $u \cdot n|_{\p \O}=0$ (see \cite[Lemma 2.9] {lll}).

\begin{lemma}\la{uup1}
	Let  $\Omega=\r^3-\bar{D}$  be as in Theorem \ref{th1}.
	Assume that $u$ is smooth enough and $u \cdot n|_{\p \O}=0$,
	then there exists a generic positive constant $C=C(\O)$   such that
	\be\la{tb11}\ba
	\|\nabla\dot{u}\|_{L^2}\le C(\|\div \dot{u}\|_{L^2}+\|\curl \dot{u}\|_{L^2}+\|\nabla u\|_{L^4}^2+\|\nabla u\|_{L^2}^2).
	\ea\ee
\end{lemma}

The following Gr\"{o}nwall-type inequality will be used to get the uniform (in time) upper bound of the density (see \cite[Lemma 2.5]{H-L}).
\begin{lemma} \la{le1}
	Let the function $y\in W^{1,1}(0,T)$ satisfy
	\be \notag
  y'(t)+\al y(t)\le  g(t)\mbox{  on  } [0,T] ,\quad y(0)=y_0,
  \ee
  where $\al$ is a positive constant and  $ g \in L^p(0,T_1)\cap L^q(T_1,T)$  for some $p,\,q\ge 1, $  $T_1\in [0,T].$ Then
	\be \la{2.34}
	\sup_{0\le t\le T} y(t) \le |y_0| + (1+\al^{-1}) \left(\|g\|_{L^p(0,T_1)} + \|g\|_{L^q(T_1,T)}\right).
	\ee
\end{lemma}

Finally, we mainly show some  elliptic estimates, which are the important ingredients of the whole analysis.

For the Neumann boundary value problem
\be\la{cxtj1}\begin{cases}
	-\Delta v=\div f,\,\, &\text{in}~ \Omega, \\
	\frac{\partial v}{\partial n}=-f\cdot n,\,&\text{on}~\partial\Omega,\\
	\nabla v\rightarrow0,\,\,&\text{as}\,\,|x|\rightarrow\infty,
\end{cases} \ee
 by virtue of \cite[Lemma 5.6]{ANIS}, we have the following conclusion.
\begin{lemma}   \la{zhle}
Considering the system (\ref{cxtj1}), for $q\in(1,\infty),$ it holds
	
	(1) If $f\in L^q(\O)$, then there exists a unique (modulo constants) solution $v\in D^{1,q}(\O)$ such that
	$$\|\nabla v\|_{L^q}\leq C(q,\O)\|f\|_{L^q}.$$
	
	(2) If $f\in W^{k,q}(\O)$ with $k\geq 1$, then $\na v\in W^{k,q}(\O)$ and
	$$\|\nabla v\|_{W^{k,q}}\leq C(k,q,\O) \|f\|_{W^{k,q}}.$$	
	In particular, if $f\cdot n=0$ on $\partial\Omega$, one gets
	$$\|\nabla^2v\|_{L^q}\leq C(q,\O) \|\div f\|_{L^q}.$$	
\end{lemma}


Denoting the material derivative of $f$, the  effective viscous flux $G$, and the vorticity $\o$ respectively by
\be\begin{cases} \la{hj1} D_t f \triangleq \dot f\triangleq f_t+u\cdot\nabla f,\\ G\triangleq(2\mu + \lambda)\div u -  R(\rho\te-1)-|H|^2/2,\\ \omega\triangleq\curl u.
\end{cases}\ee
Now, we establish some nacessary estimates for them. In fact,  the standard $L^p$-estimate for the following elliptic equation:
\begin{equation}\la{ah}
\begin{cases}
\Delta G=\div(\rho\dot{u}-H \cdot \nabla H),~~ &\text{in}\,\,\Omega,\\ \frac{\partial G}{\partial n}=(\rho\dot{u}-H \cdot \nabla H 
)\cdot n,\,\, &\text{on}\,\, \partial\Omega,\\
\na G\rightarrow 0,\,\,&\text{as}\,\,|x|\rightarrow\infty.
\end{cases}
\end{equation}
combined with Lemmas \ref{11}, \ref{crle1}, and \ref{zhle} yields the following  estimates, whose proof is analogue to that of \cite[Lemma 2.8]{lll} (see also \cite[Lemma 2.9]{C-L}).
\begin{lemma} \la{le4}
  Assume  $\Omega=\r^3-\bar{D}$  is the same as in Theorem \ref{th1}.   Let $(\rho,u,\theta,H)$ be a smooth solution of (\ref{h0})--(\ref{ch2}) on $\O \times (0,T]$. Then for any $p\in[2,6],$ there exists a positive constant $C$ depending only on $p$,  $\mu$, $\lambda$, $R$, and $\O$ such that
\be\label{tdh1}\|\nabla H\|_{L^p}\le C(\|\curl H\|_{L^p}+\|\curl H\|_{L^2}),\ee
\be\la{h19}\|\nabla G\|_{L^p}\leq C(\|\rho\dot{u}\|_{L^p}+\|H\! \cdot \nabla H\|_{L^p}),\ee
\be\la{h191}\|\nabla\omega\|_{L^p}\leq  C(\|\rho\dot{u}\|_{L^p}+\|H \cdot \nabla H\|_{L^p}+\|\rho\dot{u}\|_{L^2}+\|H \cdot \nabla H\|_{L^2}+\|\nabla u\|_{L^2}),\ee
\be\ba\la{h20}\|G\|_{L^p}\leq &  C(\|\rho\dot{u}\|_{L^2}+\|H\cdot\nabla H\|_{L^2})^{(3p-6)/2p}\\
&\cdot(\|\nabla u\|_{L^2}+\|\n\te-1\|_{L^2}+\||H|^2\|_{L^2})^{(6-p)/2p},\ea\ee
\be\la{h21}\|\omega\|_{L^p} \leq  C(\|\rho\dot{u}\|_{L^2}+\|H \cdot \nabla H\|_{L^2})^{(3p-6)/2p}\|\nabla u\|_{L^2}^{(6-p)/2p}+C\|\nabla u\|_{L^2},\ee
\be\ba\la{h17}
\|\nabla u\|_{L^p}\leq & C(\|\rho\dot{u}\|_{L^2}+\||H||\nabla H|\|_{L^2}+\|\n\te-1\|_{L^6})^{(3p-6)/2p}\|\nabla u\|_{L^2}^{(6-p)/2p}\\
&+C\|\nabla u\|_{L^2}.
\ea\ee
\end{lemma}
\begin{remark}
By virtue of (\ref{oh}) and (\ref{uwkq}), we will obtain the estimate on $\nabla^{k+2} H$. Precisely, for $p\in [2,6]$ and $k\geq 0$, there exists a positive constant $C=C(p,k,\O)$ such that
\be\ba\label{2tdh}
\|\nabla^{k+2} H\|_{L^p}&\leq C \|\curl H\|_{W^{k+1,p}}+C\|\na H\|_{L^2} \\
&\leq C(\|\curl^2  H\|_{W^{k,p}}+\|\na H \|_{L^2}).
\ea\ee
On the other hand, we can get the estimate on $\nabla^{k+2}u$, which will be devoted to getting the higher order estimates in Section \ref{se4}. In fact, applying  Lemma \ref{zhle} to elliptic equation (\ref{ah}) along with Lemmas  \ref{crle2}, \ref{crle5},  and \ref{le4} yields that for $p\in[2,6]$ and $k\geq 0$, there exists a positive constant $C=C(p,k,\O)$ such that
\be\ba\label{2tdu}
&\|\nabla^{k+2}u\|_{L^p}\\
&\le C(\|\div u\|_{W^{k+1,p}}+\|\omega\|_{W^{k+1,p}}+\|\na u\|_{L^2})\\
&\le C(\|\n\dot{u}\|_{W^{k,p}}+\|H\cdot\na H\|_{W^{k,p}}+\|\na P\|_{W^{k,p}}+\|\na|H|^2\|_{W^{k,p}})\\
&\quad+C(\|\rho\dot{u}\|_{L^2}+\|H\cdot \na H\|_{L^2}+\|\n\te-1\|_{L^2}+\||H|^2\|_{L^2}+\|\nabla u\|_{L^2}).
\ea\ee
\end{remark}

\section{\la{se3} A priori estimates (I): lower-order estimates}

In this section, we aim to derive the time-independent a priori estimates for the local-in-time smooth solution to the problem (\ref{h0})--(\ref{ch2}) obtained in Lemma \ref{th0}.
Let $(\n,u,\te,H)$ be a smooth solution to the  problem (\ref{h0})--(\ref{ch2})  on $\Omega\times (0,T]$ for some fixed time $T>0,$ with  the initial data $(\n_0,u_0,\te_0,H_0)$ satisfying \eqref{2.1}.

For
$\si(t)\triangleq\min\{1,t\}, $  we  define
$A_i(T)(i= 1,\cdots,4)$ as follows:
  \be\la{As1}
  A_1(T) \triangleq \sup_{t\in[0,T] }\left(\|\na u\|_{L^2}^2+\|\nabla H\|_{L^2}^2\right)
+ \int_0^{T}(\|\sqrt\n\dot{u}\|_{L^2}^2 +\|\curl^2 H\|_{L^2}^2+\|H_t\|_{L^2}^2)dt,
  \ee
  \be\label{AS1}
  A_2(T) \triangleq \frac{R}{2(\ga-1)}\sup_{t\in[0,T] }\int\n (\te-1)^2dx
  +\int_0^T\left( \|\na u\|_{L^2}^2+\|\na H\|_{L^2}^2+\|\na \te\|_{L^2}^2\right)dt,
  \ee
 \be\ba \label{AS2}
 A_3(T)&\triangleq  \sup_{t\in(0,T]}\sigma\left(\|\nabla u\|_{L^2}^2+\|\nabla H\|_{L^2}^2\right)\\
 &\quad+\int_0^{T}\sigma(\|\sqrt{\rho}\dot{u}\|_{L^2}^2+\|\curl^2 H\|_{L^2}^2+\|H_t\|_{L^2}^2)dt,\ea\ee
\be\ba\label{As3}
A_4(T)&\triangleq\sup_{t\in(0,T]}\sigma^2(\|\sqrt{\rho}\dot{u}\|_{L^2}^2+\|\curl^2 H\|_{L^2}^2+ \|H_t\|_{L^2}^2+\|\nabla\theta\|_{L^2}^2)\\
&\quad+\int_0^{T}\sigma^2(\|\nabla\dot{u}\|_{L^2}^2+\|\nabla H_t\|_{L^2}^2+\|\sqrt{\rho}\dot{\theta}\|_{L^2}^2)dt.
  \ea\ee
Then  we have the following key a priori estimates on $(\n,u,\te,H)$.
\begin{pro}\la{pr1}
For  given numbers $M>0$, $\on> 2,$  and $\bt> 1,$ assume further that $(\rho_0,u_0,\te_0,H_0) $  satisfies
\be \la{3.1}
0<\inf \rho_0 \le\sup \rho_0 <\on,\quad 0<\inf \te_0 \le\sup \te_0 \le \bt, \quad \|\na u_0\|_{L^2}+\|\na H_0\|_{L^2}  \le M.
\ee
Then there exist  positive constants $K $ and $\ep_0$ both depending on $\mu,\,\lambda,\,\nu,\, \ka,\, R,\, \ga,\, \on,\,\bt $, $\O$, and $M$ such that if $(\rho,u,\te,H)$ is a smooth solution to the problem (\ref{h0})--(\ref{ch2}) on $\O\times (0,T]$ satisfying
\be \la{z1}
0<  \rho\le 2\on, \,\,\, A_1(\sigma(T))\le 3 K, \,\,\, A_2(T) \le 2C_0^{1/4}, \,\,\, A_3(T)+A_4(T)\le 2C_0^{1/6},
\ee
the following estimates hold:
\be \la{zs2}
0< \rho\le 3\on/2,\,\,\, A_1(\sigma(T))\le 2 K, \,\,\, A_2(T) \le C_0^{1/4}, \,\,\, A_3(T)+A_4(T)\le C_0^{1/6},
\ee
provided \be\la{z01}C_0\le \ve_0.\ee      \end{pro}

\begin{proof}
Indeed, we conclude Proposition \ref{pr1} as a result of
the following Lemmas \ref{le2}, \ref{le6}, \ref{le3}, and \ref{le7},  with $\ve_0$ as in (\ref{xjia11}).
\end{proof}

In this section, we always assume that $C_0\le 1$ and let $C$ denote some generic positive constant depending only on $\mu$,  $\lambda$, $\nu$, $\ka$,  $R$, $\ga$, $\on$, $\bt$, $\O$,  and $M,$ and we write $C(\al)$ to emphasize that $C$ may depend  on $\al.$

We start with the following estimate on the basic energy.

\begin{lemma}\la{a13.1} Under the conditions of Proposition \ref{pr1}, there exists a positive constant $C$ depending only on $\mu$,  $R$,    and  $\on$  such that if $(\rho,u,\te,H)$ is a smooth solution to the problem (\ref{h0})--(\ref{ch2})  on $\Omega\times (0,T] $ satisfying
\be\la{3.q2}
0<\n\le 2\on ,\quad A_2(T)\le 2C_0^{1/4},
\ee
the following estimate holds:
\be \la{a2.112}
\sup_{0\le t\le T}\int\left( \n |u|^2+(\n-\tn)^2+|H|^2\right)dx \le C  C_0^{1/4}.
\ee
\end{lemma}

\begin{proof}
First, by virtue of (\ref{3.1}) and  (\ref{mn2}), we have that for all $ (x,t)\in \Omega\times(0,T),$
\be \la{3.2}
\te(x,t)>0 .
\ee

Next, note that
\be\notag
\Delta u = \na \div u - \na \times \o,\quad(\na\times H)\times H=H\cdot\na H-\na|H|^2/2,
\ee
 one can rewrite $(\ref{h0})_2$  as
\be\la{a11}\ba
\n (u_t+  u\cdot \na u )&=(2\mu+\lambda)\na{\rm div}u- \mu \na \times \o-\na P+H\cdot\na H-\nabla|H|^2/2.\ea\ee
Similarly, some simple calculations show that $\eqref{h0}_3$ is equivalent to
\be\la{mn}
H_t+u \cdot \nabla H-H \cdot \nabla u+ H \mathop{\mathrm{div}} u+\nu \nabla \times \mathop{\rm curl} H=0.
\ee

Denote the basic energy  by
\be \label{enet}E(t) \triangleq \int  \left( \frac{1}{2}\n |u|^2+R(1+\n\log\n-\n)+\frac{R}{\ga-1}\n(\te-\log\te-1)+\frac12|H|^2\right)dx. \ee
Then, multiplying $(\ref{a11})$, $(\ref{h0})_3$, and $(\ref{mn})$    by $u$, $1-\te^{-1}$, and $H$ respectively, summing them up, and  integrating the resulting equality over $\Omega$ by parts,  we deduce from $(\ref{h0})_1$,  \eqref{h1},  and  \eqref{ch2}   that
\be\la{la2.7}\ba
&E'(t)+\int \left( \frac{\lambda(\div u)^2+2\mu |\mathfrak{D}(u)|^2+\nu|\curl H|^2}{{\te}}+\ka \frac{|\na \te|^2}{\te^2} \right)dx \\
&= -\mu \int \left(|\o|^2+2(\div u)^2 - 2 |\mathfrak{D}(u)|^2\right)dx\\
&\le 2\mu \int  |\na u|^2 dx.
\ea\ee
Integrating \eqref{la2.7} with respect to $t$ over $(0,T)$ and using \eqref{3.q2}, \eqref{h3}, and \eqref{3.2} yield
 \be\la{a2.8}\ba
&\sup_{0\le t\le T} E(t)
\le C_0+ 2\mu \int_{0}^{T} \int  |\na u|^2 dxdt\le C C_0^{1/4},\ea\ee
which as well as
\be\la{a2.9}\ba
 (\n-1)^2\ge 1+\n\log\n-\n&=(\n-1)^2\int_0^1\frac{1-\al}{\al (\n-1)+1}d\al  \ge \frac{(\n-1)^2}{ 2(2\on+1)  }
 \ea\ee
infers (\ref{a2.112}).  The proof of Lemma \ref{a13.1} is  completed.
\end{proof}

The following lemma establishes the estimate on  $A_1(\sigma(T))$.
\begin{lemma}\la{le2}
	Under the conditions of Proposition \ref{pr1}, there exist positive constants  $K $  and $\ep_1 $ both depending only  on $\mu,\,\lambda,\,\nu,\, R,\, \ga,\, \on, $ $\O$, and $M$ such that if  $(\rho,u,\te,H)$ is a smooth solution to the problem  (\ref{h0})--(\ref{ch2}) on $\Omega\times (0,T] $ satisfying
	\be\la{3.q1}  0<\n\le 2\on ,\quad A_2(T)\le 2C_0^{1/4},\quad A_1(\sigma(T))\le 3K,\ee
	the following estimate holds:
	\be\la{h23} A_1(\sigma(T))\le 2K ,  \ee
	provided   $C_0\le \ep_1$.
\end{lemma}

\begin{proof}
First, it follows form \eqref{g1}, \eqref{g2}, and \eqref{2tdh} that
\be\la{go}
\|H\|_{L^\infty} \leq C \|H\|_{L^6}^{1/2}\|\nabla H\|_{L^6}^{1/2}\leq C\|\nabla H\|_{L^2}+C\|\nabla H\|_{L^2}^{1/2}\|\curl^2 H\|_{L^2}^{1/2},\ee
which combined with  \eqref{mn},  \eqref{g1}, and  \eqref{2tdh} shows that
\be\ba\la{k88}
\nu\|\curl^2H\|_{L^2}
&\le  \|H_t\|_{L^2}+\|H \cdot \nabla u-u \cdot \nabla H-H \div u\|_{L^2}\\
&\leq  \|H_t\|_{L^2}+C \|\nabla u\|_{L^2}\|H\|_{L^\infty}+C\|u\|_{L^6}\|\nabla H\|_{L^3}\\
&\leq  \|H_t\|_{L^2}+C \|\nabla u\|_{L^2}\|\nabla H\|_{L^2}^{1/2}(\|\curl^2  H\|_{L^2}+\|\na H\|_{L^2})^{1/2} \\
&\leq  \frac{\nu}{3}\|\curl^2  H\|_{L^2}+\|H_t\|_{L^2}+C(\|\nabla u\|_{L^2}+\|\nabla u\|^2_{L^2})\|\nabla H\|_{L^2}.
\ea\ee
Thus we have further estimates on $H$. More precisely, by \eqref{go} and \eqref{k88}, it holds
\be\la{go1}
\|H\|_{L^\infty} \leq C\|\nabla H\|_{L^2}(\|\nabla u\|_{L^2}+1)+C\|\nabla H\|_{L^2}^{1/2}\|H_t\|_{L^2}^{1/2},\ee
\be\ba\la{goo1}\||H||\nabla H|\|_{L^2} &\leq \|H\|_{L^\infty}\|\nabla H\|_{L^2}\\&\leq C\|\nabla H\|_{L^2}^2(\|\nabla u\|_{L^2}+1)+C\|\nabla H\|_{L^2}^{3/2}\|H_t\|_{L^2}^{1/2}.\ea\ee

Next, by virtue of   \eqref{3.q1}, \eqref{g1}, and \eqref{a2.112}, we have that for any $p\in [2,6],$
\be\la{p}\ba  \|\rho\te-1\|_{L^p}&= \| \rho(\te-1) + (\rho-1)\|_{L^p}\\
 &\le \|\n(\te-1)\|_{L^2}^{(6-p)/(2p)}
 \|\n(\te-1)\|_{L^6}^{ 3(p-2)/(2p)}+ \|\n-1\|_{L^p}
 \\&\le CC_0^{(6-p)/(16p)}
 \|\na\te \|_{L^2}^{ 3(p-2)/(2p)}+ CC_0^{1/(4p)},
  \ea\ee
which along with (\ref{h17}) and \eqref{goo1} indicates
\be\ba \la{3.30}  \|\na
u\|_{L^6} &\le C \left(  \|\n \dot u\|_{L^2}+\||H||\na H|\|_{L^2}+\|\na u\|_{L^2}+ \|\na \te\|_{L^2}+C_0^{1/24}\right)\\
&\le C \left(  \|\n \dot u\|_{L^2}+\|H_t\|_{L^2}+\|\na u\|_{L^2}+ \|\na \te\|_{L^2}+C_0^{1/24}\right)\\
&\quad+C\|\na H\|_{L^2}^2(\|\na H\|_{L^2}+\|\na u\|_{L^2}+1).
\ea\ee

Now, integrating $(\ref{mn})$  multiplied by $2H_t $ over $\Omega $ by parts, we deduce from   \eqref{h1}  and \eqref{k88} that
\be\ba\la{tdh-3}
& \left(\nu\|\curl H\|_{L^2}^2\right)_t+\frac{3}{2}\|H_t\|^2_{L^2} \\
&\leq  C\int |H \cdot \nabla u-u \cdot \nabla H-H \div u|^2 dx  \\
&\leq  C \|\nabla u\|_{L^2}^2\|\nabla H\|_{L^2}(\|\curl^2  H\|_{L^2}+\|\na H\|_{L^2}) \\
&\leq  \frac{1}{4}\| H_t\|^2_{L^2}+C(\|\nabla u\|^2_{L^2}+\|\nabla u\|^4_{L^2})\|\nabla H\|_{L^2}^2.
\ea\ee

Multiplying \eqref{a11} by $2u_t$ and integrating the resulting equality by parts  yield
 \be\ba \la{hh17}
 &\frac{d}{dt}\int \left(  {\mu} |\o|^2+ (2\mu+\lambda)(\div u)^2\right)dx+ \int\rho |\dot u|^2dx \\
&\le 2\int  P\div u_t dx+ \int \n|u\cdot \na u|^2dx+\int(2H\cdot\na H-\na |H|^2)\cdot u_t dx\\
 &=  2R\frac{d}{dt}\int  (\rho\te-1) \div u  dx-2\int P_t \div u dx+\int \n|u\cdot \na u|^2dx\\
 &\quad+\frac{d}{dt}\int(2H\cdot\na H-\na |H|^2)\cdot u dx-\int(2H\cdot\na H-\na |H|^2)_t\cdot u dx\\
 &= 2R\frac{d}{dt}\int  (\rho\te-1) \div u  dx-\frac{R^2}{2\mu+\lambda}\frac{d}{dt}\int (\rho\te-1)^2 dx\\
 &\quad-\frac{4}{2\mu+\lambda}\int P_t (2G+|H|^2) dx+ \int \n|u\cdot \na u|^2dx\\
 &\quad+\frac{d}{dt}\int(2H\cdot\na H-\na |H|^2)\cdot u dx-\int(2H\cdot\na H-\na |H|^2)_t\cdot u dx ,
 \ea\ee
where in the last equality  we have used (\ref{hj1})$_1$.

Denote
\be\la{po}\ba
B_0(t)&\triangleq \int \left(  {\mu} |\o|^2+ (2\mu+\lambda)(\div u)^2+\frac{R^2}{2\mu+\lambda}(\rho\te-1)^2\right)dx\\
&\quad-\int\left(2R(\n\te-1)\div u+\left(2H\cdot\na H-\na |H|^2\right)\cdot u\right)dx,
\ea\ee
then \eqref{hh17} can be rewritten as
\be\la{k0}\ba
B_0'(t)+\int\rho |\dot u|^2dx
&\le -\frac{4}{2\mu+\lambda}\int P_t (2G+|H|^2) dx+\int \n|u\cdot \na u|^2dx\\
&\quad-\int(2H\cdot\na H-\na |H|^2)_t\cdot u dx\triangleq \sum_{i=1}^3I_i .
\ea\ee
$I_i(i=1,2,3)$ can be estimated respectively.

Noticing that (\ref{h0})$_3$ implies
\be \la{op3} \ba
P_t=&-\div (Pu) -(\gamma-1) P\div u+(\ga-1)\ka \Delta\te\\&+(\ga-1)\left(\lambda (\div u)^2+2\mu |\mathfrak{D}(u)|^2+\nu |\curl H|^2\right),
\ea\ee
we thus obtain after using  integration by parts,  (\ref{g1}), (\ref{hj1})$_1$, (\ref{h19}), \eqref{2tdh}, (\ref{p}),  (\ref{3.30}), \eqref{3.q1},  \eqref{k88}, \eqref{goo1}, and (\ref{a2.112}) that
\be\la{a16}\ba
|I_1|
&\le C\int (P |u|+|\na\te|)(|\na G|+|\na |H|^2|)dx\\
&\quad +C\int P|\na u||2G+ |H|^2|dx+C\int (|\na u|^2+|\na H|^2)|2G+|H|^2|dx \\
&\le C(\|\n u\|_{L^2}+\|\na \te\|_{L^2})( \|\na G\|_{L^2}+\|\na |H|^2\|_{L^2}) \\
&\quad+ C\|\n(\te-1)\|_{L^2}^{1/2}\|\na\te\|_{L^2}^{1/2}(\|\na G\|_{L^2}+\|\na |H|^2\|_{L^2})\|\na u\|_{L^2}\\
&\quad + C(\|\na u\|_{L^2}+\|\n\te-1\|_{L^2})\|\na u\|_{L^2}\\
& \quad+ C (\|\na G\|_{L^2}+\|\na |H|^2\|_{L^2})(\|\na u\|_{L^2}^{3/2} \|\na u\|_{L^6}^{1/2}+\|\na H\|_{L^2}^{3/2}\|\na^2 H\|_{L^2}^{1/2})\\
&\le \de (\|\na G\|_{L^2}^2 +\||H|| \na H|\|_{L^2}^2+ \|\na u\|^2_{L^6}+\|\na^2 H\|_{L^2}^2)\\
&\quad+C(\de) \left( \|\na u\|_{L^2}^6+ \|\na \te\|_{L^2}^2+ \|\na H\|^6_{L^2}+1\right) \\
&\le C\de(\|\rho^{1/2} \dot u\|^2_{L^2}+\|H_t\|_{L^2}^2)\\ &\quad+C(\de) \left( \|\na u\|_{L^2}^6+ \|\na \te\|_{L^2}^2 +\|\na  H\|^6_{L^2}+1\right).
\ea\ee
Then, one deduces from (\ref{g1}), \eqref{3.q1}, and (\ref{3.30}) that
\be\la{op1}\ba
|I_2|&\le C\|u\|_{L^6}^2 \|\na u\|_{L^2} \|\na u\|_{L^6}  \\
&\le \de(\|\rho^{1/2} \dot u\|_{L^2}^2+\|H_t\|_{L^2}^2)+ C(\de)\left(\|\na u\|_{L^2}^6+\|\na\te\|_{L^2}^2+\|\na
H\|_{L^2}^6+1\right).
\ea\ee
Combining integration by parts with \eqref{h0}$_5$ and \eqref{go1} leads to
\be\ba\la{k1}
|I_3|&=
\left|\int \left(2(H\otimes H)_t:\na u-(|H|^2)_t\div u\right )dx\right|\\
&\le C\int|H_t||H||\na u| dx\\
&\le C\|H_t\|_{L^2}\|H\|_{L^\infty}\|\nabla u\|_{L^2}\\
&\le \de\|H_t\|_{L^2}^{2}+C(\de)(\|\nabla H\|_{L^2}^6+\|\nabla u\|_{L^2}^6+1).
\ea\ee

Substituting (\ref{a16})--(\ref{k1}) into (\ref{k0}) and adding the resulting inequality to \eqref{tdh-3}, one derives after choosing $\de$ suitably small  and using \eqref{k88} that
\be\ba\la{k2}
&\left(\nu\|\curl H\|_{L^2}^2+B_0\right)_t+\frac{2}{9}\nu^2\|\curl^2 H\|_{L^2}^2+\frac{1}{2}(\|H_t\|^2_{L^2}+\|\n^{1/2}\dot u\|_{L^2}^2)\\
&\le C \left( \|\na u\|_{L^2}^6+ \|\na \te\|_{L^2}^2 +\|\na  H\|^6_{L^2}+1\right).
\ea\ee

Finally,  by \eqref{h0}$_5$, \eqref{g1}, and \eqref{a2.112}, one arrives at
\be\ba\la{k11}
\left|\int(2H\cdot\na H-\na |H|^2)\cdot u dx\right|&
= \left|\int\left(2(H\otimes H):\na u-|H|^2\div u\right )dx\right|\\
&\le C\|\na u\|_{L^2}\|H\|_{L^2}^{1/2}\|\na H\|_{L^2}^{3/2}\\
&\le CC_0^{1/16}\|\na u\|_{L^2}\|\na H\|_{L^2}^{3/2},
\ea\ee
which together with \eqref{k2}, \eqref{ljq01}, (\ref{3.q1}), \eqref{po}, \eqref{p}, and Holder's inequality yields
\bnn\la{h81} \ba
&\sup_{0\le t\le \sigma(T)}(\|\na u\|_{L^2}^2+\|\na H\|_{L^2}^2)+ \int_0^{\sigma(T)}\left(\|\n^{1/2}\dot u\|_{L^2}^2+\|H_t\|^2_{L^2}+\|\curl^2 H\|_{L^2}^2\right)dt\\
&\le CM^2+CC_0^{1/4} + \hat C_1C_0^{1/4}\sup_{0\le t\le \sigma(T)}(\|\na u\|_{L^2}^4+\|\na H\|_{L^2}^4)\\
&\le K+9K^2\hat C_1C_0^{1/4}\\
&\le 2K,
\ea \enn
with $K\triangleq CM^2+C +1$, provided
\be\notag C_0\le \ep_1 \triangleq \min\left\{1,\xl(9\hat C_1K\xr)^{-4}\right\}.\ee
The proof of Lemma \ref{le2} is completed.
\end{proof}

Next, we adopt the approach due to Hoff \cite{Hof1} (see also Huang-Li \cite{H-L})  to establish the elementary estimates on $\dot u$, $\dot \te$, and $H_t$, which are the footstone for the estimates on $A_3(T)$ and $A_2(T)$,  and the ideas in Cai-Li \cite{C-L} to deal with the emerging  boundary terms.

\begin{lemma}\la{a113.4}
	Under the conditions of Proposition \ref{pr1}, let $(\rho,u,\te,H)$ be a smooth solution to the problem (\ref{h0})--(\ref{ch2}) on $\Omega\times (0,T] $ satisfying (\ref{z1}) with $K$ as in Lemma \ref{le2}.  Then  there exist positive constants $C$, $ C_1$, and $C_2$ depending only on $\mu,\,\lambda,\,\nu, \,k,\, R,\, \ga,\, \on,$  $\O$, and $ M$  such that, for any $\beta,\eta\in (0,1]$ and $m\geq0,$
the following estimates hold:
\be\ba  \la{an1}
&(\sigma B_1)'(t) + \frac{3}{2}\si\|\n^{1/2}\dot u\|_{L^2}^2\\
&\le   C C_0^{1/4} \sigma' +2\beta\si^2\|\n^{1/2}\dot\theta\|_{L^2}^2+C\si^2\|\na u\|_{L^4}^4+C\si\|H_t\|^2_{L^2}\\ &\quad+C\beta^{-1}\left(\|\na u\|_{L^2}^2+\|\na H\|_{L^2}^2+\|\na\te\|_{L^2}^2\right),
\ea\ee
\be\la{ae0}\ba
&(\sigma^{m}B_2)'t+C_1 \sigma^{m}\left(\|\na\dot{u}\|_{L^2}^2+\|\na H_t\|_{L^2}^2\right)\\
&\le C_2  \si^m \|\rho^{1/2} \dot \te\|_{L^2}^2+ C(\si^{m-1}\si'+\si^m) (\|\rho^{1/2}\dot u\|_{L^2}^2+\|H_t\|_{L^2}^2)\\
&\quad+ C( \|\na u\|^2_{L^2}+\|\na H\|^2_{L^2}+\|\na \te\|^2_{L^2})\\
&\quad+C \si^m (\|\na u\|^4_{L^4}+\|H_t\|^3_{L^2}+ \|\te \na u\|_{L^2}^2),\ea\ee
  and
 \be\la{nle7}\ba  &(\si^mB_3 )'(t)+\si^m \|\n^{1/2}\dot \te\|_{L^2}^2\\
&\le C \eta \si^m\left(\|\na\dot u\|_{L^2}^2+\|\na H_t\|_{L^2}^2\right)+C \|\na
\te \|_{L^2}^2+C\si^m \|\na u\|_{L^4}^4\\
&\quad+C\eta^{-1} \si^m\left( \|\te\na u\|_{L^2}^2+\| H_t\|_{L^2}^3+\|\na\te\|_{L^2}^3+\|\na H\|_{L^2}^2\right),\ea\ee
where
\be\ba\la{an2}
B_1(t)\triangleq &\mu\|\o\|_{L^2}^2+ (2\mu+\lambda) \|\div u\|_{L^2}^2-2 R\int \div u(\n\te-1) dx\\
&-\int(2H\cdot\na H-\na|H|^2)\cdot udx, \ea\ee
\be\la{k6}
B_2(t)\triangleq
\|\rho^{1/2}\dot{u}\|_{L^2}^2+\|H_t\|_{L^2}^2+2\int_{\p \O}  (u \cdot \na n \cdot u) G dS,
\ee
and
 \be\la{e6}
B_3(t)\triangleq\frac{\ga-1}{R}\left(\ka \|\na
\te\|_{L^2}^2-2 \int (\lambda (\div u)^2+2\mu|\mathfrak{D}(u)|^2+\nu|\curl H|^2)\te dx\right).\ee
\end{lemma}
\begin{proof} The proof is divided into  the following three parts.

\noindent{}\textbf{Part I: The proof of (\ref{an1}).}

First, we can rewrite \eqref{a11} as
\be\la{a111}
\n\dot u+\mu\na\times\curl u-\na G-H\cdot\na H=0.
\ee
Multiplying
$(\ref{a111}) $ by $\sigma \dot{u}$ and integrating the resulting
equality by parts yield
\be\la{m0} \ba  \int \sigma
\rho|\dot{u}|^2dx & = \int (\sigma \dot{u}\cdot\nabla G - \si \mu \na \times \curl u\cdot\dot{u}+\si H\cdot\na H\cdot\dot u)dx  \\
&=\int_{\partial \O} \sigma (u\cdot\na u \cdot n) G dS - \int \sigma \div \dot{u} G dx - \mu \int  \si  \curl u \cdot \curl \dot{u} dx \\
&\quad+\int\si H\cdot\na H\cdot\dot udx \triangleq \sum_{i=1}^{4}M_i. \ea \ee

Note that \eqref{h1} implies
\be\la{pzw1} u\cdot\nabla u\cdot n=-u\cdot\nabla n\cdot u ~~\quad \mbox{on}~\p \O,\ee
and
\begin{align}\la{eee1}
		u=u^{\perp} \times n~~~~~\text{on}\, \p \O,
\end{align}
with $u^{\perp}=-u\times n$.
Then, the combination of \eqref{eee0} and \eqref{pzw1} leads to
\be \ba \la{bb2}
M_1&=-\int_{\partial \O} \sigma  (u\cdot \na n\cdot u) G dS\\
 &\le C\sigma \|u\|_{L^4(\partial\Omega)}^2\|G\|_{L^2(\partial\Omega)}\\
 &\le C\sigma \|\na u\|_{L^2} ^2\|\na G\|_{L^2}\\
 &\le \delta \sigma\|\rho^{1/2} \dot u\|_{L^2}^2+\si\|H_t\|_{L^2}^2+C(\de) \sigma\|\na u\|_{L^2}^2,
\ea \ee
where in the last inequality we have used the following simple facts:
\be \la{infty12}
\sup_{t\in[0,T]}(\|\na u\|_{L^2}^2+\|\na H\|_{L^2}^2)\leq A_1(\sigma(T))+A_3(T)\leq C,
\ee
and
\be\la{o0}
\|\na G\|_{L^2}\le C(\|\n\dot u\|_{L^2}+\|H_t\|_{L^2}+\|\na H\|_{L^2})
\ee
owing to \eqref{h19},  \eqref{goo1}, and \eqref{infty12}.

Notice that
\be \ba \label{jia1}
P_t=(R\n \te)_t=R\n \dot{\te}-\div (Pu),
\ea \ee
which as well as  direct calculations shows
\be \ba\notag
\div \dot u G=& (\div u_t + \div (u \cdot \na u))((2\mu+\lambda)\div u - R(\rho\theta-1)-\frac{1}{2}|H|^2)\\
=& \frac{2\mu+\lambda}{2} (\div u)^2_t - (R(\rho\theta-1)\div u)_t +R\rho \dot\te \div u - \div(Pu) \div u\\
& + (2\mu+\lambda) \div (u \cdot \na u)\div u -R(\rho\theta-1)\div (u \cdot \na u)\\
&-\frac{1}{2}(|H|^2\div u)_t+H\cdot H_t\div u-\frac{1}{2}|H|^2 \div (u \cdot \na u)\\
=& \frac{2\mu+\lambda}{2} (\div u)^2_t - (R(\rho\theta-1)\div u)_t +R\rho \dot\te \div u \\
& + (2\mu+\lambda) \na u : (\na u)^{\rm tr} \div u + \frac{2\mu+\lambda}{2} u \cdot \na(\div u)^2 \\
&- \div(R(\rho\theta-1)u \div u)- R(\rho\theta-1) \na u : (\na u)^{\rm tr}-R(\div u)^2\\
&-\frac{1}{2}(|H|^2\div u)_t+H\cdot H_t\div u-\frac{1}{2}|H|^2\na u : (\na u)^{\rm tr}-\frac{1}{2}|H|^2 u \cdot \na \div u.
\ea \ee
Combining this with  integration by parts, \eqref{g1}, and \eqref{z1} gives that for any $\beta\in
(0,1],$
\be\la{m1} \ba
M_2
=&  -\frac{1}{2} \left(\int \sigma  \left((2\mu+\lambda)(\div u)^2-2R(\rho\theta-1) \div u-|H|^2\div u \right)dx \right)_t\\
&+ \frac{1}{2} \si' \int \left((2\mu+\lambda)(\div u)^2-2R(\rho\theta-1) \div u-|H|^2\div u \right)dx \\
& - R\si \int \rho \dot\te \div u dx-(2\mu+\lambda)  \si \int  \na u : (\na u)^{\rm tr} \div u dx \\
&+ \frac{2\mu  +\lambda}{2} \si \int  (\div u)^3 dx+ R\si \int (\rho\theta-1) \na u : (\na u)^{\rm tr}dx \\
&+  R\si \int (\div u)^2 dx-\si\int H\cdot H_t\div u dx+\frac{1}{2}\si\int|H|^2\na u : (\na u)^{\rm tr}dx\\
&+\frac{1}{2}\si\int|H|^2 u \cdot \na \div u dx\\
\le & -\frac{1}{2} \left(\int \sigma  \left((2\mu+\lambda)(\div u)^2-2R(\rho\theta-1) \div u-|H|^2\div u \right)dx \right)_t\\
& + C \si' \|\rho\theta-1\|_{L^2}^2+ \beta \si^2 \|\rho^{1/2} \dot\te\|_{L^2}^2 +  C \beta^{-1}  \|\na u\|_{L^2}^2+\si\|H_t\|_{L^2}^2+C\|H\|_{L^4}^4 \\
&+ C\si^2 \|\na u\|_{L^4}^4 + C \si \int
\te |\na u|^2dx+C\si\|u\|_{L^6}\|\na u\|_{L^2}\|H\|_{L^6}\|\na H\|_{L^6}\\
\le & -\frac{1}{2} \left(\int \sigma  \left((2\mu+\lambda)(\div u)^2-2R(\rho\theta-1) \div u-|H|^2\div u \right)dx \right)_t\\
&+ C C_0^{1/4}\si'+ \beta \si^2 \|\rho^{1/2} \dot\te\|_{L^2}^2+ C\de \si\|\n^{1/2}\dot u\|_{L^2}^2+C\si\|H_t\|^2_{L^2}  \\
&+ C\si^2 \|\na u\|_{L^4}^4+  C(\de) \beta^{-1} ( \|\na u\|_{L^2}^2+\|\na H\|_{L^2}^2 + \|\na \te \|_{L^2}^2),
\ea \ee
where in the last inequality we have used  \eqref{infty12},  (\ref{p}), and
the following  facts:
\be \la{2.48}\ba   \int
\te |\na u|^2dx   & \le
\int|\te-1||\na u|^2dx+  \int |\na u|^2dx\\ &\le C
\|\te-1 \|_{L^6}\|\na u\|_{L^2}^{3/2}
\|\na u\|_{L^6}^{1/2}+  \|\na u\|_{L^2}^2
\\ &\le \|\na  u\|_{L^2}^2+C \|\na\te\|_{L^2}\|\na u\|_{L^2}^{3/2}\\
&\quad\cdot\left(   \|\n \dot u\|_{L^2}+\|H_t\|_{L^2}+\|\na H\|_{L^2}+\|\na u\|_{L^2}+ \|\na\te\|_{L^2}+1\right)^{1/2}
\\ &\le \de
\left(  \|\na\te\|^2_{L^2}  + \|\n^{1/2}  \dot
u\|_{L^2}^2+\|H_t\|_{L^2}^2 \right) + C(\de)  \|\na
u\|_{L^2}^2 \ea\ee
due to  \eqref{g1}, (\ref{3.30}), (\ref{z1}),  (\ref{infty12}), and
\be\la{o1}
\|\na^2 H\|_{L^2}\le C(\|H_t\|_{L^2}+\|\na H\|_{L^2})
\ee
due to \eqref{2tdh}, \eqref{k88}, and \eqref{infty12}.

Analogously, we can deal with the  term $M_4$ as follows
\be\ba\la{k3}
M_4
=&\left(\int\si H\cdot\na H\cdot u dx\right)_t-\si'\int \div(H\otimes��H)��\cdot u dx\\
&-\si\int \div (H\otimes H)_t u dx+\int\si H\cdot\na H\cdot(u\cdot\na u)dx\\
\le&\left(\int\si H\cdot\na H\cdot u dx\right)_t+C\si'\|\na u\|_{L^2}\|H\|_{L^2}^{1/2}\|\na H\|_{L^2}^{3/2}\\
&+C\si\|H_t\|_{L^2}\|H\|_{L^2}^{1/4}\|\na H\|_{L^2}^{3/4}\|\na u\|_{L^4}+C\si\|u\|_{L^6}\|\na u\|_{L^2}\|H\|_{L^6}\|\na H\|_{L^6}\\
\le&\left(\int\si H\cdot\na H\cdot u dx\right)_t+C\si\|H_t\|^2_{L^2}+\si^2\|\na u\|_{L^4}^4\\
&+C( \|\na u\|_{L^2}^2+\|\na H\|_{L^2}^2).
\ea\ee

Finally, some direct calculations infer that
 \be\la{m2} \ba
M_3
& = -\frac{\mu }{2}\int\sigma |\curl u|^2_t dx
    -\mu \sigma  \int \curl u \cdot \curl (u\cdot \na u)dx \\
& = -\frac{\mu }{2}\left(\sigma \|\curl u\|_{L^2}^2\right)_t +
\frac{\mu }{2}\si'  \|\curl u\|_{L^2}^2
-\mu \sigma  \int \curl u \cdot (\na u^i \times \na_i  u)  dx \\&\quad+ \frac{\mu}{2}\si
\int |\curl u|^2\div udx \\
& \le  -\frac{\mu }{2}\left(\sigma \|\curl u\|_{L^2}^2\right)_t + C
\|\na u\|_{L^2}^2 +   C\sigma^2 \|\na u\|_{L^4}^4  . \ea \ee

Putting \eqref{bb2}, (\ref{m1}), \eqref{k3},  and  (\ref{m2}) into (\ref{m0}) and choosing $\de$ suitably small, we arrive at  \eqref{an1} directly.

\noindent{}\textbf{Part II:  The proof of  (\ref{ae0}).}

For $m\ge 0,$ operating $ \si^m\dot u^j[\pa/\pa t+\div (u\cdot)]$ to $ (\ref{a111})^j$ and integrating the resulting equality with respect to $x$ over $\Omega$ by parts, one has
\be\la{m4} \ba
& \left(\frac{\sigma^m}{2}\int\rho|\dot{u}|^2dx \right)_t -\frac{m}{2}\sigma^{m-1}\si'\int\n|\dot{u}|^2dx\\
 &=  \int_{\p \O}  \sigma^m \dot{u} \cdot n G_t dS - \int  \sigma^m [\div \dot{u} G_t + u \cdot \na \dot u \cdot \na G]dx \\
&\quad- \mu \int\sigma^m\dot{u}^j\xl[(\na \times \curl u)_t^j + \div (u (\na \times \curl u)^j)\xr] dx \\
&\quad+\int\si^m \xl[\dot{u}\cdot \div (H\otimes H)_t+\dot u^j\div ( H\cdot\na H^ju)\xr] \triangleq\sum_{i=1}^{4}N_i.
\ea \ee

First, we need to deal with the boundary integral $N_1$. By \eqref{h1} and \eqref{pzw1}, it holds
\be \la{bz8}\ba
N_1&=- \int_{\p \O}  \sigma^m (u \cdot \na n \cdot u) G_t dS\\
&=- \left(\int_{\p \O} \sigma^m (u \cdot \na n \cdot u) G dS\right)_t + m \si^{m-1} \si' \int_{\p \O}( u \cdot \na n \cdot u) G dS \\
&\quad+  \int_{\p \O}  \sigma^m (\dot u \cdot \na n \cdot u) G dS+ \int_{\p \O}  \sigma^m ( u \cdot \na      n \cdot \dot u) G dS\\
&\quad-  \int_{\p \O}  \sigma^m G( u \cdot \na) u \cdot \na n \cdot u  dS -\int_{\p \O}  \sigma^m G u \cdot \na n \cdot (u\cdot \na )u  dS\\
&\le - \left(\int_{\p \O}  \sigma^m (u \cdot \na n \cdot u) G dS\right)_t + C\si^{m-1} \si' \| \na u\|_{L^2}^2\|\na G\|_{L^2} \\
&\quad+\de \si^m \|\na\dot u\|_{L^2}^2+ C(\de) \si^m \|\na u\|_{L^2}^2 \|\na G\|_{L^2}^2\\
& \quad- \int_{\p \O}  \sigma^m G( u \cdot \na) u \cdot \na n \cdot u  dS -\int_{\p \O}  \sigma^m G u \cdot \na n \cdot (u\cdot \na )u  dS,
\ea \ee
where we have used the following estimates:
\be \ba\notag
\int_{\partial \O}  (\dot u\cdot \na n\cdot u+ u\cdot \na n\cdot\dot u) G dS &\le C \|\dot u\|_{L^4(\partial\Omega)} \| u\|_{L^2(\partial\Omega)} \|G\|_{L^4(\partial\Omega)} \\&\le C \|\na\dot u\|_{L^2} \|\na u\|_{L^2} \|\na G\|_{L^2},
\ea \ee
and
\be\la{b2}
\int_{\partial \O}  ( u\cdot \na n\cdot u) G dS   \le C \|\na u\|_{L^2} ^2\|\na G\|_{L^2}
\ee
owing to \eqref{eee0}.

For the last two boundary terms in \eqref{bz8}, we adopt the ideas in \cite{C-L} to handle. In fact,  combining (\ref{eee1})  with \eqref{g1} and integration by parts infers
\be \la{bz3}\ba &- \int_{\partial\Omega} G (u\cdot \na) u\cdot\na n\cdot u dS \\&= -\int_{\partial\Omega}  G u^\bot\times n \cdot\na u^i \nabla_i n\cdot u  dS \\&= - \int_{\partial\Omega} G n\cdot ( \na u^i \times  u^\bot)    \nabla_i n\cdot u dS\\
&= - \int\div( G( \na u^i \times  u^\bot)   \nabla_i n\cdot u) dx \\
&= - \int \na (\nabla_i n\cdot u G) \cdot ( \na u^i \times  u^\bot)   dx  - \int \div( \na u^i \times  u^\bot)    \nabla_i n\cdot u   G  dx \\
&= - \int \na (\nabla_i n\cdot u G) \cdot ( \na u^i \times  u^\bot)   dx  + \int  G \na  u^i \cdot \na\times  u^\bot     \nabla_i n\cdot u     dx \\
& \le C \int |\na G||\na u||u|^2dx+C \int |G| (|\na u|^2|u|+|\na u||u|^2)dx
\\& \le C  \|\na G\|_{L^6}\|\na u\|_{L^2}\|u\|^2_{L^6}
+C  \| G\|_{L^6}\|\na u\|^2_{L^3}\|u\|_{L^6}+C\| G\|_{L^{6}} \|\na      u\|_{L^2}\|u\|_{L^6}^2
\\& \le \de \|\na G\|_{L^6}^2+C(\de) \|\na u\|^6_{L^2}+C\|\na u\|^4_{L^3}+ C  \|  \na G\|_{L^2}^2 (\|\na u\|^2_{L^2} +1 ).\ea\ee
Similarly,
\be \la{bz4}\ba &-  \int_{\partial\Omega}G  u\cdot\na n\cdot ({u}\cdot\na) u dS\\& \le \de \|\na G\|_{L^6}^2+C(\de) \|\na u\|^6_{L^2}+C\|\na u\|^4_{L^3}+ C \|  \na G\|_{L^2}^2(\|\na u\|^2_{L^2} +1 ).\ea\ee

Next, by virtue of \eqref{hj1} and \eqref{jia1}, we obtain
\be \ba \notag
G_t
=& (2\mu+\lambda) \div \dot u - (2\mu+\lambda) \div (u\cdot \na u) - R\rho \dot\te + \div(Pu)-H\cdot H_t\\
=& (2\mu+\lambda) \div \dot u - (2\mu+\lambda) \na u : (\na u)^{\rm tr} -  u \cdot  \na G + P \div u - R\rho \dot\te\\
&-u\cdot\na H\cdot H-H\cdot H_t,
\ea \ee
which together with  \eqref{z1} leads to
\be\la{m5} \ba
N_2 
=& - (2\mu+\lambda) \int  \sigma^m (\div \dot{u} )^2 dx +  (2\mu+\lambda)  \int \sigma^m \div \dot{u} \na u : (\na u)^{\rm tr} dx \\
&+\int  \sigma^m \div \dot{u}  u \cdot  \na G  dx - \int \sigma^m \div \dot{u} P \div u dx  \\
&+ R \int \sigma^m \div \dot{u} \rho \dot\te dx - \int  \sigma^m u \cdot \na \dot u \cdot \na Gdx \\
&+\int \sigma^m \div \dot{u}u\cdot\na H\cdot H dx+\int \sigma^m \div \dot{u}H\cdot H_t dx\\
\le &   - (2\mu+\lambda) \int  \sigma^m (\div \dot{u} )^2 dx\\
& + C \si^m \|\na \dot{u}\|_{L^2} \|\na u\|_{L^4}^2 + C \si^m \|\na \dot{u}\|_{L^2} \|\na G\|_{L^2}^{1/2} \|\na G\|_{L^6}^{1/2} \|u\|_{L^6} \\
&+ C \si^m \|\na \dot{u}\|_{L^2} \|\te \na u\|_{L^2}+ C \si^m \|\na \dot{u}\|_{L^2} \|\rho^{1/2} \dot \te\|_{L^2}\\
&+C\si^m\|\na\dot u\|_{L^2}\|u\|_{L^6}\|H\|_{L^6}\|\na H\|_{L^6}+C\si^m\|\na\dot u\|_{L^2}\|H_t\|_{L^2}\|H\|_{L^\infty}.\ea \ee

Analogously, it holds that
\be\ba\la{k4}
N_4
=&-\int\si^m  (H\otimes H)_t:\na\dot u dx-\int\si^m u\cdot\na\dot u^j H\cdot\na H^j dx \\
\le & C\si^m\|\na\dot u\|_{L^2}\|u\|_{L^6}\|H\|_{L^6}\|\na H\|_{L^6}+C\si^m\|\na\dot u\|_{L^2}\|H_t\|_{L^2}\|H\|_{L^\infty}.
\ea\ee

Observe that
$$\curl u_t=\curl \dot u-u\cdot \na \curl u-\na u^i\times \nabla_iu,$$
which as well as some direct calculations shows that
\be\la{ax3999}\ba
N_3 &=- \mu \int\sigma^{m}|\curl\dot{u}|^{2}dx+\mu \int\sigma^{m}\curl\dot{u}\cdot(\nabla u^i\times\nabla_i u) dx \\&\quad+\mu \int\sigma^{m} u\cdot\na  \curl u \cdot\curl\dot{u} dx  +\mu \int\sigma^{m}    u \cdot \na \dot{u}\cdot (\nabla\times\curl u) dx  \\
&\le - \mu \int\sigma^{m}|\curl\dot{u}|^{2}dx+\de  \sigma^{m}(\|\na \dot u\|_{L^2}^2 +
\|\na \curl u\|_{L^6}^2)\\&\quad +C(\de) \sigma^{m} \|\na u\|_{L^4}^4 +C(\de)  \sigma^{m}\|\na  u\|_{L^2}^4\|\na \curl u\|_{L^2}^2.\\
\ea\ee

Finally,   combining Lemma \ref{le4} with \eqref{g1}, \eqref{g2}, \eqref{goo1}, \eqref{infty12},  and \eqref{o1}  yields
\be \la{bz6}\ba
\|\na G\|_{L^6}
&\le C\|\n \dot u\|_{L^6}+C\|H\cdot\na H\|_{L^6}\\
&\le C\|\na \dot u\|_{L^2}+C\|\na H\|_{L^2}^{1/2}\|\na^2 H\|_{L^2}^{3/2}\\
&\le C\|\na \dot u\|_{L^2}+C\|H_t\|_{L^2}^{3/2}+C\|\na H\|_{L^2}, \ea\ee
\be\la{jh}
\|\na\o\|_{L^6}\le C(\|\na\dot u\|_{L^2}+\|\n\dot u\|_{L^2}+\|\na u\|_{L^2}+\|\na H\|_{L^2})+C\|H_t\|_{L^2}^{3/2}.
\ee

Thus, substituting \eqref{bz8} and \eqref{m5}--\eqref{ax3999} into \eqref{m4}, it follows from \eqref{bz3}, \eqref{bz4}, \eqref{z1},   \eqref{a2.112}, \eqref{bz6}, \eqref{jh},  \eqref{go1},  \eqref{infty12}, \eqref{o0}, and \eqref{o1}
that
\be\la{ax40}\ba
&\left(\frac{\sigma^{m}}{2}\|\rho^{1/2}\dot{u}\|_{L^2}^2\right)_t+(2\mu+\lambda)\sigma^{m}\|\div\dot{u}\|_{L^2}^2+\mu\sigma^{m}\|\curl\dot{u}\|_{L^2}^2\\
&\le -\left(\int_{\p \O}  \sigma^m (u \cdot \na n \cdot u) G dS\right)_t+ C\de \si^m \|\na \dot{u}\|_{L^2}^2+ C(\de) \si^m \|\rho^{1/2} \dot \te\|_{L^2}^2 \\
&\quad + C(\de)(\si^{m-1}\si'+\si^m) (\|\rho^{1/2}\dot u\|_{L^2}^2+\|H_t\|_{L^2}^2)\\
&\quad +C(\de)( \|\na u\|^2_{L^2}+\|\na H\|^2_{L^2})+C(\de)\si^m (\|\na u\|^4_{L^4}+\|H_t\|^3_{L^2}+\|\te \na u\|_{L^2}^2).
\ea\ee

Next, it remains to estimate  $\|\nabla H_t\|_{L^2}$. Notice that \eqref{mn} implies
\begin{equation}\label{sm}
\begin{cases}
H_{tt}+\nu \nabla \times (\curl H_t)
=(H \cdot \nabla u-u \cdot \nabla H-H \div u)_t,&\text{in}\quad\Omega,\\
H_t \cdot n=0,\quad \curl H_t \times n=0,& \text{on}\quad \partial\Omega,\\
H_t\rightarrow 0,&\text{as}\quad |x|\rightarrow 0.
 \end{cases}
\end{equation}
Multiplying \eqref{sm} by $\si^m H_t$ with $m\geq 0$ and integrating by parts, one has
\be\ba\label{ht1}
&\left(\frac{\sigma^{m}}{2}\|H_t\|_{L^2}^2\right)_t+\nu\sigma^{m}\|\curl H_t\|_{L^2}^2-\frac{m}{2}\sigma^{m-1}\sigma' \|H_t\|\ltwo\\
&= \int \sigma^{m}(H_t \cdot \nabla u-u \cdot \nabla H_t-H_t \div u)\cdot H_t dx\\
&\quad+ \int \sigma^{m}(H \cdot \nabla \dot{u}-\dot{u} \cdot \nabla H-H \div \dot{u})\cdot H_t dx  \\
&\quad - \int \sigma^{m}(H \cdot \nabla (u \cdot \nabla u)-u \cdot \nabla u\cdot \nabla H-H \div(u \cdot \nabla u) )\cdot H_t dx  \triangleq \sum _{i=1}^3K_i.
\ea\ee
Combining \eqref{g1} with  direct calculations and \eqref{pzw1} leads to
\be\ba\la{k5}
K_1
& \leq C\sigma^{m}\|H_t\|_{L^2}^{1/2}\|\nabla H_t\|_{L^2}^{1/2}(\|H_t\|_{L^6}\|\nabla u\|_{L^2}
+\|\nabla H_t\|_{L^2}\|u\|_{L^6})\\
&\leq \delta\sigma^{m}\|\nabla H_t\|_{L^2}^{2}+C(\de)\sigma^{m}\|\nabla u\|_{L^2}^4\|H_t\|_{L^2}^{2},
\ea\ee
\be\ba\label{htk20}
K_2
& \leq C\sigma^{m}\|H_t\|_{L^2}^{1/2}\|\nabla H_t\|_{L^2}^{1/2}(\|H\|_{L^6}\|\nabla \dot{u}\|_{L^2}+
\|\nabla H\|_{L^2}\|\dot{u}\|_{L^6})\\
&\leq \delta\sigma^{m}(\|\nabla \dot{u}\|_{L^2}^{2}+\|\nabla H_t\|_{L^2}^{2})+C(\de)\sigma^{m}\|\nabla H\|_{L^2}^{4}\|H_t\|_{L^2}^2,
\ea\ee
and
\be\ba\label{htk3}
K_3
& =\int \sigma^{m} H \cdot\nabla H_t \cdot( u \cdot \nabla u) dx
+\int_{\partial \Omega}\sigma^{m} H\cdot H_t\,(u \cdot \nabla u \cdot n) dS\\
& \quad -\int \sigma^{m} u \cdot \nabla u \cdot \nabla H_t \cdot H dx\\
&\leq -\int_{\partial \Omega}\sigma^{m} H\cdot H_t\,(u \cdot \nabla n \cdot u) dS+C\sigma^{m}\|H\|_{L^6}\|\nabla H_t\|_{L^2}\|\nabla u\|_{L^6}\|u\|_{L^6}\\
&\leq \delta\sigma^{m}\|\nabla H_t\|_{L^2}^2+C(\de)\sigma^{m}(\|\nabla H\|_{L^2}^2+\|\nabla H\|_{L^2}^8)(\|\nabla u\|_{L^2}^2+\|\nabla u\|_{L^2}^4)\\
&\quad +C(\de)\si^m\|\nabla H\|_{L^2}^2\|\nabla u\|_{L^2}^2(\|\n^{1/2}\dot{u}\|_{L^2}^2+\|H_t\|_{L^2}^2+\|\na\te\|_{L^2}^2),
\ea\ee
where in the last inequality we have used \eqref{3.30} and the following fact:
\be\ba\notag
 \int_{\partial \Omega}H \cdot H_t(u \cdot \nabla n \cdot u) dS
&\le C\|H\|_{L^4(\p\O)}\|H_t\|_{L^4(\p\O)}\|u\|_{L^4(\p\O)}^2\\
&\leq C\|\na H\|_{L^2}\|\na H_t\|_{L^2}\|\na u\|_{L^2}^2
\ea\ee
owing to \eqref{eee0}.
Putting \eqref{k5}--\eqref{htk3} into \eqref{ht1} and using \eqref{infty12},  we derive
\be\ba\label{ht3}
& \left(\frac{\sigma^{m}}{2}\|H_t\|_{L^2}^2\right)_t+\nu\sigma^{m}\|\curl H_t\|_{L^2}^2\\
&\leq 3\delta\sigma^{m}(\|\nabla \dot{u}\|_{L^2}^{2}+\|\nabla H_t\|_{L^2}^{2})+C(\de)(\sigma^{m-1}\sigma'+\si^m) \|H_t\|_{L^2}^{2}\\
&\quad+C(\de)\sigma^{m}\|\n^{1/2}\dot{u}\|_{L^2}^{2}+C(\de)(\|\nabla u\|_{L^2}^2+\|\nabla \te\|_{L^2}^{2}).
\ea\ee

Adding \eqref{ax40} to \eqref{ht3} and using  \eqref{tb11}, \eqref{ljq01}, and \eqref{infty12},  one concludes \eqref{ae0} after choosing $\de$ small enough.

\noindent{}\textbf{Part III: The proof of (\ref{nle7}).}

For $m\ge 0,$
multiplying $(\ref{h0})_3 $ by $\sigma^m \dot\te$ and integrating
the resulting equality over $\Omega $,  it holds  that
 \be\la{e1} \ba &\frac{\ka
{\sigma^m}}{2}\left( \|\na\te\|_{L^2}^2\right)_t+\frac{R\sigma^m}{\ga-1} \int\rho|\dot{\te}|^2dx
\\&=-\ka\sigma^m\int\na\te\cdot\na(u\cdot\na\te)dx
+\lambda\sigma^m\int  (\div u)^2\dot\te dx+2\mu\sigma^m\int |\mathfrak{D}(u)|^2\dot\te dx\\&\quad
+\nu\si^m\int|\curl H|^2\dot\te dx-R\si^m\int\n\te \div
u\dot\te dx \triangleq \sum_{i=1}^5J_i . \ea\ee

 First,  by (\ref{g1}),   one gets
\be\la{e2} \ba
J_1
&\le C\sigma^{m}   \|\na u\|_{L^2}\|\na\te\|^{1/2}_{L^2}
\|\na^2\te\|^{3/2}_{L^2}  \\
&\le  \de\sigma^{m} \|\n^{1/2}\dot\te\|^2_{L^2}+C\si^m \left(\|\na u\|_{L^4}^4+\|\na H\|_{L^2}\|\na^2 H\|_{L^2}^3+\|\te\na u\|_{L^2}^2\right)\\
&\quad +C(\de)\sigma^{m}
\|\na u\|_{L^2}^4\|\na\te\|^2_{L^2}   ,\ea\ee
where in the last inequality we have used the following estimate:
\be  \la{lop4}\ba
	\|\na^2\te\|_{L^2} &\le C\left(\|\n\dot \te\|_{L^2}+ \|\na u\|_{L^4}^2+ \|\na H\|_{L^2}^{1/2}\|\na^2 H\|_{L^2}^{3/2}+\|\te\na u\|_{L^2}\right),
	\ea\ee
which is derived from the combination of Lemma \ref{zhle} and the following elliptic problem:
\be\la{3.29}
\begin{cases}
\ka\Delta \te-\frac{R}{\ga-1}\n\dot\te\\
= R\n\te\div
		u-\lambda (\div u)^2-2\mu |\mathfrak{D}(u)|^2-\nu|\curl H|^2,& \text{in}\,\O\times [0,T],\\
		\na \theta  \cdot n =0,& \text{on}\,\p\O \times [0,T],\\
		\nabla\te\rightarrow0,\,\,&\text{as}\,\,|x|\rightarrow\infty.
	\end{cases}\ee

Next, by \eqref{g1},  we obtain that for  any $\eta,\de\in (0,1],$
\be\la{e3}\ba J_2 =&\lambda\si^m\int (\div u)^2 \te_t
dx+\lambda\si^m\int (\div u)^2u\cdot\na\te
dx\\=&\lambda\si^m\left(\int (\div u)^2 \te
dx\right)_t-2\lambda\si^m \int \te \div u \div (\dot u-u\cdot\na u)
dx\\&+\lambda\si^m\int (\div u)^2u\cdot\na\te
dx \\=&\lambda\si^m
\left(\int (\div u)^2 \te dx\right)_t-2\lambda\si^m\int \te \div u
\div \dot udx\\&+2\lambda\si^m\int \te \div u \pa_i u^j\pa_j  u^i dx
+ \lambda\si^m\int u \cdot\na\left(\te   (\div u)^2 \right)dx
 \\
\le &\lambda\left(\si^m\int (\div u)^2 \te dx\right)_t-\lambda
m\si^{m-1}\si'\int (\div u)^2 \te dx\\& +\eta\si^m\|\na \dot
u\|_{L^2}^2+C\eta^{-1}\si^m\|\te\na u\|_{L^2}^2+C\si^m\|\na u\|_{L^4}^4,\ea\ee
 \be \la{e55}\ba J_3&\le 2\mu\left(\si^m\int
|\mathfrak{D}(u)|^2 \te dx\right)_t-2\mu m\si^{m-1}\si'\int
|\mathfrak{D}(u)|^2 \te dx
 \\&\quad+ \eta\si^m\|\na \dot
u\|_{L^2}^2+C\eta^{-1}\si^m\|\te\na u\|_{L^2}^2+C\si^m\|\na u\|_{L^4}^4 ,    \ea\ee
\be \la{e5}\ba J_4&\le \nu\left(\si^m\int
|\curl H|^2 \te dx\right)_t-\nu m\si^{m-1}\si'\int
|\curl H|^2 \te dx
 \\&\quad-2\nu\si^m\int\curl H\cdot\curl H_t\te dx+\nu\si^m\int|\curl H|^2u\cdot\na\te dx\\
 &\le \nu\left(\si^m\int
|\curl H|^2 \te dx\right)_t+C\si^m\|u\|_{L^6}\|\na\te\|_{L^2}\|\na H\|_{L^6}^2\\
&\quad+C\si^m\|\na H_t\|_{L^2}(\|\te-1\|_{L^6}\|\na H\|_{L^2}^{1/2}\|\na^2 H\|_{L^2}^{1/2}+\|\na H\|_{L^2})\\
&\le \nu\left(\si^m\int
|\curl H|^2 \te dx\right)_t+\eta\si^m \|\na H_t\|_{L^2}^2+C\si^m\|\na u\|_{L^2}\|\na\te\|_{L^2}\|\na^2 H\|_{L^2}^2\\
&\quad+C\eta^{-1}\si^m(\|\na\te\|_{L^2}^2\|\na H\|_{L^2}\|\na^2 H\|_{L^2}+\|\na H\|_{L^2}^2),
\ea\ee
and
 \be\la{e39}\ba
 |J_5|    \le  \delta \si^m \int \n |\dot\te|^2dx+C( \delta)\si^m \|\te\na u\|_{L^2}^2.  \ea\ee

Hence, substituting  (\ref{e2}) and \eqref{e3}--(\ref{e39}) into (\ref{e1})  and choosing $\de$ suitably
small, we deduce from
  \eqref{h3}, \eqref{infty12}, and  \eqref{o1} that \eqref{nle7} holds.

The proof of Lemma \ref{a113.4} is completed.
\end{proof}

Next, with  the estimates \eqref{an1}--\eqref{nle7} at hand, we are ready to derive the a priori estimate on $A_3(T)$.

\begin{lemma}\la{le6} Under the conditions of Proposition \ref{pr1},   there exists a positive constant $\ve_2$
depending only on   $\mu,\,\lambda,\,\nu,\, \ka,\, R,\, \ga,\, \on, $ $\O$, and $M$
such that if $(\rho,u,\te,H)$ is a smooth solution to the problem (\ref{h0})--(\ref{ch2})  on $\Omega\times (0,T] $ satisfying      (\ref{z1}) with $K$ as
in Lemma \ref{le2}, the following estimate holds: \be\la{b2.34}  A_3(T)+A_4(T) \le C_0^{1/6},\ee provided $C_0\le
\ve_2$.
\end{lemma}

\begin{proof}
First, multiplying \eqref{tdh-3} by $\si$ and integrating the resulting inequality with respect to $t$, one gets after using \eqref{tdh1}, \eqref{z1}, \eqref{infty12}, and \eqref{k88} that
\be\la{k7}
\sup_{0\le t\le T}\si\|\na H\|_{L^2}^2+\int_0^T\si(\|\curl^2H\|_{L^2}^2+\|H_t\|_{L^2}^2)dt\le CC_0^{1/4}.
\ee

Note that by virtue of (\ref{ok}), (\ref{h20}), (\ref{h21}), \eqref{p}, \eqref{g1}, \eqref{a2.112}, \eqref{goo1},  and (\ref{infty12}), it holds
\be\la{ae9}\ba
&\|\na u\|_{L^4}^4\\
&\le C \|G\|_{L^4}^4+C \|\curl u\|_{L^4}^4 +C \|\n\te-1\|_{L^4}^4
+C\||H|^2\|_{L^4}^4+C\|\na u\|_{L^2}^4\\
&\le C\left(\|\na u\|_{L^2}+\|H\|_{L^4}^2+1\right)(\|\n\dot u\|_{L^2}+\||H||\na H|\|_{L^2})^3 +C\|\na\te\|_{L^2}^3\\
&\quad+C\|\n-1\|_{L^4}^4 +C\|\na u\|_{L^2}^4\\	
&\le C\left( \|\n \dot u\|_{L^2}^3 + \|\na \te\|_{L^2}^3\right) +C\|\n-1\|_{L^4}^4+C(\|H_t\|_{L^2}^2+\|\na u\|_{L^2}^2+\|\na H\|_{L^2}^2),
\ea\ee
which along with	(\ref{z1})  leads to
\be \la{m22}\ba
\si\|\na u\|_{L^4}^4
&\le CC_0^{1/12}\|\n^{1/2} \dot u\|_{L^2}^2 +C\si\|\n-1\|_{L^4}^4\\ &\quad+C \left(\|\na \te\|_{L^2}^2+\|\na u\|_{L^2}^2+\|H_t\|_{L^2}^2+\|\na H\|_{L^2}^2\right)
.
\ea\ee
Combining this with \eqref{an1}  shows that
\be\ba  \la{k14}
&(\sigma B_1)'(t) + \sigma\|\rho^{1/2}\dot u\|_{L^2}^2\\
&\le   C C_0^{1/4} \sigma' +2\beta\si^2\|\n^{1/2}\dot\theta\|_{L^2}^2+C\si^2\|\n-1\|_{L^4}^4\\ &+C\si\|H_t\|_{L^2}^2+C\beta^{-1}\left(\|\na u\|_{L^2}^2+\|\na H\|_{L^2}^2+\|\na\te\|_{L^2}^2\right),
\ea\ee
provided that $C_0\le \epsilon_{2,1}\triangleq \min\{1,(2C)^{-12}\}$.

Next, for $C_2$ as in  (\ref{ae0}), adding \eqref{nle7} multiplied by $C_2+1$ to \eqref{ae0} and choosing $\eta$ suitably small, one derives
\be\la{k8}\ba
&(\sigma^{m}\varphi)'(t)+\frac{C_1}{2} \sigma^{m}\left(\|\na\dot{u}\|_{L^2}^2+\|\na H_t\|_{L^2}^2\right)+\si^m\|\n^{1/2}\dot\te\|_{L^2}^2\\
&\le  C(\si^{m-1}\si'+\si^m) (\|\rho^{1/2}\dot u\|_{L^2}^2+\|H_t\|_{L^2}^2)\\
&\quad+ C( \|\na u\|^2_{L^2}+\|\na H\|^2_{L^2}+\|\na \te\|^2_{L^2})\\
&\quad+C \si^m (\|\na u\|^4_{L^4}+\|H_t\|^3_{L^2}+\|\na\te\|_{L^2}^3+\|\te \na u\|_{L^2}^2),\ea\ee
where  $\varphi(t)$ is defined by
	\be\la{wq3} \varphi(t) \triangleq   B_2(t) +(C_2+1)  B_3(t).\ee
Taking $m=2$ in \eqref{k8}  and using \eqref{m22}  and \eqref{z1} lead to
\be\la{k10}\ba
&(\sigma^2\varphi)'(t)+\frac{C_1}{2} \sigma^2\left(\|\na\dot{u}\|_{L^2}^2+\|\na H_t\|_{L^2}^2\right)+\si^2\|\n^{1/2}\dot\te\|_{L^2}^2\\
&\le  C_3\si\|\rho^{1/2}\dot u\|_{L^2}^2+C\si\|H_t\|_{L^2}^2+C\si^2\|\n-1\|_{L^4}^4\\
&\quad+ C( \|\na u\|^2_{L^2}+\|\na H\|^2_{L^2}+\|\na \te\|^2_{L^2}),\ea\ee
where we have used the following estimate:
\be \la{m20}\ba
\|\te\na u\|_{L^2}^2
&\le \|\te-1 \|_{L^6}^2 \|\na u\|_{L^2} \|\na u\|_{L^6}+\|\na u\|_{L^2}^2 \\
&\le C  \left( \|\na u\|_{L^2}^2+\|\na\te\|_{L^2}^2\right)\left( \|\n\dot u\|_{L^2}+ \|H_t\|_{L^2}+\|\na\te\|_{L^2} +1\right)
\ea\ee
due to \eqref{g1}, (\ref{3.30}), and \eqref{infty12}.

Notice that by \eqref{g1},  \eqref{o1}, and \eqref{infty12}, we have
\be\la{nnn}\ba
\int|\curl H|^2\te dx
&\le C\|\na H\|_{L^2}^2+C\|\te-1\|_{L^6}\|\na H\|_{L^2}^{3/2}\|\na^2H\|_{L^2}^{1/2}\\
&\le \de(\|H_t\|_{L^2}^2+\|\na\te\|_{L^2}^2)+C(\de)\|\na H\|_{L^2}^2,
\ea\ee
which together with  \eqref{k6}, (\ref{e6}), \eqref{b2}, \eqref{o0}, \eqref{2.48}, and \eqref{infty12} indicates
	\be\la{wq2}  \varphi(t) \ge  \frac{1}{2}(\|\n^{1/2}\dot u\|_{L^2}^2 +\|H_t\|_{L^2}^2)+\frac{\ka(\ga-1)}{R}\|\na\te\|_{L^2}^2-C_4(\|\na u\|_{L^2}^2+\|\na H\|_{L^2}^2). \ee
For  $B_1$ defined in \eqref{an2}, it holds
\be\la{k13}B_1\geq C_5\|\na u\|_{L^2}^2-C\|\na H\|_{L^2}^2-CC_0^{1/4}\ee
owing  to \eqref{ljq01}, \eqref{p}, \eqref{infty12}, and \eqref{k11}.
Thus adding \eqref{k14} multiplied by
 $C_6\triangleq C_5^{-1}(C_4+1+C_5(C_3+1))$ to \eqref{k10} and choosing $\beta$ small enough, one arrives at that
\be\ba\la{k15}
&(C_6\si B_1+\si^2\varphi)'(t)\\
&+\sigma\|\rho^{1/2}\dot u\|_{L^2}^2+\frac{C_1}{2} \sigma^2\left(\|\na\dot{u}\|_{L^2}^2+\|\na H_t\|_{L^2}^2\right)+\frac{1}{2}\si^2\|\n^{1/2}\dot\te\|_{L^2}^2\\
&\le CC_0^{1/4}\si'+C\si\|H_t\|_{L^2}^2+C\si^2\|\n-1\|_{L^4}^4+ C( \|\na u\|^2_{L^2}+\|\na H\|^2_{L^2}+\|\na \te\|^2_{L^2}).\ea\ee

Finally, we claim
 \be \ba\label{eee2} \int_0^T\si^2\|\n-1\|_{L^4}^4 dt\leq CC_0^{1/4}.\ea\ee
Hence, integrating \eqref{k15} with respect to $t$ together  with \eqref{wq2}, \eqref{k13}, (\ref{eee2}), \eqref{k7}, \eqref{k88}, and (\ref{z1}) yields
\be\ba\notag
A_3(T)+A_4(T)\le \hat C_2C_0^{1/4}\le C_0^{1/6},
\ea\ee
which implies (\ref{b2.34}), provided
\be \ba\notag C_0\le\ve_2\triangleq \min\{\epsilon_{2,1},\hat C_2^{-12}\}.\ea\ee

It remains to prove \eqref{eee2}. It follows from $(\ref{h0})_1$ and (\ref{hj1})$_1$ that $\n-1$ satisfies
\be \ba\label{eee3} &(\n-1)_t+\frac{R}{2\mu+\lambda}(\n-1)\\
=&-u\cdot\na(\n-1)-(\n-1)\div u-\frac{G}{2\mu+\lambda}-\frac{R\n(\te-1)}{2\mu+\lambda}-\frac{|H|^2}{2(2\mu+\lambda)}.
\ea\ee
Multiplying \eqref{eee3} by $4(\n-1)^3$ and integrating the resulting equality over $\Omega$ by parts, we obtain after using \eqref{g1}, \eqref{p}, \eqref{z1}, \eqref{infty12}, \eqref{a2.112}, \eqref{o0}, and \eqref{goo1} that
\be\ba\label{eee5} &\left(\|\n-1\|_{L^4}^4\right)_t+\frac{4R}{2\mu+\lambda}\|\n-1\|_{L^4}^4\\
&=-3\int(\n-1)^4\div u dx-\frac{4}{2\mu+\lambda}\int(\n-1)^3Gdx\\
&\quad-\frac{4R}{2\mu+\lambda}\int(\n-1)^3\n(\te-1)dx-\frac{2}{2\mu+\lambda}\int(\n-1)^3|H|^2dx\\
&\leq \frac{2R}{2\mu+\lambda}\|\n-1\|_{L^4}^4+C\|\na u\|_{L^2}^2+C\|\n-1\|_{L^4}^3\|G\|_{L^2}^{1/4}\|\na G\|_{L^2}^{3/4}\\
&\quad+C\|\n-1\|_{L^4}^3(\|\n(\te-1)\|_{L^2}^{1/4}\|\na\te\|_{L^2}^{3/4}+\||H|^2\|_{L^2}^{1/4}\|\na |H|^2\|_{L^2}^{3/4})\\
&\leq \frac{3R}{2\mu+\lambda}\|\n-1\|_{L^4}^4+C(\|\na u\|_{L^2}^2+\|\na H\|_{L^2}^2)\\
&\quad+C(\|\n^{1/2}\dot u\|_{L^2}^3+\|H_t\|_{L^2}^3+\|\na\te\|_{L^2}^3).\ea\ee
Multiplying \eqref{eee5} by $\si^n$ with $n\ge 1$ and integrating the resulting inequality with respect to $t$, we deduce from \eqref{z1} and \eqref{a2.112} that
\be\ba\la{67}
&\int_0^T\si^n\|\n-1\|_{L^4}^4 dt\\
&\le CA_3^{1/2}(T)\int_0^T\si^{n-1}(\|\n^{1/2}\dot u\|_{L^2}^2+\|H_t\|_{L^2}^2+\|\na \te\|_{L^2}^2) dt\\
&\quad+CC_0^{1/4}+C\int_0^{\si(T)}\|\n-1\|_{L^4}^4dt\\
&\le CC_0^{1/4}+CC_0^{1/12}\int_0^T\si^{n-1}(\|\n^{1/2}\dot u\|_{L^2}^2+\|H_t\|_{L^2}^2)dt.
\ea\ee
Taking $n=2$ in \eqref{67}  as well as  \eqref{z1}  infers \eqref{eee2} directly.


The proof of Lemma \ref{le6} is completed.
\end{proof}

Since the basic energy estimate in Lemma \ref{a13.1} is ``weaker'', which is not enough to control $A_2(T)$, we need to re-establish the basic energy estimate  for short time $[0, \si(T)]$  to overcome this difficulty, and then show that the spatial $L^2$-norm of $\te-1$ could be  bounded by the combination of the initial energy and the spatial $L^2$-norm of $\na \te$.

\begin{lemma}\la{a13} Under the conditions of Proposition \ref{pr1},
	there exists  positive constant $C$
	depending only on
	$\mu,\,\lambda,\,\nu,\, \ka,\, R,\, \ga,\, \on, $ $\O$, and $M$ such
	that  if $(\rho,u,\te,H)$ is a smooth solution to the problem (\ref{h0})--(\ref{ch2})  on $\Omega\times (0,T] $ satisfying      (\ref{z1}) with $K$ as
	in Lemma \ref{le2},
	the following estimates  hold:
	\be \la{a2.121} \ba
	&\sup_{0\le t\le \si(T)}\int\left( \n |u|^2+(\n-1)^2 + \n(\te-\log\te-1)+|H|^2 \right)dx\le C C_0,\ea\ee
	and
	\be  \la{a2.17}  \ba
	\|\te(\cdot,t)-1\|_{L^2} \le C \left(C_0^{1/2} +C_0^{1/3}\|\na\te(\cdot,t)\|_{L^2}\right),
	\ea\ee
	for all $t\in(0,\si(T)].$
\end{lemma}
\begin{proof}
	The proof is divided into the following two steps.
	
	\noindent\textbf{Step I: The proof of (\ref{a2.121}).}
	
	First, adding (\ref{a11}) multiplied by $u$ to \eqref{mn} multiplied by $H$ , we obtain after using integration by parts and \eqref{h0}$_1$ that
	\be \la{a2.12} \ba
	&\frac{d}{dt}\int\left(\frac{1}{2}\n |u|^2+R(1+\n\log \n-\n)+\frac{1}{2}|H|^2
	\right)dx\\
	&+ \int(\mu|\o|^2+(2\mu+\lambda)(\div u)^2+\nu|\curl H|^2)dx \\
	&=  R \int \rho (\te -1) \div u dx.
	\ea\ee
Note that
 \be\la{hb}\te-\log\te-1 =(\te-1)^2\int_0^1\frac{\al}{\al (\te-1)+1}d\al\geq\frac{1}{2(\|\te(\cdot,t)\|_{L^{\infty}}+1)}(\te-1)^2,\ee
which as well as \eqref{a2.12}, \eqref{ljq01},  \eqref{z1}, and Cauthy inequality shows that
\be \la{a2.222} \ba
&\frac{d}{dt}\int\left(\frac{1}{2}\n |u|^2+R(1+\n\log\n-\n)+\frac{1}{2}|H|^2
\right)dx + C_7 \int(|\na u|^2+|\na H|^2)dx \\
	&\le C(\|\te(\cdot,t)\|_{L^{\infty}} +1)  \int \rho  (\te -\log \te -1)dx.
	\ea\ee
	Then, adding \eqref{a2.222} multiplied by $(2\mu+1) {C_7}^{-1}$  to \eqref{la2.7} leads to
\be \la{a2.22} \ba
&\xl((2\mu+1) {C_7}^{-1}+1\xr)\frac{d}{dt}\int\left(\frac{1}{2}\n |u|^2+R(1+\n\log \n-\n)+\frac{1}{2}|H|^2\right)dx
	\\&+ \frac{R}{\ga-1} \frac{d}{dt} \int \n(\te-\log \te-1)dx+\int(|\na u|^2+|\na H|^2)dx\\
	&\le C (\|\te(\cdot,t)\|_{L^{\infty}} +1)  \int \rho (\te -\log \te -1)dx.
	\ea\ee
	
Next, we claim that
	\be \ba\la{k}
	\int_0^{\si(T)}\|\te\|_{L^\infty}dt \le C.
	\ea \ee
	Combining this with \eqref{a2.22}, \eqref{a2.9}, and Gr\"onwall inequality  infers \eqref{a2.121} directly.
	
Now, we prove the claim \eqref{k}. Note that taking $n=1$ in \eqref{67} together with \eqref{z1} implies	\be\ba\la{eee23}	\int_0^T\si\|\n-1\|_{L^4}^4 dt\leq C.	\ea\ee
Then considering \eqref{k8} with $m=1$  and integrating the resulting inequality with respect to $t$, it follows from \eqref{wq2}, \eqref{m20}, \eqref{m22}, (\ref{z1}),  and \eqref{eee23} that
\be\ba\label{ae26}
&\sup_{0\le t\le T}\si \left(\|\rho^{1/2}\dot{u}\|_{L^2}^2+\|\na\te\|_{L^2}^2+\|H_t\|_{L^2}^2\right)\\
&+\int_0^T\si \left(\|\na\dot u\|_{L^2}^2+\|\n^{1/2}\dot\te\|_{L^2}^2+\|\na H_t\|_{L^2}^2\right)dt\\
	&\le C\int_0^t (\|\n^{1/2}  \dot u\|_{L^2}^2+\|H_t\|_{L^2}^2 +\|\na u\|_{L^2}^2+\|\na \te\|_{L^2}^2+\|\na H\|_{L^2}^2)d\tau \\
	&\quad+  C\int_0^t \si \|\n-1\|_{L^4}^4d\tau
	\\
	&\le  C.\ea\ee
Moreover, by virtue of (\ref{z1}), (\ref{lop4}), \eqref{m20},  \eqref{m22}, \eqref{eee23}, \eqref{o1}, and \eqref{ae26}, it holds
	\be\ba  \la{p1}
	\int_0^{T }\si \|\na^2\te\|_{L^2}^2dt
	&\le   C\int_0^{T } \left(\si \|\n \dot\te\|_{L^2}^2+
	\|\n\dot u \|_{L^2}^2+\|H_t\|_{L^2}^2  \right) dt\\
	&\quad+ C\int_0^{T } \left( \| \na u\|_{L^2}^2+\| \na\te\|_{L^2}^2+\| \na H\|_{L^2}^2+\si\|\n-1\|_{L^4}^4\right) dt \\
	&\le C,
	\ea\ee
which combined with
 (\ref{z1}) yields
	\be\la{3.88}\ba
	  \int_0^{\si(T)}\|\te-1\|_{L^\infty}dt
	&\le C  \int_0^{\si(T)}\|\na\te\|_{L^2}^{1/2} \left(\si\|\na^2\te\|^2_{L^2}\right)^{1/4}\si^{-1/4}dt\\
	&\le C \left(\int_0^{\si(T)} \|\na \te\|_{L^2}^2dt \int_0^{\si(T)}\si\|\na^2\te\|_{L^2}^2dt\right)^{1/4}	\\
	&\le CC_0^{1/16},
	\ea\ee
where we have used
\begin{equation}\label{6yue}\ba
\|\te-1\|_{L^\infty} \le  C\|\te-1\|_{L^6}^{1/2} \|\na\te\|_{L^6}^{1/2}
\le C\|\na\te\|_{L^2}^{1/2}\|\na^2 \te\|_{L^2}^{1/2}
\ea\end{equation}
due to  \eqref{g2} and \eqref{g1}.	Then, \eqref{k}  is a consequence of
\eqref{3.88}.

	\noindent \textbf{Step II: The proof of (\ref{a2.17}).}
	
	Denote
	\begin{equation*}
		(\te(\cdot,t)> 2)\triangleq \left.\left\{x\in
		\Omega\right|\te(x,t)> 2\right\},\quad
		(\te(\cdot,t)< 3)\triangleq
		\left.\left\{x\in \Omega\right|\te(x,t)<3\right\}.
	\end{equation*}
	Direct calculations combined with \eqref{hb} implies
	\be\la{eee9}\ba
	\te-\log\te-1
	&\ge \frac{1}{8} (\te-1)1_{(\te(\cdot,t)>2)
	}+\frac{1}{12}(\te-1)^21_{(\te(\cdot,t)<3)},  \ea\ee
which together with \eqref{a2.121} shows that
	\be \la{a2.11}\ba
	\sup_{0\le t\le \si(T)}\int \left(\n(\te-1)1_{(\te(\cdot,t)>2)}+\n(\te-1)^21_{(\te(\cdot,t)<3)}\right)dx \le C C_0.
	\ea\ee
	
	Now, for $t \in (0,\sigma(T)]$, taking $g(x)=\n(x,t)$, $f(x)=\te(x,t)-1$, $s=2$ and $\Sigma=(\te(\cdot,t)< 3)$ in \eqref{eee8}, it follows from  \eqref{a2.11} and \eqref{a2.121} that
	\begin{equation}\la{eee11}
		\|\te -1\|_{L^2(\te(\cdot,t)<3)}
		\le CC_0^{1/2}+ C C_0^{1/3} \|\na \te\|_{L^2}.
	\end{equation}
	Similarly, taking $g(x)=\n(x,t)$, $f(x)=\te(x,t)-1$, $s=1$ and $\Sigma=(\te(\cdot,t)>2)$ in \eqref{eee8}, and using \eqref{a2.11} and \eqref{a2.121}, it holds that
	\begin{equation}\notag
		\|\te -1\|_{L^1(\te(\cdot,t)>2)}
		\le CC_0+ C C_0^{5/6} \|\na \te\|_{L^2},
	\end{equation}
	which along with H\"older's inequality and \eqref{g1} leads to
	\be\notag\ba
	\|\te-1\|_{L^2(\te(\cdot,t)> 2)}
	&\le \|\te-1\|^{2/5}_{L^1(\te(\cdot,t)> 2)} \| \te-1\|_{L^6}^{3/5}\\
	&\le C \left(C_0^{2/5}+C_0^{1/3} \|\na\te\|_{L^2}^{2/5}\right) \|\na \te\|_{L^2}^{3/5}\\
	&\le CC_0^{1/2}+ C C_0^{1/3} \|\na \te\|_{L^2}.
	\ea\ee
Combining this with  \eqref{eee11}  concludes \eqref{a2.17}. 	
	The proof of Lemma \ref{a13}  is completed.
\end{proof}

\begin{remark}	One can deduce from (\ref{a2.8}) and (\ref{eee9}) that (\ref{a2.11}) holds for $0\le t\le T$ with $C_0$ replaced by $C_0^{1/4}$. Thus for all $t\in[0,T],$
	\be  \la{eee22}  \ba
	 \|\te(\cdot,t)-1\|_{L^2} \le C \left(C_0^{1/8} +C_0^{1/12}\|\na\te(\cdot,t)\|_{L^2}\right).
	\ea\ee
\end{remark}

With  \eqref{a2.17}, which is better than \eqref{eee22} in small time, in mind, we can establish the  estimate on $A_2(T)$.

\begin{lemma}\la{le3} Under the conditions of Proposition \ref{pr1},  there exists a positive constant $\ve_3$ depending only on $\mu,\,\lambda,\,\nu,\, \ka,\, R,\, \ga,\, \on,\,\bt $, $\O$, and $M$
	such that if $(\rho,u,\te,H)$ is a smooth solution to the problem (\ref{h0})--(\ref{ch2})  on $\Omega\times (0,T] $ satisfying  (\ref{z1})   with $K$
	as in Lemma \ref{le2}, the following estimate holds:
	\be\la{a2.34} A_2(T) \le C_0^{1/4},\ee
	provided $C_0\le \ve_3$.
\end{lemma}

\begin{proof}
First,	multiplying $(\ref{h0})_3$ by $(\te-1)$ and integrating the resulting equality by parts yield
\be\la{a2.231} \ba	& \frac{R}{2(\ga-1)} \frac{d}{dt}\int\n {(\te-1)^2}dx+ {\ka} \|\na\te\|_{L^2}^2\\&=-\int R \n \te{(\te-1)} \div u dx + \int  {(\te-1)} (\lambda (\div u)^2+2\mu |\mathfrak{D}(u)|^2+\nu|\curl H|^2) dx.\ea\ee
	Adding \eqref{a2.231} to \eqref{a2.12} infers that
	\be\la{a2.23} \ba
	&  \frac{d}{dt}\int
	\left(\frac{1}{2}\n |u|^2+R(1+\n\log \n-\n)+\frac{1}{2}|H|^2+\frac{R}{2(\ga-1)} {\n(\te-1)^2}\right)dx\\
	&+ \mu \|\o\|_{L^2}^2 +(2\mu+\lambda)\|\div u\|_{L^2}^2+  \ka \|\na\te\|_{L^2}^2+\nu\|\curl H\|_{L^2}^2\\
	&=-\int R \n {(\te-1)^2} \div u dx + \int  {(\te-1)} (\lambda (\div u)^2+2\mu |\mathfrak{D}(u)|^2+\nu|\curl H|^2) dx.
	\ea\ee
	
Note that by virtue of 	\eqref{g1}, \eqref{infty12}, and \eqref{z1}, one has
	\be\la{eee14} \ba
	\int \n|\te-1|^2 \div u dx
	\leq& C\|\n^{1/2}(\te-1)\|_{L^2}^{1/2}\|\n^{1/2}(\te-1)\|_{L^6}^{3/2}\|\na u\|_{L^2}\\
	\leq& CA_2^{1/4}(T)\|\na\te\|_{L^2}^{3/2}\|\na u\|_{L^2}\\
	\leq& CC_0^{1/16}\left(\|\na\te\|_{L^2}^2+\|\na u\|_{L^2}^2\right).
	 \ea\ee

 For the second term on the righthand side of \eqref{a2.23}, we deal with it for the short time $[0,\si(T)]$ and the large time $[\si(T),T]$, respectively.
	
	For $t\in[0,\si(T)]$, it follows from \eqref{a2.17}, \eqref{g1}, \eqref{3.30}, \eqref{infty12}, and \eqref{o1} that
\be\la{eee15} \ba
&\int|\te-1| (|\na u|^2+|\na H|^2) dx\\
	&\leq C\|\te-1\|_{L^2}^{1/2}\|\te-1\|_{L^6}^{1/2}(\|\na u\|_{L^2}\|\na u\|_{L^6}+\|\na H\|_{L^2}\|\na^2 H\|_{L^2})\\
	&\leq C \left(C_0^{1/4}\|\na\te\|_{L^2}^{1/2}+C_0^{1/6}\|\na\te\|_{L^2}\right)(\|\na u\|_{L^2}+\|\na H\|_{L^2})\\
	&\quad\cdot\left(\|\n \dot
	u\|_{L^2}+\|\na u\|_{L^2}+\|\na H\|_{L^2}+ \|\na \te\|_{L^2}+\|H_t\|_{L^2}+C_0^{1/24}\right)\\
	&\leq  C C_0^{7/24}(\|\n^{1/2} \dot u\|_{L^2}^2+\|H_t\|_{L^2}^2+1)\\
	&\quad+CC_0^{1/24}\left(\|\na\te\|_{L^2}^2+\|\na u\|_{L^2}^2+\|\na H\|_{L^2}^2\right). \ea\ee
	
	For $t\in[\si(T),T]$, one deduces from \eqref{g1} and \eqref{z1} that
	\be\la{eee16} \ba
&\int|\te-1|( |\na u|^2+|\na H|^2)dx\\
&\leq C\|\te-1\|_{L^6}(\|\na u\|_{L^2}^{3/2}\|\na u\|_{L^6}^{1/2}+\|\na H\|_{L^2}^{3/2}\|\na^2 H\|_{L^2}^{1/2})\\
	&\leq CC_0^{1/16}\left(\|\na\te\|_{L^2}^2+\|\na u\|_{L^2}^2+\|\na H\|_{L^2}^2\right), \ea\ee
	where one has used the following fact:
	\be\la{aa0}\ba \sup_{0\le t\le T}(\si(\|\na
	u\|_{L^6}+\|\na^2 H\|_{L^2}))\le CC_0^{1/24} \ea\ee
owing to \eqref{z1}, \eqref{o1}, and \eqref{3.30}.
	
	Submitting \eqref{eee14}--\eqref{eee16}  into \eqref{a2.23} and  using \eqref{z1}, it holds that
	\be \ba  \label{jia9}
	&  \sup_{ 0\le t\le T} \int \left(\frac{1}{2}\n |u|^2+ R(1+\n\log \n-\n)+\frac{1}{2}|H|^2+\frac{R}{2(\ga-1)} \n {(\te-1)^2}
	\right)dx\\
	&\quad+ \int_{0}^{T} (\mu \|\o\|_{L^2}^2 +(2\mu+\lambda)\|\div u\|_{L^2}^2+\nu\|\curl H\|_{L^2}^2+ \ka \|\na\te\|_{L^2}^2)dt\\
	&\le  CC_0+CC_0^{1/24} \int_{ 0}^{T}  ( \|\na u\|_{L^2}^2+\| \na \te \|_{L^2}^2 +\|\na H\|_{L^2}^2) dt
	\\&\quad + C C_0^{7/24}\int_{0}^{\si(T)} (\|\rho \dot u\|_{L^2}^2+\|H_t\|_{L^2}^2+1) dt\\
	&\le  C C_0^{7/24}.
	\ea \ee
Then combining \eqref{jia9} with     \eqref{ljq01} indicates that
	\be\notag A_2(T)\le \hat C_3 C_0^{7/24} \le C_0^{1/4},\ee
	provided \be \ba \label{jia11} C_0\le \ve_3\triangleq\min
	\left\{1,  \hat C_3^{-24}\right\}.\ea\ee The proof
	of Lemma \ref{le3} is completed.
\end{proof}

Now, we are in position to obtain the time-independent upper bound for the density, which turns out to be the key to obtaining all the higher order estimates and thus to extending the classical solution globally.

\begin{lemma}\la{le7}
Under the conditions of Proposition \ref{pr1}, there exists a positive constant $\ve_0$ depending only on   $\mu,\,\lambda,\,\nu,\, \ka,\, R,\, \ga,\, \on,\, \bt $, $\O$, and $M$ such that if $(\rho,u,\te,H)$ is a smooth solution to the problem (\ref{h0})--(\ref{ch2}) on $\Omega\times (0,T] $ satisfying  (\ref{z1}) with $K$ as
in Lemma \ref{le2}, the following estimate holds:
\be \la{a3.7}
\sup_{0\le t\le T}\|\n(\cdot,t)\|_{L^\infty}  \le
\frac{3\on }{2},
\ee
provided $C_0\le \ve_0$.
\end{lemma}

\begin{proof}
First, by virtue of  \eqref{6yue}, \eqref{p1}, and \eqref{z1},  one derives
\be\la{3.89}\ba
\int_{\si(T)}^T\|\te-1\|^2_{L^\infty}dt
\le & C\int_{\si(T)}^T\|\na\te\|_{L^2} \|\na^2\te\|_{L^2}dt\\
\le& C\left(\int_{\si(T)}^T\|\na\te\|^2_{L^2}dt\right)^{1/2}\left(\int_{\si(T)}^T \|\na^2\te\|^2_{L^2}dt\right)^{1/2} \\
\le& C C_0^{1/8}.
\ea\ee

Next, it follows from \eqref{g2}, (\ref{o0}), (\ref{bz6}),  \eqref{ae26}, and (\ref{z1}) that
\be\la{3.90}\ba &\int_0^{\si(T)}\|G\|_{L^\infty}dt\\
&\le C\int_0^{\si(T)}\|\na G\|_{L^2}^{1/2} \|\na G\|_{L^6}^{1/2}dt\\
&\le C \int_0^{\si(T)}(\|\n \dot u\|_{L^2}+\|H_t\|_{L^2}+\|\na H\|_{L^2})^{1/2}(\|\na\dot u\|_{L^2}+1)^{1/2}dt\\
&\quad+C\int_0^{\si(T)}(\|\n \dot u\|_{L^2}+\|H_t\|_{L^2}+\|\na H\|_{L^2})^{1/2}\|H_t\|_{L^2}^{3/4}dt\\
&\le C \sup_{0\le t\le T}(\si(\|\n \dot u\|_{L^2}+\|H_t\|_{L^2}+\|\na H\|_{L^2}))^{1/4}\\
&\quad\cdot \int_0^{\si(T)}(\si(\|\n \dot u\|_{L^2}^2+\|H_t\|_{L^2}^2+\|\na H\|_{L^2}^2))^{1/8}(\si(\|\na \dot u\|^2_{L^2}+1))^{1/4}\si^{-5/8}dt\\
&\quad+C \sup_{0\le t\le T}(\si(\|\n \dot u\|_{L^2}+\|H_t\|_{L^2}+\|\na H\|_{L^2}))^{1/2}\int_0^{\si(T)}(\si\|H_t\|_{L^2}^2)^{3/8}\si^{-7/8}dt\\
&\le CC_0^{1/48},
\ea\ee
and
 \be\la{3.91}\ba
&\int_{\si(T)}^T\|G\|^2_{L^\infty}dt\\
&\le C\int_{\si(T)}^T\|\na G\|_{L^2} \|\na G\|_{L^6} dt
   \\ &\le C\int_{\si(T)}^T\left(\|\n^{1/2} \dot u\|^2_{L^2}+
 \|\na\dot u\|_{L^2}^2 +\|\na H\|_{L^2}^2+\|H_t\|_{L^2}^2+\|H_t\|_{L^2}^3\right)dt\\ &\le CC_0^{1/6}.
    \ea\ee
Analogously, by \eqref{g1}, \eqref{g2}, \eqref{z1}, and \eqref{o1}, it holds that
\be\la{o2}\ba &\int_0^{\si(T)}\|H\|_{L^\infty}^2dt\\
&\le C\int_0^{\si(T)}\|\na H\|_{L^2} \|\na^2 H\|_{L^2}dt\\
&\le C \left(\int_0^{\si(T)}\|\na H\|_{L^2}^2dt\right)^{1/2}\left(\int_0^{\si(T)}(\|\na H\|_{L^2}^2+\|H_t\|_{L^2}^2)dt\right)^{1/2}\\
&\le CC_0^{1/8},
\ea\ee
and
 \be\la{o3}\ba
\int_{\si(T)}^T\|H\|^4_{L^\infty}dt
&\le C\int_{\si(T)}^T\|\na H\|_{L^2}^2 \|\na^2 H\|_{L^2}^2 dt
   \\ &\le C\int_{\si(T)}^T(\|\na H\|_{L^2}^2+\|H_t\|_{L^2}^2)dt\\ &\le CC_0^{1/6}.
    \ea\ee

Note that   $(\ref{h0})_1$ can be rewritten in terms of the Lagrangian coordinates as follows
\bnn\ba
(2\mu+\lambda) D_t \n&=-R \n(\n-1)- R\n^2(\te-1)-\n G-\frac{1}{2}\n|H|^2\\
&\le -R(\n-1)+C(\|\te-1\|_{L^\infty}+\| G\|_{L^\infty}+\|H\|_{L^\infty}^2),
\ea\enn
which leads to
\be\la{3.92}\ba
D_t (\n-1)+\frac{R}{2\mu+\lambda} (\n-1)\le C(\|\te-1\|_{L^\infty}+\| G\|_{L^\infty}+\|H\|_{L^\infty}^2).
\ea\ee

Next, taking
$$y=\n-1, \quad\al=\frac{R}{2\mu+\lambda} ,\quad g=C(\|\te-1\|_{L^\infty}+\| G\|_{L^\infty}+\|H\|_{L^\infty}^2),\quad T_1=\si(T),$$
in Lemma \ref{le1},  we thus obtain after using  (\ref{3.88}), \eqref{3.89}--(\ref{3.92}), and (\ref{2.34}) that
 \bnn\ba\n
 & \le \on+1 +C\left(\|g\|_{L^1(0,\si(T))}+\|g\|_{L^2(\si(T),T)}\right) \le \on+1 +\hat C_4C_0^{1/48} ,
 \ea\enn
 which implies \eqref{a3.7},
 provided
  \be \label{xjia11}C_0\le \ve_0\triangleq\min\left\{\ve_1,\cdots,\ve_4\right\}\ee
  with $\ve_4\triangleq\left(\frac{\hat \n-2 }{2\hat C_4 }\right)^{48}$.
The proof of Lemma \ref{le7} is finished.
\end{proof}

At the end of this section, we  summarize some uniform estimates on $(\n,u,\te,H)$, which guarantees the large-time behavior of the classcal solutions and the higher-order estimates in  next section.
\begin{lemma}\la{le8}
	In addition to the conditions of Proposition \ref{pr1}, assume $(\rho_0,u_0,\te_0,H_0)$ satisfies     (\ref{z01}) with $\ve_0$ as in Proposition \ref{pr1}.
Then there exists a  positive constant    $C $     depending only   on  $\mu,\,\lambda,\,\nu,\, \ka,\, R,\, \ga,\, \on,\,\bt$, $\O$, and $M$  such that if $(\rho,u,\te,H)$  is a smooth solution to the problem (\ref{h0})--(\ref{ch2}) on $\Omega\times (0,T] $  satisfying (\ref{z1}) with $K$ as in Lemma \ref{le2}, it holds:
	\be \la{ae3.7}\sup_{0\le t\le T}\si^2\int \n|\dot\te|^2dx + \int_0^T\si^2 \|\na\dot\te\|_{L^2}^2dt\le C.\ee
Moreover, it holds that
\be\la{vu15}\ba
&\sup_{0\le t\le T}\left(  \si\|\na u \|^2_{L^6}+\si\|H\|_{H^2}^2+\si^2\| \na\te\|^2_{H^1}\right)\\
&+\int_0^T(\si \|\na u \|_{L^4}^4+\si \|\na \te \|_{H^1}^2+\si\|H_t\|_{H^1}^2+\si^2\|\na\te_t\|^2_{L^2}+\si\|\n -1\|_{L^4}^4)dt\le C.
\ea\ee
\end{lemma}

\begin{proof}
First, operating the operator $\pa_t+\div(u\cdot) $ to (\ref{h0})$_3 $ and using  (\ref{h0})$_1$, one has
\be\la{3.96}\ba
&\frac{R}{\ga-1} \n \left(\pa_t\dot \te+u\cdot\na\dot \te\right)\\
&=\ka \Delta  \te_t +\ka \div (\Delta \te u)+\left( \lambda (\div u)^2+2\mu |\mathfrak{D}(u)|^2+\nu|\curl H|^2\right)\div u \\
&\quad +R\n \te  \pa_ku^l\pa_lu^k-R\n \dot\te \div u-R\n \te\div \dot u +2\lambda \left( \div\dot u-\pa_ku^l\pa_lu^k\right)\div u\\
&\quad + \mu (\pa_iu^j+\pa_ju^i)\left( \pa_i\dot u^j+\pa_j\dot u^i-\pa_iu^k\pa_ku^j-\pa_ju^k\pa_ku^i\right)\\
&\quad+2\nu(\curl H_t+u\cdot\na\curl H)\cdot\curl H.
\ea\ee
Multiplying (\ref{3.96}) by $\dot \te$ and integrating the resulting equality over $\O$ yield
\be\la{3.99}\ba
& \frac{R}{2(\ga-1)}\left(\int \n |\dot\te|^2dx\right)_t + \ka   \|\na\dot\te\|_{L^2}^2 \\
&\le  C  \int|\na \dot \te|\left( |\na^2\te||u|+ |\na \te| |\na u|\right)dx+C\int  \n|\te-1| |\na\dot u| |\dot \te|dx\\
&\quad +C  \int|\na u|^2|\dot\te|\left(|\na u|+|\te-1| \right)dx+C   \int |\na\dot u|\n|\dot \te| dx \\
&\quad +C   \int\left( |\na u|^2|\dot \te|+\n  |\dot
\te|^2|\na u|+|\na u| |\na\dot u| |\dot \te|\right)dx \\
&\quad+C\int|\na H||\dot\te|(|\na H||\na u|+|\na H_t|+|u||\na^2 H|)dx\\
&\le C\|\na u\|^{1/2}_{L^2}\|\na u\|^{1/2}_{L^6}\|\na^2\te\|_{L^2}\|\na \dot \te\|_{L^2}\\
&\quad+C\|\n(\te-1)\|_{L^2}^{1/2}\|\na\te\|_{L^2}^{1/2}\|\na\dot u\|_{L^2} \|\dot\te\|_{L^6}\\
&\quad+C  \|\na u\|_{L^2}\|\na u\|_{L^6}(\|\na u\|_{L^6}+\|\na \te\|_{L^2})
\|\dot\te\|_{L^6} +C  \|\na\dot u\|_{L^2} \|\n\dot\te\|_{L^2} \\
&\quad+C \|\na u\|^{1/2}_{L^6}\|\na u\|^{1/2}_{L^2} \|\dot\te\|_{L^6}(\|\na u\|_{L^2}
+\|\n\dot\te\|_{L^2}+\|\na\dot u\|_{L^2})  \\
&\quad+C\|\dot\te\|_{L^6}(\|\na^2H\|_{L^2}^2\|\na u\|_{L^2}+\|\na H\|_{L^2}^{1/2}\|\na^2 H\|_{L^2}^{1/2}\|\na H_t\|_{L^2})\\
&\le \frac{\ka}{2}  \|\na\dot\te\|_{L^2}^2 + C\|\na u\|_{L^6}^2\|\na \te \|_{L^2}^2+C\|\na u\|_{L^2}^2\|\na u\|_{L^6}^4\\
&\quad+C (1+\|\na u\|_{L^6}+\|\na\te\|_{L^2})(\|\na^2\te\|_{L^2}^2+\|\na\dot u\|_{L^2}^2+\|\rho^{1/2} \dot \te\|_{L^2}^2)\\
&\quad+ C\|\na u\|_{L^6}\|\na u\|_{L^2}^2+C\|\na^2 H\|_{L^2}^4+C\|\na^2 H\|_{L^2}\|\na H_t\|_{L^2}^2,
\ea\ee
where we have used  \eqref{g1}, \eqref{g2}, \eqref{z1}, and the following fact:
\be \ba\notag
 \int  (\Delta  \te_t + \div (\Delta \te u)) \dot \te dx &=  - \int  (\na  \te_t \cdot \na \dot\te + \Delta \te u \cdot \na \dot \te) dx\\
&= - \int  |\na \dot\te|^2 dx  + \int ( \na(u\cdot \na \te) \cdot \na \dot \te - \Delta \te u \cdot \na \dot \te) dx.
\ea \ee

Multiplying (\ref{3.99}) by $\si^2$ and integrating the resulting inequality over $(0,T),$
it follows from integration by parts, (\ref{z1}), \eqref{aa0}, (\ref{ae26}),  (\ref{p1}), and \eqref{o1} that
\bnn\ba
& \sup_{0\le t\le T}\si^2\int \n|\dot\te|^2dx + \int_0^T\si^2 \|\na\dot\te\|_{L^2}^2dt  \\
&\le C \sup_{0\le t\le T} \left(\si^2\|\na u\|_{L^6}^2 \right) \int_0^T\left( \|\na u\|_{L^2}^2 \|\na u\|_{L^6}^2+\|\na \te \|_{L^2}^2\right)dt \\
&\quad + C \sup_{0\le t\le T} \left(\si\left(1+\|\na u\|_{L^6}+\|\na\te\|_{L^2}\right)\right)\\
&\quad \quad \quad \quad \quad \quad \quad \cdot\int_0^T\si\left(\|\na^2\te\|_{L^2}^2+\|\na\dot u\|_{L^2}^2+\|\rho^{1/2} \dot \te\|_{L^2}^2\right)dt\\
&\quad +C\sup_{0\le t\le T} \left(\si\|\na u\|_{L^6} \right) \int_0^T\|\na u\|_{L^2}^2dt+C\sup_{0\le t\le T} \left(\si^2\|\na^2 H\|_{L^2}^2 \right) \int_0^T\|\na^2 H\|_{L^2}^2dt\\
&\quad+C\int_0^T\si\|\rho^{1/2} \dot \te\|_{L^2}^2dt+C\sup_{0\le t\le T} \left(\si\|\na^2 H\|_{L^2}\right) \int_0^T\si\|\na H_t\|_{L^2}^2dt\\
&\le  C, \ea\enn
where in the last inequality we have used
\be\ba\notag&\int_0^T\|\na u\|_{L^2}^2\|\na u\|_{L^6}^2dt\\
&\le C\int_0^T\|\na u\|_{L^2}^2\left(\|\n^{1/2} \dot u\|_{L^2}^2+\|\na u\|_{L^2}^2+ \|\na \te\|_{L^2}^2+ \|\na H\|_{L^2}^2+ \|H_t\|_{L^2}^2+1\right)dt \\
&\le C\int_0^T\left(\|\n^{1/2} \dot u\|_{L^2}^2+\|\na u\|_{L^2}^2+ \|\na \te\|_{L^2}^2+ \|\na H\|_{L^2}^2+ \|H_t\|_{L^2}^2\right)dt\\
&\le C\ea\ee
due to \eqref{3.30}, \eqref{z1}, and \eqref{infty12}.

Finally, we deduce from (\ref{z1}),  (\ref{3.30}), \eqref{infty12}, \eqref{a2.112}, \eqref{o1},  (\ref{lop4}), (\ref{m20}), (\ref{ae9}), (\ref{ae3.7}), and (\ref{eee23})--(\ref{p1}) that
\be \la{vu02}\ba
&\sup_{0\le t\le T}\left(  \si\|\na u \|^2_{L^6}+\si\|H\|_{H^2}^2+\si^2\| \na\te  \|^2_{H^1}\right)\\
 & +\int_0^T \left(\si\|\na u  \|_{L^4}^4+\si\|\na\te \|_{H^1}^2+\si\|H_t\|_{H^1}^2+\si\|\n-1\|_{L^4}^4\right)dt
 \le C, \ea\ee
which together with  (\ref{z1}) and \eqref{ae3.7} shows
 that\be\la{vu01}\ba
\int_0^T  \si^2 \|  \na\te _t\|_{L^2}^2dt
&\le C\int_0^T  \si^2\|\na \dot \te \|_{L^2}^2  dt+ C\int_0^T  \si^2\|\na(u \cdot\na  \te )\|_{L^2}^2dt\\
&\le C+C\int_0^T\si^2\|\na u \|_{L^2}\|\na u\|_{L^6}\|\na^2 \te \|_{L^2}^2dt  \\ &\le C.\ea\ee
The combination of (\ref{vu02}) and  (\ref{vu01}) leads to (\ref{vu15}) immediately.

The proof of Lemma \ref{le8} is completed.
\end{proof}

\section{\la{se4} A priori estimates (II): higher-order estimates}

This section is devoted to establishing the higher-order estimates of the classical solution $(\rho, u, \te,H)$ to the problem (\ref{h0})--(\ref{ch2})  on $ \Omega\times (0,T]$ with initial data $(\n_0 ,u_0,\te_0,H_0)$ satisfying (\ref{2.1}) and (\ref{3.1}). Moreover,
we assume that both (\ref{z1}) and (\ref{z01}) hold as well. To proceed,
we define $\tilde g $ as
\be \la{co12}\tilde g\triangleq\n_0^{-1/2}\left(
-\mu \Delta u_0-(\mu+\lambda)\na\div u_0+R\na (\n_0\te_0)-(\na\times H_0)\times H_0\right).\ee
Thus we deduce from (\ref{2.1}) and (\ref{3.1}) that
\be\la{wq01}\tilde g\in L^2.\ee
From now on, the generic constant $C $ will depend only  on \bnn
T, \,\, \| \tilde g\|_{L^2},    \,\|\n_0-1\|_{H^2 \cap W^{2,q}}  ,   \,  \,\| u_0\|_{H^2},  \ \,
\| \te_0-1\|_{H^1} ,\,  \,\| H_0\|_{H^2}, \enn
besides  $\mu,\,\lambda,\,\nu,\, \ka,\, R,\, \ga,\, \on,\,\bt, $ $\O$, and $M.$

To begin with, we derive the following estimates on the spatial gradient of
the smooth solution $(\rho,u,\te,H).$

\begin{lemma}\la{le11}
	The following estimates hold:
	\be\label{lee2}\ba
	&\sup_{0\le t\le T} \left(\|\rho^{1/2}\dot u\|_{L^2}^2 + \sigma\|\rho^{1/2}\dot \te\|_{L^2}^2 +\|\te-1\|_{H^1}^2 +\|H\|_{H^2}^2+ \sigma \|\na^2 \theta\|_{L^2}^2+\|H_t\|_{L^2}^2
 \right)  \\
	  &+\ia\left(  \|\nabla\dot u\|_{L^2}^2  + \|\rho^{1/2}\dot \te\|_{L^2}^2+ \|\nabla^2 \theta\|_{L^2}^2 +\sigma \|\nabla\dot \te\|_{L^2}^2+\|\na H_t\|_{L^2}^2 \right) dt\le C,
	\ea\ee
\be\ba\la{qq1}
&\sup_{0\le t\le T}\left(\|u\|_{H^2} +\|\n-1\|_{H^2}\right)\\
&	+ \int_0^{T}\left( \|\nabla u\|_{L^{\infty}}^{3/2} + \si \| \na^3 \te\|_{L^2}^2+\|u\|_{H^3}^2+\|H\|_{H^3}^2 \right)dt\le C.
\ea	\ee
\end{lemma}

\begin{proof}
First, for $\varphi(t)$ as in \eqref{wq3}, taking $m=0$ in \eqref{k8}, we obtain after  using (\ref{z1}), \eqref{a2.112}, \eqref{wq2}, (\ref{ae9}),  and \eqref{m20} that
\be\la{ae8}\ba
& \varphi'(t) + \frac{C_1}{2}\left(\|\nabla\dot{u}\|_{L^2}^2+\|\na H_t\|_{L^2}^2\right)+\|\n^{1/2}\dot\te\|_{L^2}^2 \\
&\le C\left(\|\n^{1/2} \dot u\|_{L^2}^2+\|H_t\|_{L^2}^2+\|\na u\|_{L^2}^2+\|\na H\|_{L^2}^2+\|\na    \te\|_{L^2}^2\right)\\
&\quad +C\left(\|\n\dot u\|_{L^2}^3+\|\na\te\|_{L^2}^3+\|H_t\|_{L^2}^3+\|\n-1\|_{L^4}^4\right) \\
&\le C   \left( \|\n^{1/2}\dot u\|_{L^2}^2 +\|H_t\|_{L^2}^2+ \|\na\te\|_{L^2}^2\right) (\varphi(t)+1)+C.
\ea  \ee
Taking into account  \eqref{h0} and \eqref{co12}, we can define
\bnn \ba
\sqrt{\n} \dot u(x,t=0) \triangleq
-\tilde g,
\ea\enn
and
\be\la{fa}
H_t(x,t=0)\triangleq-u_0 \cdot \nabla H_0+H_0 \cdot \nabla u_0- H_0 \mathop{\mathrm{div}} u_0-\nu \nabla \times \mathop{\rm curl} H_0,
\ee
which as well as \eqref{wq3}, \eqref{k6},  (\ref{e6}), (\ref{2.48}),  \eqref{wq01}, \eqref{b2},  \eqref{nnn}, and \eqref{o0} yields that
\be \la{wq02}
|\varphi(0)|\le C\| \tilde g\|_{L^2}^2+C\le C.
\ee
Then, integrating \eqref{ae8}  with respect to $t$
and applying Gr\"{o}nwall's inequality to the resulting inequality, we deduce from  \eqref{z1}, \eqref{wq02}, and (\ref{wq2}) that
\be\label{lee3}\ba
&\sup_{0\le t\le T} \left(\|\rho^{1/2}\dot u\|_{L^2}^2+\|H_t\|_{L^2}^2+\|\na \te\|_{L^2}^2 \right) \\
&+ \ia(\|\nabla\dot{u}\|_{L^2}^2+\|\na H_t\|_{L^2}^2+\|\n^{1/2}\dot\te\|_{L^2}^2)dt\le C,
\ea\ee
which together with \eqref{eee22} leads to
\be\la{ff1} \ba
\sup_{0\le t\le T} \|\te-1\|_{L^2} \le C.
\ea \ee

Next,
multiplying (\ref{3.99}) by $\sigma$ and integrating over $(0,T)$ show that
\be  \ba \la{a5}
&\sup\limits_{0\le t\le T} \si \int \n|\dot\te|^2dx+\int_0^T \si \|\na\dot\te\|_{L^2}^2dt\\
&\le
C\int_0^T\left(  \|\na^2\te\|_{L^2}^2+ \|\n^{1/2}\dot\te\|_{L^2}^2
+ \|\na H_t\|_{L^2}^2+\|\na\dot u\|_{L^2}^2 \right)dt + C\\
&\le  C,
\ea\ee
where we have used (\ref{lee3}),  (\ref{z1}), (\ref{m20}),  (\ref{lop4}), and the following fact:
\be\la{w1}\sup_{0\le t\le T}(\|\na u\|_{L^6}+\|H\|_{H^2})\le C\ee
owing to \eqref{3.30}, \eqref{o1}, \eqref{z1}, \eqref{a2.112}, and \eqref{lee3}.
Hence, it follows from \eqref{lee3}, \eqref{a5}, (\ref{lop4}),  \eqref{m20}, \eqref{z1}, and (\ref{w1}) that
\be \notag \sup\limits_{0\le t\le T} \si\|\na^2\te\|_{L^2}^2 + \int_{0}^{T}\|\na^2\te\|_{L^2}^2 dt \le C,\ee
which together with \eqref{lee3}--\eqref{w1} concludes \eqref{lee2}.

It remains to prove (\ref{qq1}).

First, note that by virtue of  \eqref{2tdu},   \eqref{z1}, \eqref{a2.112},  and  \eqref{lee2}, it holds
\be\ba\la{ua1}
\|\na^2 u\|_{L^p} 
& \le C\left(\|\n\dot u\|_{L^p}+\|\na(\n\te)\|_{L^p}+\||H||\na H|\|_{L^p}\right)+C\\
&\le  C\left(\|\na\dot u\|_{L^2}+\|\na^2\te\|_{L^2}\right)+C(\|\na^2\te \|_{L^2} + 1)\|\nabla\n\|_{L^p}+C.
\ea\ee
Combining this with standard calculations infers  that for $ 2\le p\le 6$,
\be\la{L11}\ba
\partial_t\norm[L^p]{\nabla\rho}
&\le C\norm[L^{\infty}]{\nabla u} \norm[L^p]{\nabla\rho}+C\|\na^2u\|_{L^p}\\
&\le C\left(1+\|\na u\|_{L^{\infty}}+\|\na^2\te \|_{L^2}\right)
\norm[L^p]{\nabla\rho}\\
&\quad +C\left(1+\|\na\dot u\|_{L^2}+\|\na^2\te \|_{L^2}\right). \ea\ee
To deal with $\|\na u\|_{L^\infty}$, we deduce from Lemma \ref{le9}, \eqref{z1}, and (\ref{ua1})  that
\be\la{u13}\ba
\|\na u\|_{L^\infty }
&\le C\left(\|{\rm div}u\|_{L^\infty}+\|\curl u\|_{L^\infty}\right)\log(e+\|\na^2 u\|_{L^6}) +C\|\na u\|_{L^2}+C \\
&\le C\left( \|{\rm div}u\|_{L^\infty } + \|\curl u\|_{L^\infty }
\right)\log(e+ \|\na\dot u\|_{L^2 } + \|\na^2\te \|_{L^2})\\
&\quad +C\left(\|{\rm div}u\|_{L^\infty }+ \|\curl u\|_{L^\infty } \right)
\log\left(e  + \|\na \n\|_{L^6}\right)+C,
\ea\ee
where
\be \la{w2}\ba
& \int_0^T\left(\|\div u\|^2_{L^\infty}+\|\curl u\|^2_{L^\infty} \right)dt \\
& \le  C\int_0^T\left(\|G\|^2_{L^\infty}+ \|\curl u\|^2_{L^\infty}+\|\n\te-1\|^2_{L^\infty}+\|H\|_{L^\infty}^4\right)dt \\
&\le  C\ia\left(\| \na G\|^2_{L^2}+\| \na G\|^2_{L^6} + \| \curl u\|^2_{W^{1,6}} + \|\te-1\|_{L^\infty}^2\right)dt + C \\
&\le C\ia(\|\na \dot u\|^2_{L^2}+\|\na^2\te \|_{L^2}^2)dt+C \\
&\le  C
\ea\ee
owing to (\ref{hj1})$_1$,  (\ref{g1}), \eqref{g2},  (\ref{lee2}), \eqref{o0}, \eqref{w1},  (\ref{z1}), (\ref{bz6}), and \eqref{jh}.

Denote
\bnn\la{gt}\begin{cases}
	f(t)\triangleq  e+\|\na	\n\|_{L^6},\\
	\tilde f(t)\triangleq 1+  \|{\rm div}u\|_{L^\infty }^2+ \|\curl u\|_{L^\infty }^2
	+ \|\na\dot u\|_{L^2 }^2 +\|\na^2\te \|_{L^2}^2,
\end{cases}\enn
then, in light of \eqref{u13},  (\ref{L11}) with  $p=6$ is equivalent to
\bnn f'(t)\le   C \tilde f(t) f(t)\ln f(t) ,\enn
which implies \be\notag  (\ln(\ln f(t)))'\le  C \tilde f(t),\ee
which along with  Gr\"{o}nwall's inequality, \eqref{w2}, and \eqref{lee2}  leads to
\be \la{u113} \sup\limits_{0\le t\le T}\|\nabla \rho\|_{L^6}\le C.\ee
This together with (\ref{u13}), \eqref{w2}, and (\ref{lee2}) gives
\be \la{v6}\ia\|\nabla u\|_{L^\infty}^{3/2}dt \le C.\ee
On the other hand, taking $p=2$ in \eqref{L11} combined with \eqref{v6}, (\ref{lee2}), and  Gr\"{o}nwall's inequality implies
  \be\la{aa94}\ba
\sup\limits_{0\le t\le T}\|\nabla \n\|_{L^2}
\le C,\ea\ee
which gives
  \be\la{aaa94}\ba
\sup\limits_{0\le t\le T}\|\nabla P\|_{L^2}
\le C\sup\limits_{0\le t\le T}\left(\|\na\te\|_{L^2}+\|\na\n\|_{L^2}+\|\te-1 \|_{L^3}\|\na\n\|_{L^6}\right)
\le C \ea\ee
owing to \eqref{z1}, (\ref{lee2}), and (\ref{u113}).
Then it follows from \eqref{aaa94}, \eqref{2tdu}, \eqref{z1}, \eqref{lee2},  and \eqref{a2.112}  that
\be\ba\la{aa95}
\sup\limits_{0\le t\le T} \|u\|_{H^2}
\le &C \sup\limits_{0\le t\le T}\left(\|\n\dot u\|_{L^2}+\|\na P\|_{L^2}+\||H||\na H|\|_{L^2}+1\right)\le C.
\ea\ee

Next, direct calculations show that
\be\la{ua2}\ba
\frac{d}{dt}\|\na^2\n\|^2_{L^2}
& \le C(1+\|\na u\|_{L^{\infty}})\|\na^2\n\|_{L^2}^2+C\|\na u\|^2_{H^2}\\
& \le C(1+\|\na u\|_{L^{\infty}} +\|\na^2\te \|_{L^2}^2)(1+\|\na^2\n\|_{L^2}^2) +C\|\na\dot u \|^2_{L^2}, \ea\ee
where one has used \eqref{z1}, \eqref{u113}, and the following estimate:
\be\ba\la{va2}
\|\nabla u\|_{H^2}&\le C\left(\|\n\dot u\|_{H^1}+\||H||\na H|\|_{H^1}+\|\n\te-1\|_{H^2}\right)+C\\
&\le C\|\na\n\|_{L^3}\|\dot u\|_{L^6}+C\|\na \dot u\|_{L^2}+C\|\n-1\|_{H^2}(\|\te-1\|_{H^2}+1)
\\&\quad+C\|\na^2\te\|_{L^2}+C
\\ &\le  C\|\na\dot u \|_{L^2}+ C  (1+\|\na^2\te \|_{L^2})(1+\|\na^2\n\|_{L^2})
\ea\ee
due to \eqref{2tdu}, (\ref{a2.112}), (\ref{lee2}), (\ref{u113}), \eqref{aa94},  and (\ref{z1}).
Then applying Gr\"{o}nwall's inequality to \eqref{ua2} and using  (\ref{lee2})  and (\ref{v6}), it holds
\be\la{ja3} \sup_{0\le t\le T} \|\na^2\n \|_{L^2}  \le C,\ee
which as well as  \eqref{va2}, \eqref{aa95}, and \eqref{lee2} yields
\be\la{ja4} \int_0^T(\|u\|_{H^3}^2+\|H\|_{H^3}^2) dt\le C ,\ee
where we have used the following fact:
\be\ba\la{d3}
\|\na^3 H\|_{L^2}
&\le C\|\curl^2 H\|_{H^1}+C\|\na H\|_{L^2}\\
&\le C(\|H_t\|_{H^1}+\||H||\na u|\|_{H^1}+\||\na H||u|\|_{H^1})+C\\
&\le C(\|\na H_t\|_{L^2}+1)
\ea\ee
owing to \eqref{2tdh}, \eqref{lee2}, \eqref{mn}, and \eqref{aa95}.
Finally,
combining the standard $H^1$-estimate to   elliptic problem (\ref{3.29}) with  \eqref{z1},  \eqref{lee2}, \eqref{u113}, \eqref{aa94}, and \eqref{aa95}, we obtain that
\be\la{ex4}\ba
\|\na^2\te\|_{H^1}
& \le C\left(\|\n \dot \te\|_{H^1}+\|\n\te\div u\|_{H^1}+\||\na u|^2\|_{H^1}+\||\na H|^2\|_{H^1}+\|\na\te\|_{L^2}\right)\\
& \le C\left(1+ \|\na \dot \te\|_{L^2} +  \|\rho^{1/2} \dot \theta\|_{L^2}  + \|\na(\n\te\div u)\|_{L^2}+ \||\na u||\na^2u|\|_{L^2} \right) \\
&\quad+C \||\na H||\na^2H|\|_{L^2}\\
& \le C\left(1+ \|\na \dot \te\|_{L^2} +  \|\rho^{1/2} \dot \theta\|_{L^2}+C\|\na^3 u\|_{L^2}+\|\na^3 H\|_{L^2} \right).
\ea\ee
Then (\ref{qq1}) is a consequence of  \eqref{lee2}, \eqref{a2.112},   \eqref{ja3}, \eqref{ja4},  \eqref{u113}--\eqref{aa94}, and \eqref{aa95}  immediately.  The proof of Lemma \ref{le11} is finished.
\end{proof}

\begin{lemma}\la{le9-1}
	The following estimates hold:
	\be\la{va5}\ba&
	\sup\limits_{0\le t\le T}
	\|\n_t\|_{H^1}
	+\int_0^T\left(\|  u_t\|_{H^1}^2+\si \| \te_t\|_{H^1}^2+\| \n u_t\|_{H^1}^2+\si \|\n \te_t\|_{H^1}^2
	\right)dt\le C,
	\ea  \ee
\be\la{d7}
\sup_{0\le t\le T}\si(\|\na H_t\|_{L^2}^2+\|H\|_{H^3}^2)+\int_0^T \sigma \left(\|H_{tt}\|_{L^2}^2+ \|\na H_t\|_{H^{1}}^2\right)dt\le C,
\ee
	\be\la{vva5}\ba
	\int_0^T \sigma \left(\|(\n u_t)_t\|_{H^{-1}}^2+\|(\n \te_t)_t\|_{H^{-1}}^2
	\right)dt\le C.\ea\ee
\end{lemma}
\begin{proof}
First,  one deduces from $(\ref{h0})_1$,  (\ref{qq1}), and \eqref{hs} that
\be\ba\la{sp1}
\|\n_t\|_{H^1}\le& \|\div (\rho u)\|_{H^1}
\le  C \|u\|_{H^2}
(\|\n-1\|_{H^2}+1)
\le C.
\ea \ee
Next, it follows from  (\ref{lee2}) and (\ref{qq1})   that
\be\label{va1}\ba
&\sup_{0\le t\le T}\int\left( \n|u_t|^2 + \si
\n\te_t^2\right)dx +\int_0^T \left(\|\na u_t\|_{L^2}^2+ \si \|\na\te_t\|_{L^2}^2\right)dt\\
&\le \sup\limits_{0\le t\le T}\int \left(\n|\dot u|^2+ \si\n|\dot\te|^2 \right)dx
+\int_0^T\left(\|\na\dot u\|_{L^2}^2+\si \|\na\dot\te\|_{L^2}^2 \right)dt\\
&\quad + \sup\limits_{0\le t\le T}\int \n\left(|u\cdot\na u|^2+\si |u\cdot\na\te|^2\right)dx\\
&\quad +\int_0^T\left(\|\na(u\cdot\na u)\|_{L^2}^2+\si \|\na(u\cdot\na \te)\|_{L^2}^2\right)dt\\
&\le C,\ea\ee
which along with (\ref{qq1}) shows that
\be\notag \ba
& \int_0^T\left(\|\na(\n u_t)\|_{L^2}^2+ \si \|\na(\n\te_t)\|_{L^2}^2\right)dt\\
& \le  C\int_0^T\left(\|  \na u_t \|_{L^2}^2+\|  \na \n\|_{L^3}^2\| u_t \|_{L^6}^2
+ \si \|  \na \te_t \|_{L^2}^2+ \si \|  \na \n\|_{L^3}^2\| \te_t \|_{L^6}^2 \right)dt\\
&\le C.\ea\ee
This together with (\ref{va1}),  (\ref{sp1}), \eqref{qq1}, and \eqref{eee8} infers   (\ref{va5}).

Note that  Lemma \ref{le11} leads to
\be\ba\la{d5}
&\|(H\cdot\na u-u\cdot\na H-H\div u)_t\|_{L^2}\\
&\le C(\|H_t\|_{L^6}\|\na u\|_{L^3}+\|H\|_{L^\infty}\|\na u_t\|_{L^2}+\|u_t\|_{L^6}\|\na H\|_{L^3}+\|\na H_t\|_{L^2}\| u\|_{L^\infty})\\
&\le C(\|\na H_t\|_{L^2}+\|\na u_t\|_{L^2}).
\ea\ee
Multiplying \eqref{sm} by $2 H_{tt}$ and integrating by parts, one has after using  \eqref{d5} that
\be\ba\la{pl}
&\nu(\|\curl H_t\|_{L^2}^2)'(t)+\|H_{tt}\|_{L^2}^2
\le C(\|\na H_t\|_{L^2}^2+\|\na u_t\|_{L^2}^2),
\ea\ee
which as well as \eqref{lee2},  \eqref{va5}, and \eqref{ljq01} yields
\be\la{d6}
\sup_{0\le t\le T}(\si\|\na H_t\|_{L^2}^2)+\int_0^T\si\|H_{tt}\|_{L^2}^2dt\le C.
\ee

Next, notice that \eqref{2tdh} holds as well for $H_t$. Combining this with \eqref{sm} and \eqref{d5} leads to
\be\la{d8}\ba
\|\na^2H_t\|_{L^2}
&\le C(\|\curl^2H_t\|_{L^2}+\|\na H_t\|_{L^2})\\
&\le C(\|H_{tt}\|_{L^2}+\|(H\cdot\na u-u\cdot\na H-H\div u)_t\|_{L^2}+\|\na H_t\|_{L^2})\\
&\le C(\|H_{tt}\|_{L^2}+\|\na H_t\|_{L^2}+\|\na u_t\|_{L^2}),
\ea\ee
which combined with \eqref{d6}, \eqref{lee2}, \eqref{d3}, and \eqref{va5} gives \eqref{d7}.

Finally, differentiating $(\ref{a11})$ with respect to $t$ yields that
\be   \la{va7}\ba (\n u_t)_t
=&-(\n u\cdot\na u)_t + \left((2\mu+\lambda)\na\div u-\mu\na\times\o \right)_t -\na P_t\\
&+(H\cdot\na H-\na H\cdot H)_t.\ea\ee
It   follows from (\ref{va5}), \eqref{qq1}, \eqref{lee2}, and \eqref{lop4} that
\be   \la{va9}\ba  \|(\n u\cdot\na u)_t  \|_{L^{2}}
&=\| \n_t u\cdot\na u+ \n u_t\cdot\na u + \n u\cdot\na u_t  \|_{L^{2}}\\
&\le C\|\n_t\|_{L^6} \|\na u\|_{L^3}+  C\|u_t\|_{L^6} \|\na u\|_{L^3}+  C\|u \|_{L^\infty} \|\na u_t\|_{L^2}\\
&\le C+   C\|\na u_t\|_{L^2},\ea\ee
\be   \la{va10}\ba  \|\na P_t  \|_{L^2}
&=R\|\n_t\na\te +\n \na\te_t +\na\n_t\te +\na\n\te_t\|_{L^2}\\
&\le C\left(\|\n_t\|_{L^6}\|\na\te\| _{L^3}+\|\na\te_t\|_{L^2}
+\|\te\|_{L^\infty}\|\na \n_t\|_{L^2}+\|\na\n\|_{L^3}\|\te_t\| _{L^6}\right)\\
&\le C+C\|\na\te_t\|_{L^2}+C\|\n^{1/2}\te_t\|_{L^2},\ea\ee
and
\be\la{d4}\ba
\|(H\cdot\na H-\na H\cdot H)_t\|_{L^2}
&\le C\|H_t\|_{L^6}\|\na H\|_{L^3}+C\|H\|_{L^\infty}\|\na H_t\|_{L^2}\\
&\le C\|\na H_t\|_{L^2}.
\ea\ee
The combination of (\ref{va7})--(\ref{d4}) and (\ref{va5}) concludes
\be\la{vva04}\ba\int_0^T \si \|(\n u_t)_t\|_{H^{-1}}^2dt\le C.\ea\ee
Analogously,  \bnn\int_0^T \si  \|(\n \te_t)_t\|_{H^{-1}}^2dt\le C,  \enn
which together with (\ref{vva04}) implies (\ref{vva5}). The proof of Lemma \ref{le9-1} is completed.
\end{proof}

\begin{lemma}\la{pe1}
	The following estimate holds:
	\be\la{nq1}
	\sup\limits_{0\le t\le T} \si\left(\|\nabla u_t\|^2_{L^2}+\|\n_{tt} \|^2_{L^2} + \|u\|_{H^3}^2\right)
	+ \int_0^T\si \left(\|\rho^{1/2} u_{tt}\|_{L^2}^2+ \|\nabla u_t\|_{H^1}^2\right)dt\le C.
	\ee
\end{lemma}
\begin{proof}
First, differentiating  $(\ref{a11})$  with respect to $t$ leads to \be\la{nt0}\ba \begin{cases}
(2\mu+\lambda)\na\div u_t-\mu \na\times \curl u_t+(H\cdot\na H-\frac{1}{2}\na |H|^2)_t\\
= \n u_{tt} +\n_tu_t+\n_tu\cdot\na u+\n u_t\cdot\na u +\n u\cdot\na u_t+\na P_t
,&\, \text{in}\,\O\times[0,T],\\
u_{t}\cdot n=0,\  \curl u_{t}\times n=0,  &\,\text{on}\,\p\O\times[0,T],\\
u_t\rightarrow0,\,\,&\,\text{as}\,\,|x|\rightarrow\infty.
 \end{cases}\ea\ee
Multiplying (\ref{nt0})$_1$ by $u_{tt}$   and integrating  the resulting equality  by parts yield
\be\la{sp9} \ba&
\frac{1}{2}\frac{d}{dt}\int \left(\mu|\curl u_t|^2 + (2\mu +\lambda)({\rm div}u_t)^2\right)dx
+\int \rho| u_{tt}|^2dx\\
&=\frac{d}{dt}\left(-\frac{1}{2}\int_{ }\rho_t |u_t|^2 dx-\int_{}\rho_t u\cdot\nabla u\cdot u_tdx
+\int_{ }P_t {\rm div}u_tdx\right.\\
&\quad\quad\quad\left.-\int\left((H\otimes H)_t:\na u_t-\frac{1}{2}|H|^2_t\div u_t\right)dx\right)\\
& \quad + \frac{1}{2}\int_{ }\rho_{tt} |u_t|^2 dx+\int_{ }(\rho_{t} u\cdot\nabla u )_t\cdot u_tdx
-\int_{ }\rho u_t\cdot\nabla u\cdot u_{tt}dx\\
&\quad - \int_{ }\rho u\cdot\nabla u_t\cdot u_{tt}dx - \int_{ }\left(P_{tt}-
\ka(\ga-1)\Delta\te_t\right){\rm div}u_tdx\\
&\quad +\ka(\ga-1)\int_{ } \na\te_t\cdot\na {\rm div}u_tdx\\
&\quad+\frac{1}{2}\int\left(2(H\otimes H)_{tt}:\na u_t-|H|^2_{tt}\div u_t\right)dx
\triangleq\frac{d}{dt}\tilde{I}_0+ \sum\limits_{i=1}^7 \tilde{I}_i. \ea \ee

 Each term $\tilde{I}_i(i=0,\cdots,7)$ can be estimated as follows:

We deduce from simple calculations, $(\ref{h0})_1,$   (\ref{va5}), \eqref{qq1}, \eqref{lee2}, \eqref{eee8}, and (\ref{va1}) that
\be \ba \la{sp10}
|\tilde{I}_0|&=\left|-\frac{1}{2}\int\rho_t |u_t|^2 dx\right|+C\norm[L^3]{\rho_t}\|u\|_{L^\infty}\norm[L^2]
{\nabla u}\norm[L^6]{u_t}\\
&\quad+C\|(\n\te)_t\|_{L^2}\|\nabla u_t\|_{L^2}+C\|H\|_{L^\infty}\|H_t\|_{L^2}\|\na u_t\|_{L^2}\\
&\le C\int  \n |u||u_t||\nabla u_t| dx+C\norm[L^2]{\na u_t}\\
&\quad +C(\|\n^{1/2} \te_t\|_{L^2}+\|\n_t\|_{L^2}+\|\n_t\|_{L^3}\|\te-1\|_{L^6})\|\nabla u_t\|_{L^2}\\
&\le C (1+\|\n^{1/2}\te_t\|_{L^2}) \|\nabla u_t\|_{L^2} ,\ea\ee
\be \la{sp11}\ba
2|\tilde{I}_1|&=\left|\int \rho_{tt} |u_t|^2 dx\right|\\
& \le  C\|\n_{tt}\|_{L^2}\|u_t\|_{L^2}^{1/2}\|u_t\|_{L^6}^{3/2}\\
& \le  C\|\n_{tt}\|_{L^2}(1+\|\na u_t\|_{L^2})^{1/2}\|\na u_t\|_{L^2}^{3/2}\\
& \le  C\|\n_{tt}\|_{L^2}^2+C \|\na u_t\|_{L^2}^4+C,
\ea \ee
\be \la{sp12}\ba
|\tilde{I}_2|&=\left|\int \left(\rho_t u\cdot\nabla u \right)_t\cdot u_{t}dx\right|\\
& = \left|  \int\left(\rho_{tt} u\cdot\nabla u\cdot u_t +\rho_t
u_t\cdot\nabla u\cdot u_t+\rho_t u\cdot\nabla u_t\cdot u_t\right)dx\right|\\
&\le   C\norm[L^2]{\rho_{tt}}\norm[L^6]{\na u}\norm[L^6]{u}\norm[L^6]{u_t}
+C\norm[L^2]{\rho_t}\norm[L^6]{u_t}^2\norm[L^6]{\nabla u} \\
&\quad+C\norm[L^3]{\rho_t}\norm[L^{\infty}]{u}\norm[L^2]{\nabla u_t}\norm[L^6]{u_t}\\
& \le C\norm[L^2]{\rho_{tt}}^2 + C\norm[L^2]{\nabla u_t}^2, \ea \ee
\be\ba\la{sp13}
|\tilde{I}_3|+|\tilde{I}_4|&= \left| \int \rho u_t\cdot\nabla u\cdot u_{tt} dx\right|
+\left| \int \rho u\cdot\nabla u_t\cdot u_{tt}dx\right|\\
& \le   C\|\n^{1/2}u_{tt}\|_{L^2}\left(\|u_t\|_{L^6}\|\na u\|_{L^3}
+\|u\|_{L^\infty}\|\na u_t\|_{L^2}\right) \\
& \le  \frac{1}{4}\norm[L^2]{\rho^{{1/2}}u_{tt}}^2 + C \norm[L^2]{\nabla u_t}^2,\ea\ee
and
\be\la{i7}\ba
|\tilde I_7|
&=\left|\frac{1}{2}\int\left(2(H\otimes H)_{tt}:\na u_t-|H|^2_{tt}\div u_t\right)dx\right|\\
&\le C\|\na u_t\|_{L^2}(\|H\|_{L^\infty}\|H_{tt}\|_{L^2}+\|H_t\|_{L^2}^{1/2}\|\na H_t\|_{L^2}^{3/2})\\
&\le C(\|\na u_t\|_{L^2}^2+\|H_{tt}\|_{L^2}^2+\|\na H_t\|_{L^2}^3).
\ea\ee
Then it follows from
(\ref{op3}), \eqref{va10}, \eqref{d3}, and Lemma \ref{le11} that
\bnn\ba & \|P_{tt}-\ka(\ga-1)\Delta \te_t\|_{L^2}\\
&\le  C\|(u\cdot\na P)_t\|_{L^2}+C\|(P\div u)_t\|_{L^2}+C\||\na u||\na u_t|\|_{L^2}+C\||\na H||\na H_t|\|_{L^2}\\
&\le  C\|u_t\|_{L^6}\|\na P\|_{L^3}+C\|u\|_{L^\infty}\|\na P_t\|_{L^2}
+C\|P_t\|_{L^6}\|\na u\|_{L^3}\\
&\quad  +C\|P\|_{L^\infty}\|\na u_t\|_{L^2}+C\|\na u\|_{L^\infty}\|\na u_t\|_{L^2}+C\|\na H\|_{L^\infty}\|\na H_t\|_{L^2}\\
&\le C\left(1+\|\na u\|_{L^\infty}+\|\na^2\te \|_{L^2}\right)\|\na u_t\|_{L^2}\\
&\quad +C\left(1+\|\na\te_t\|_{L^2}+\|\n^{1/2}\te_t\|_{L^2}\right)+C\|\na H_t\|_{L^2}^2,
\ea\enn
which implies
\be\ba\la{sp15}
|\tilde{I}_5|&=\left|\int\left(P_{tt}-\ka(\ga-1)\Delta \te_t\right){\rm div}u_tdx\right|\\
&\le C\norm[L^2]{P_{tt}-\ka(\ga-1)\Delta \te_t}\norm[L^2]{\na u_t}\\
&\le C\left(1+\|\na u\|_{L^\infty}+\|\na^2\te \|_{L^2}\right)\|\na u_t\|_{L^2}^2\\
&\quad+C\left(1+\|\na\te_t\|^2_{L^2}+\|\n^{1/2}\te_t\|_{L^2}^2+\|\na H_t\|_{L^2}^4\right).
\ea\ee
Next,  for the term $\tilde I_6$, Cauthy's inequality implies
\be \la{asp16}\ba
|\tilde{I}_6| &= \left| \ka(\ga-1)\int_{ } \na\te_t\cdot\na{\rm div}u_tdx \right|   \\
& \le  C \|\na^2u_t\|_{L^2}\|\na\te_t\|_{L^2}\\
& \le \frac{1}{4} \|\n^{1/2} u_{tt}\|^2_{L^2}+ C\left(1+\|\na u_t\|^2_{L^2}+\|\n^{1/2} \te_t\|^2_{L^2}+\|\na\te_t\|^2_{L^2}+\|\na H_t\|_{L^2}^2\right),
\ea\ee
where we have used  the following  standard elliptic estimate to Lam\'{e}'s system \eqref{nt0}:
\be\la{nt4}\ba
&\|\na^2u_t\|_{L^2}\\
&\le C\|\n u_{tt}\|_{L^2}
+C\|\n_t\|_{L^3}\|u_t\|_{L^6}+C\|\n_t\|_{L^3}\|\na u\|_{L^6}\|u\|_{L^\infty}\\
&\quad  +C\|u_t\|_{L^6}\|\na u\|_{L^3}+C\|\na u_t\|_{L^2}\|u\|_{L^\infty}+C\|\na P_t\|_{L^2}\\
&\quad+C\|H_t\|_{L^6}\|\na H\|_{L^3}+C\|H\|_{L^\infty}\|\na H_t\|_{L^2}\\
&\le C\left(\|\n u_{tt}\|_{L^2}+\|\na u_t\|_{L^2}+\|\n^{1/2} \te_{t}\|_{L^2}+\|\na H_t\|_{L^2}+\|\na \te_t\|_{L^2}+1\right)\ea\ee
due to Lemma  \ref{le11}, \eqref{va5}, and \eqref{va10}.

Submitting   (\ref{sp11})--(\ref{asp16}) into
(\ref{sp9}) arrives at
\be\la{4.052} \ba
& \frac{d}{dt}\int \left(\mu|\curl u_t|^2 + (2\mu +\lambda)({\rm div}u_t)^2-2\tilde{I}_0\right)dx
+\int \rho| u_{tt}|^2dx\\
&\le  C\left(1+\|\na u\|_{L^\infty}+\|\na u_t\|_{L^2}^2+\|\na^2\te \|_{L^2} \right)\|\na u_t\|_{L^2}^2 \\
&\quad +C\left(1+\|\n_{tt}\|_{L^2}^2+\|\n^{1/2} \te_t\|^2_{L^2}+\|\na \te_t\|_{L^2}^2+\|H_{tt}\|_{L^2}^2+\|\na H_t\|_{L^2}^4\right).\ea\ee
Multiplying \eqref{4.052} by $\sigma$ and integrating the resulting inequality over $(0,T)$, it follows from  \eqref{ljq01}, \eqref{lee2}, (\ref{qq1}),  (\ref{va1}), (\ref{sp10}), \eqref{d7},  (\ref{va5}),
and Gr\"{o}nwall's inequality that
\be\la{nq11}
\sup\limits_{0\le t\le T} \si \|\nabla u_t\|^2_{L^2}
+ \int_0^T\si\int\rho |u_{tt}|^2dxdt
\le C,
\ee
where we have used the following estimate:
\be \la{s4} \ba
\|\n_{tt}\|_{L^2} &= \|\div(\rho u)_t\|_{L^2} \\
& \le C\left(\|\n_t\|_{L^6}\|\nabla u\|_{L^3}+ \|\nabla u_t\|_{L^2}
+\|u_t\|_{L^6}\|\nabla \n\|_{L^3}+\|\nabla \n_t\|_{L^2}\right) \\ &\le C+C\|\na u_t\|_{L^2}\ea\ee
owing to  $(\ref{h0})_1$, \eqref{qq1}, and (\ref{va5}).

Finally, we deduce from Lemma \ref{le11}, \eqref{s4}, (\ref{va2}), \eqref{nt4}, \eqref{va5}, and  (\ref{nq11}) that
\be\notag
\sup\limits_{0\le t\le T}\si\left(\|\n_{tt}\|_{L^2}^2+\|u\|^2_{H^3}\right) + \int_0^T \si\|\nabla u_t\|_{H^1}^2 dt\le C,
\ee
which combined with \eqref{nq11} shows (\ref{nq1}).
The proof of Lemma \ref{pe1} is finished.
\end{proof}

\begin{lemma}\la{pr3} For $q\in (3,6)$ as in Theorem \ref{th1}, it holds that
	\be\la{y2}\ba
	&\sup_{0\le t\le T} \|\n-1\|_{W^{2,q}} +\int_0^T
	 \|\na^2u\|_{W^{1,q}}^{p_0}  dt\le C,
	\ea \ee
	where  \be 1< \la{pppppp} p_0<4q/(5q-6) \in (1,4/3).\ee
\end{lemma}

\begin{proof}
First, it follows from \eqref{2tdu} and Lemma \ref{le11} that
\be\la{a4.74}\ba
\|\na^2u\|_{W^{1,q}}
\le & C\left(\|\n \dot u\|_{W^{1,q}}+\|\na P\|_{W^{1,q}}+\||H||\na H|\|_{W^{1,q}}\right)+C\\
\le & C
(\|\na\dot u\|_{L^2}+\|\na(\n\dot u)\|_{L^q}+  \|\na^2\te\|_{L^2}+ \|\te\na^2\n\|_{L^q}\\
& +\| |\na\n||\na\te|\|_{L^q}+  \|\na^2\te\|_{L^q}+\|\na H\|_{L^{2q}}^2+\||H||\na^2H|\|_{L^q}+1)\\
\le & C\left(\|\na\dot u\|_{L^2}+\|\na(\n\dot u)\|_{L^q}  + \| \na^2 \theta\|_{L^q}+\|\na^3 H\|_{L^2} \right)\\
&+ C(1+ \|\na^2 \theta\|_{L^2})(\| \na^2 \n\|_{L^q} +1),
\ea\ee
which combined with $(\ref{h0})_1$, \eqref{qq1}, and (\ref{va5}) yields
\be\la{sp28}\ba
&\frac{d}{dt}\|\na^2\n\|_{L^q}^q\\
&\le C\|\na u\|_{L^\infty}\|\na^2\n\|_{L^q}^q
+C\|\na^{2} u\|_{W^{1,q}}\|\na^2\n\|_{L^q}^{q-1}(\|\na\n\|_{L^q}+1)\\
&\le C\| u\|_{H^3}\|\na^2\n\|_{L^q}^q+C\|\na^2 u\|_{W^{1,q}}\|\na^2\n\|_{L^q}^{q-1}\\
&\le C\left(\| u\|_{H^3}   + \|\na\dot u\|_{L^2} +\|\na(\n\dot u)\|_{L^q}+\|\na^2 \theta\|_{L^q}+\|\na^3 H\|_{L^2}+\|\na^2 \theta\|_{L^2}+1\right)\\
&\quad\cdot\left(\|\na^2 \n\|_{L^q}^q+1\right).
\ea\ee

Note that by Lemma \ref{le11}, (\ref{g1}),  and \eqref{nq1}, one has
\be\la{4.49}\ba     \|\na(\n\dot u)\|_{L^q}
&\le C\|\na \n\|_{L^6}\|\na \dot u \|_{L^2}^{q/(3(q-2)}\|\na\dot u\|_{L^q}^{2(q-3)/3(q-2)}+C\|\na\dot u \|_{L^q}\\
&\le C\|\na\dot u \|_{L^2}+C\|\na\dot u \|_{L^q}\\
&\le C\|\na u_t \|_{L^2}+ C\|\na u_t \|_{L^q}+C\|\na(u\cdot \na u ) \|_{L^2}+C\|\na(u\cdot \na u ) \|_{L^q}\\
&\le C\si^{-1/2}+ C\|\na u_t \|_{L^2}^{(6-q)/2q}\|\na u_t \|_{L^6}^{3(q-2)/{2q}}\\
& \quad+C\|\na u \|_{L^6}^{6/q}\| \na u \|_{L^\infty}^{2(q-3)/{q}}
+C\| u \|_{L^\infty}\|\na^2 u \|_{L^q}\\
&\le C\si^{-1/2} \left(\si\|\na u_t \|^2_{H^1}\right)^{3(q-2)/{4q}}
+C\|u\|_{H^3}+C\si^{-1/2},\ea \ee
and
\begin{equation}\notag
    \begin{aligned}
    \| \na^2 \theta\|_{L^q} \le& C \|\na^2 \theta\|_{L^2}^{(6-q)/2q} \|\na^3 \theta\|_{L^2}^{3(q-2)/2q}\\
    \le & C \si^{-1/2} \left(\si\|\na^3 \theta \|^2_{L^2} \right)^{3(q-2)/{4q}}.
\end{aligned}
\end{equation}
Combining these with Lemma \ref{le11} and \eqref{nq1} shows that, for $p_0$ as in (\ref{pppppp}),
\be \la{4.53}\int_0^T \left(\|\na(\n \dot u)\|^{p_0}_{L^q} + \|\na^2 \theta\|_{L^q}^{p_0} \right) dt\le C. \ee

Finally, applying  Gr\"{o}nwall's
inequality to (\ref{sp28}), we deduce from  \eqref{lee2},  (\ref{qq1}), and (\ref{4.53}) that
\bnn  \sup\limits_{0\le t\le T}\|\na^2 \n\|_{L^q}\le C,\enn
which along with   Lemma \ref{le11},    (\ref{4.53}),
and (\ref{a4.74}) infers (\ref{y2}).  The proof of Lemma \ref{pr3} is finished.
\end{proof}

\begin{lemma}\la{sq90} For $q\in (3,6)$ as in Theorem \ref{th1}, the following estimate holds:
	\be \ba\la{eg17}
	&\sup_{ 0\le t\le T}\si \left(\|\te_t\|_{H^1}+\|H_{tt}\|_{L^2}+\| u_t\|_{H^2}+\| H_t\|_{H^2}\right)\\
	&+\sup_{ 0\le t\le T}\si \left(\|\na^2\te\|_{H^1}+\| u\|_{W^{3,q}}+\|H\|_{H^4}\right)\\
	&   +\int_0^T   \si^2(\|\na u_{tt}\|_{L^2}^2+\|\na H_{tt}\|_{L^2}^2) dt\le C.\ea  \ee
	
\end{lemma}

\begin{proof}
First,  differentiating $(\ref{nt0})$ and $(\ref{sm})$ with
respect to $t$ respectively yields
\be\la{sp30}\ba \begin{cases}
\n u_{ttt}+\n u\cdot\na u_{tt}-(2\mu+\lambda)\nabla{\rm div}u_{tt} +\mu\na \times \curl u_{tt}\\
= 2{\rm div}(\n u)u_{tt} +{\rm div}(\n u)_{t}u_t-2(\n u)_t\cdot\na u_t- \n u_{tt}\cdot\na u
\\\quad -(\n_{tt} u+2\n_t u_t) \cdot\na u-\na P_{tt}+(H\cdot\na H-\frac{1}{2}\na |H|^2)_{tt}, & \text{in}\,\O\times[0,T],\\
u_{tt} \cdot n=0,\quad \curl u_{tt} \times n=0, &  \text{on}\,\p\O\times[0,T],\\
u_{tt}\rightarrow0,\,\,&\text{as}\,\,|x|\rightarrow\infty,
\end{cases}\ea \ee
\be\ba\label{smm}
\begin{cases}
H_{ttt}+\nu \nabla \times (\curl H_{tt})
=(H \cdot \nabla u-u \cdot \nabla H-H \div u)_{tt},&\text{in}\quad\Omega\times[0,T],\\
H_{tt} \cdot n=0,\quad \curl H_{tt} \times n=0,& \text{on}\quad \partial\Omega\times[0,T],\\
H_{tt}\rightarrow 0,&\text{as}\quad |x|\rightarrow 0.
 \end{cases}
 \ea\ee
Multiplying (\ref{sp30})$_1$ and (\ref{smm})$_1$ by $u_{tt}$ and $H_{tt}$ respectively and integrating the resulting equality over ${\Omega}$ by parts lead to
\be \la{sp31}\ba
&\frac{1}{2}\frac{d}{dt}\int (\n |u_{tt}|^2+|H_{tt}|^2)dx\\
&+\int \left((2\mu+\lambda)({\rm div}u_{tt})^2+\mu|\curl u_{tt}|^2+\nu|\curl H_{tt}|^2\right)dx\\
&=-4\int  u^i_{tt}\n u\cdot\na u^i_{tt} dx
-\int (\n u)_t\cdot \left(\na (u_t\cdot u_{tt})+2\na u_t\cdot u_{tt}\right)dx\\
&\quad -\int (\n_{tt}u+2\n_tu_t)\cdot\na u\cdot u_{tt}dx
-\int   \n u_{tt}\cdot\na u\cdot  u_{tt} dx\\
& \quad+\int  P_{tt}{\rm div}u_{tt}dx-\frac{1}{2}\int\left(2(H\otimes H)_{tt}:\na u_{tt}-|H|^2_{tt}\div u_{tt}\right)dx\\
&\quad+\int(H\cdot\na u-u\cdot\na H-H\div u)_{tt}H_{tt}dx
\triangleq\sum_{i=1}^7\tilde{J}_i.\ea\ee
We deduce from Lemmas \ref{le11}--\ref{pe1}, (\ref{va1}), \eqref{eee8},   and \eqref{s4} that, for $\eta\in(0,1],$
\be \la{sp32} \ba
|\tilde{J}_1| &\le C\|\n^{1/2}u_{tt}\|_{L^2}\|\na u_{tt}\|_{L^2}\| u \|_{L^\infty} \le \eta \|\na u_{tt}\|_{L^2}^2+C(\eta) \|\n^{1/2}u_{tt}\|^2_{L^2},\ea\ee
\be \la{sp33}\ba
|\tilde{J}_2| &\le C\left(\|\n u_t\|_{L^3}+\|\n_t u\|_{L^3}\right)
\left(\| \na u_{tt}\|_{L^2}\| u_t\|_{L^6}+\| u_{tt}\|_{L^6}\| \na u_t\|_{L^2}\right)\\
&\le C\left(\|\n^{1/2} u_t\|^{1/2}_{L^2}\|u_t\|^{1/2}_{L^6}+\|\n_t\|_{L^6}\| u\|_{L^6}\right)\| \na u_{tt}\|_{L^2}\| \na u_t\|_{L^2}\\
&\le \eta \|\na u_{tt}\|_{L^2}^2+C(\eta)\| \na u_t\|_{L^2}^{3}+C(\eta)\\
&\le \eta \|\na u_{tt}\|_{L^2}^2+C(\eta)\si^{-3/2}  ,\ea\ee
\be  \la{sp34}\ba
|\tilde{J}_3| &\le C\left(\|\n_{tt}\|_{L^2}\|u\|_{L^6}+
\|\n_{t}\|_{L^2}\|u_{t}\|_{L^6} \right)\|\na u\|_{L^6}\|u_{tt}\|_{L^6} \\
&\le \eta \|\na u_{tt}\|_{L^2}^2+C(\eta)\si^{-1}  ,\ea\ee
\be  \la{sp36}\ba &
|\tilde{J}_4|+|\tilde{J}_5|\\
&\le  C\|\n u_{tt}\|_{L^2} \|\na u\|_{L^3}\|u_{tt}\|_{L^6}
+C \|(\n_t\te+\n\te_t)_t\|_{L^2}\|\na u_{tt}\|_{L^2}\\
&\le  \eta \|\na u_{tt}\|_{L^2}^2+C(\eta) \left(\|\n^{1/2}u_{tt}\|^2_{L^2}
+\|\n_{tt}\te\|_{L^2}^2+\|\n_{t}\te_t\|_{L^2}^2 +\|\n^{1/2}\te_{tt}\|_{L^2}^2\right) \\
&\le  \eta \|\na u_{tt}\|_{L^2}^2+C(\eta)\left( \|\n^{1/2}u_{tt}\|^2_{L^2}
+\|\na\te_t\|_{L^2}^2 +\|\n^{1/2}\te_{tt}\|_{L^2}^2+\sigma^{-2}\right), \ea\ee
\be\la{j6}\ba
|\tilde J_6|
&\le C\|\na u_{tt}\|_{L^2}(\|H\|_{L^\infty}\|H_{tt}\|_{L^2}+\|H_t\|_{L^2}^{1/2}\|\na H_t\|_{L^2}^{3/2})\\
&\le \eta\|\na u_{tt}\|_{L^2}^2+C(\eta)(\|H_{tt}\|_{L^2}^2+\|\na H_t\|_{L^2}^3)\\
&\le \eta\|\na u_{tt}\|_{L^2}^2+C(\eta)(\|H_{tt}\|_{L^2}^2+\si^{-3/2}),
\ea\ee
and
\be\la{j7}\ba
|\tilde J_7|
&\le C\|H_{tt}\|_{L^2}(\|H\|_{L^\infty}\|\na u_{tt}\|_{L^2}+\|u\|_{L^\infty}\|\na H_{tt}\|_{L^2})\\
&\quad+C\|H_{tt}\|_{L^2}(\|{H_{tt}}\|_{L^6}\|\na u\|_{L^3}+\|{u_{tt}}\|_{L^6}\|\na H\|_{L^3})\\
&\quad+C\|H_{tt}\|_{L^6}(\|{H_{t}}\|_{L^3}\|\na u_t\|_{L^2}+\|{u_{t}}\|_{L^3}\|\na H_t\|_{L^2})\\
&\le \eta(\|\na u_{tt}\|_{L^2}^2+\|\na H_{tt}\|_{L^2}^2)+C(\eta)(\|H_{tt}\|_{L^2}^2+\|H_t\|_{H^1}^2\|u_t\|_{H^1}^2)\\
&\le \eta(\|\na u_{tt}\|_{L^2}^2+\|\na H_{tt}\|_{L^2}^2)+C(\eta)(\|H_{tt}\|_{L^2}^2+\si^{-2}).
\ea\ee
Putting (\ref{sp32})--(\ref{j7}) into (\ref{sp31}), we obtain after using \eqref{ljq01} and choosing $\eta$ suitably small  that
\be \la{ex12}\ba
& \frac{d}{dt}\int (\n |u_{tt}|^2+|H_{tt}|^2)dx+C_8 \int(|\na u_{tt}|^2+|\na H_{tt}|^2)dx \\
& \le  C\si^{-2} +C\|\n^{1/2}u_{tt}\|^2_{L^2}
+C\|\na\te_t\|_{L^2}^2+C\|H_{tt}\|_{L^2}^2+C_9\|\n^{1/2}\te_{tt}\|_{L^2}^2.\ea\ee

Next, differentiating \eqref{3.29} with respect to $t$ gives
\be\la{eg1}\ba \begin{cases}
-\frac{\ka(\ga-1)}{R}\Delta \te_t+\n\te_{tt}\\
=-\n_t\te_{t}- \n_t\left(u\cdot\na \te+(\ga-1)\te\div u\right)-\n\left( u\cdot\na
\te+(\ga-1)\te\div u\right)_t\\
\quad+\frac{\ga-1}{R}\left(\lambda (\div u)^2+2\mu |\mathfrak{D}(u)|^2+\nu|\curl H|^2\right)_t,  \qquad \qquad   \quad\,\,\text{in}\,\O\times[0,T],\\
\na \te_t\cdot n=0,  \qquad \qquad \qquad \qquad\qquad\qquad\qquad \qquad\qquad\ \ \ \ \ \ \,   \text{on}\,\p\O\times[0,T],\\
\nabla\te_t\rightarrow0,\qquad \qquad \qquad \qquad\qquad\qquad\qquad \qquad\qquad\qquad\ \ \ \  \text{as}\,|x|\rightarrow\infty.
\end{cases}\ea\ee
Multiplying  (\ref{eg1})$_1$ by $\te_{tt}$ and integrating the resulting
equality over $\Omega$ imply
\be\la{ex5}\ba
& \left(\frac{\ka(\ga-1)}{2R}\|\na \te_t\|_{L^2}^2+H_0\right)_t+ \int\n\te_{tt}^2dx \\
&=\frac{1}{2}\int\n_{tt}\left( \te_t^2
+2\left(u\cdot\na \te+(\ga-1)\te\div u\right)\te_t\right)dx\\
&\quad + \int\n_t\left(u\cdot\na\te+(\ga-1)\te\div u \right)_t\te_{t}dx\\
& \quad-\int\n\left(u\cdot\na\te+(\ga-1)\te\div u\right)_t\te_{tt}dx\\
& \quad -\frac{\ga-1}{R}\int \left(\lambda (\div u)^2+2\mu |\mathfrak{D}(u)|^2+\nu|\curl H|^2\right)_{tt}\te_t dx
\triangleq\sum_{i=1}^4H_i,\ea\ee
where
\bnn\ba H_0\triangleq & \frac{1}{2}\int \n_t\te_{t}^2dx
+\int\n_t\left(u\cdot\na\te+(\ga-1)\te\div u\right) \te_tdx\\
&- \frac{\ga-1}{R}\int\left(\lambda (\div u)^2+2\mu |\mathfrak{D}(u)|^2 +\nu|\curl H|^2\right)_t\te_t dx. \ea\enn
It follows from  $(\ref{h0})_1,$  (\ref{eee8}),  (\ref{va1}), \eqref{s4}, and Lemmas \ref{le11}--\ref{pe1} that
\be\la{ex6}\ba
|H_0|\le & C\int \n|u||\te_{t}||\na\te_{t}|dx+ C(\|\na u\|_{L^3}\|\na u_t\|_{L^2}+\|\na H\|_{L^3}\|\na H_t\|_{L^2}) \|\te_t\|_{L^6} \\
&+C\|\n_t\|_{L^3}\|\te_t\|_{L^6}\left( \|\na\te\|_{L^2} \|u\|_{L^\infty}+ \|\na u\|_{L^2}+ \|\theta-1\|_{L^6}\|\na u\|_{L^3}\right)\\
\le &C  \|\na\te_t\|_{L^2}\left(\|\n^{1/2}\te_t\|_{L^2}+\|\na u_t\|_{L^2}+\|\na H_t\|_{L^2}+1\right)\\
\le &\frac{\ka(\ga-1)}{4R} \|\na\te_t\|_{L^2}^2+C\si^{-1},\ea\ee
\be\la{ex7}\ba
|H_1|& \le C\|\n_{tt}\|_{L^2}\left(\|\te_t\|_{L^4}^{2}
+\|\te_t\|_{L^6}\left(\|u\cdot\na \te\|_{L^3}+\|\na u\|_{L^3}+\|\te-1\|_{L^6}\|\na u\|_{L^6} \right)\right)\\
&\le C\|\n_{tt}\|_{L^2}\left(\|\rho^{1/2} \te_t\|_{L^2}^{2} + \|\na \te_t\|_{L^2}^{2}
+\si^{-1/2}  \right) \\
& \le  C(1+\|\na u_{t}\|_{L^2} )\|\na \te_t\|^2_{L^2}+C\si^{-3/2},\ea\ee
\be\la{ex10}\ba
|H_4|&\le C\int \left(|\na u_t|^2+|\na u||\na u_{tt}|+|\na H_t|^2+|\na H||\na H_{tt}|\right)|\te_t|dx\\
&\le C\left(\|\na u_t\|_{L^2}^{3/2}\|\na u_t\|_{L^6}^{1/2}
+ \|\na u\|_{L^3} \|\na u_{tt}\|_{L^2}\right)\|\te_t\|_{L^6}\\
&\quad+ C\left(\|\na H_t\|_{L^2}^{3/2}\|\na H_t\|_{L^6}^{1/2}
+ \|\na H\|_{L^3} \|\na H_{tt}\|_{L^2}\right)\|\te_t\|_{L^6}\\
&\le \de(\|\na u_{tt}\|^2_{L^2}+\|\na H_{tt}\|^2_{L^2})+C(\|\na^2 u_t\|^2_{L^2}+\|\na^2 H_t\|^2_{L^2})\\
&\quad+C(\de)\|\na\te_t\|_{L^2}^2+C\si^{-2}(\|\na u_t\|_{L^2}^2+\|\na H_t\|^2_{L^2}),\ea\ee
and
\be\la{ex9}\ba
|H_2|+|H_3|&\le C\left(\si^{-1/2}\|\na u_t\|_{L^2}+\|\na\te_t\|_{L^2}\right)
\left(\|\n_t\|_{L^3} \|\te_t\|_{L^6}+\|\n \te_{tt}\|_{L^2}\right)\\
&\le \frac{1}{2}\int\n\te_{tt}^2dx+C\|\na\te_t\|_{L^2}^2+C\si^{-1 } \|\na u_t\|^2_{L^2}, \ea\ee
where in the last inequality we have used the following fact:
\be\la{eg12}\ba
 &\|\left(u\cdot\na\te+(\ga-1)\te\div u \right)_t\|_{L^2}\\
& \le  C\left(\|u_t\|_{L^6}\|\na\te\|_{L^3}+\|\na\te_t\|_{L^2}+\|\te_t\|_{L^6}\|\na u\|_{L^3}
+\|\te\|_{L^\infty}\|\na u_t\|_{L^2}\right)\\
& \le  C\|\na u_t\|_{L^2}(\|\na^2 \te\|_{L^2}+1)+ C\|\na \te_t\|_{L^2}\ea\ee
due to  Lemma \ref{le11}.

Then, substituting (\ref{ex7})--(\ref{ex9}) into (\ref{ex5}) infers
\be\la{ex11}\ba
& \left(\frac{\ka(\ga-1)}{2R}\|\na \te_t\|_{L^2}^2+H_0\right)_t
+\frac{1}{2}\int\n\te_{tt}^2dx \\
&\le  \de(\|\na u_{tt}\|^2_{L^2}+\|\na H_{tt}\|^2_{L^2})+C(\de)((1+\|\na u_{t}\|_{L^2}) \|\na \te_t\|^2_{L^2}+\si^{-3/2})\\
&\quad +C(\|\na^2 u_t\|^2_{L^2}+\|\na^2 H_t\|^2_{L^2})+C\si^{-2}(\|\na u_t\|_{L^2}^2+\|\na H_t\|^2_{L^2}).\ea\ee

Finally, for $C_9$ as in (\ref{ex12}), adding (\ref{ex11}) multiplied by
$2 (C_9+1) $ to  (\ref{ex12})
and choosing $\de$ suitably small, one derives
\be\la{ex13}\ba
& \left[ 2 (C_9+1)\left(\frac{\ka(\ga-1)}{2R}\|\na \te_t\|_{L^2}^2+H_0\right)
+\int (\n |u_{tt}|^2+|H_{tt}|^2)dx\right]_t\\
&\quad + \int\n\te_{tt}^2dx+\frac{C_8}{2}\int (|\na u_{tt}|^2+|\na H_{tt}|^2)dx\\
&\le C (1+\|\na u_{t}\|_{L^2}^2+\|\na H_{t}\|_{L^2}^2) (\si^{-2} +\|\na \te_t\|^2_{L^2})
+C\|\n^{1/2}u_{tt}\|^2_{L^2}\\
&\quad +C\|H_{tt}\|_{L^2}^2+ C(\|\na^2 u_t\|^2_{L^2}+\|\na^2 H_t\|^2_{L^2}).\ea\ee
Multiplying (\ref{ex13}) by $\si^2$ and integrating the resulting inequality over $(0,T),$
we  obtain after using (\ref{ex6}),  \eqref{lee2},   (\ref{nq1}), \eqref{d7},  (\ref{va5}), and Gr\"{o}nwall's inequality that
\be \la{eg10}\ba
&\sup_{ 0\le t\le T}\si^2\int \left(|\na\te_t|^2+\n |u_{tt}|^2+|H_{tt}|^2\right)dx\\
&+\int_{0}^T\si^2\int \left(\n\te_{tt}^2+|\nabla u_{tt}|^2+|\na H_{tt}|^2\right)dxdt\le C,\ea\ee
which as well as  Lemmas \ref{le11}--\ref{pr3}, (\ref{nt4}), (\ref{ex4}),  \eqref{va1}, (\ref{a4.74}), \eqref{d8},
and (\ref{4.49}) yields
\be\la{sp20} \sup_{ 0\le t\le T}\si \left(\|\na u_t\|_{H^1}
+ \|H_t\|_{H^2}+\|H\|_{H^4}+\|\na^2\te\|_{H^1}+\|\na^2u\|_{W^{1,q}} \right)\le C,\ee
where we have used the following estimate:
\be\ba\notag
\|\na ^4H\|_{L^2}
&\le C(\|\curl^2 H\|_{H^2}+\|\na H\|_{L^2})\\
&\le C (\|H_{t}\|_{H^2}+\|H\cdot\na u-u\cdot\na H-H\div u\|_{H^2}+1)\\
&\le C(\|H_{t}\|_{H^2}+\|H\|_{H^2}\|\na u\|_{H^2}+\|u\|_{H^2}\|\na H\|_{H^2}+1)
\ea\ee
owing to Lemma \ref{le11}, \eqref{2tdh}, \eqref{mn}, and  \eqref{hs}.

Hence, one concludes  (\ref{eg17}) by using  (\ref{eg10}),   (\ref{sp20}), \eqref{va1}, \eqref{eee8},
and (\ref{qq1}). The proof of Lemma \ref{sq90} is completed.
\end{proof}

\begin{lemma}\la{sq91} The following estimate holds:
	\be \la{egg17}\sup_{ 0\le t\le T}\si^2  \left(\|\na^2\te\|_{H^2}+\| \te_t\|_{H^2}+\|\n^{1/2}\te_{tt}\|_{L^2}  \right)
	+\int_0^T\si^4\|\na \te_{tt}\|_{L^2}^2 dt\le C.\ee
\end{lemma}

\begin{proof}
First, differentiating $(\ref{eg1})$ with respect to $t$ gives
\be\la{eg2}\ba \begin{cases}
\n\te_{ttt}-\frac{\ka(\ga-1)}{R}\Delta \te_{tt}\\
=-\n u
\cdot\na\te_{tt} - \n_{tt}\left(\te_t+ u\cdot\na \te+(\ga-1)\te\div u\right)\\
\quad + 2\div(\n u)\te_{tt}- 2\n_t\left(u\cdot\na \te+(\ga-1)\te\div u\right)_t\\
\quad - \n\left(u_{tt}\cdot\na \te+2u_t\cdot\na\te_t+(\ga-1)(\te\div u)_{tt}\right)\\
\quad +\frac{\ga-1}{R}\left(\lambda (\div u)^2+2\mu |\mathfrak{D}(u)|^2+\nu|\curl H|^2 \right)_{tt}, \quad\quad&\text{in}\,\O\times[0,T],\\
\na\te_{tt}\cdot n=0,  \qquad \qquad\qquad &\text{on}\,\p\O\times[0,T],\\
\na\te_{tt}\rightarrow 0\qquad\qquad\qquad\qquad\qquad\qquad\ \ \ \ \   &\text{as}\,|x|\rightarrow\infty.
\end{cases}\ea\ee
Multiplying (\ref{eg2})$_1$ by $\te_{tt}$ and integrating the resulting equality over $\Omega$ lead to
\be\la{eg3}\ba
&\frac{1}{2}\frac{d}{dt}\int\n|\te_{tt}|^2dx +\frac{\ka(\ga-1)}{R}\int|\na \te_{tt}|^2dx\\
&=-4\int \te_{tt}\n u\cdot\na\te_{tt}dx  -\int \n_{tt}\left(\te_t
+ u\cdot\na \te+(\ga-1)\te\div u\right)\te_{tt}dx\\
&\quad - 2\int\n_t\left(u\cdot\na \te+(\ga-1)\te\div u\right)_t\te_{tt}dx\\
& \quad - \int\n\left(u_{tt}\cdot\na \te+2u_t\cdot\na\te_t
+(\ga-1)(\te\div u)_{tt}\right)\te_{tt}dx\\
& \quad +\frac{\ga-1}{R}\int  \left(\lambda (\div u)^2+2\mu |\mathfrak{D}(u)|^2+\nu|\curl H|^2\right)_{tt}\te_{tt}dx
\triangleq \sum_{i=1}^5\tilde K_i.\ea\ee

It follows from
Lemmas \ref{le11}--\ref{pe1}, \ref{sq90}, and \eqref{eee8}   that
 \be \la{eg4} \ba
\si^4|\tilde K_1|&\le C\si^4\|\n^{1/2}\te_{tt}\|_{L^2}\|\na \te_{tt}\|_{L^2}\|u\|_{L^\infty}\\
&\le \de \si^4\|\na \te_{tt}\|_{L^2}^2+C(\de) \si^4\|\n^{1/2}\te_{tt}\|^2_{L^2} ,\ea\ee
\be \la{eg16}\ba
\si^4|\tilde K_2|&\le C \si^4\|\n_{tt}\|_{L^2}\|\te_{tt}\|_{L^6} \left( \|\te_t\|_{H^1}+\|\na\te\|_{L^3}
+1\right) \\
&\le \de\si^4\|\na \te_{tt}\|_{L^2}^2+C(\de)
,\ea\ee
\be \la{eg7}\ba
\si^4|\tilde K_4|&\le C\si^4\|\te_{tt}\|_{L^6}
\left( \|\na\te\|_{L^3}\|\n u_{tt}\|_{L^2}
+\|\na\te_t\|_{L^2}\|\na u_t\|_{L^2}\right)\\
&\quad+ C\si^4\|\te_{tt}\|_{L^6} \left( \|\na u\|_{L^3}\|\n \te_{tt}\|_{L^2}
+ \|\na u_t\|_{L^2}\| \te_t\|_{L^3}\right)\\
&\quad+C\si^4\|\te\|_{L^\infty}\|\n\te_{tt}\|_{L^2} \|\na u_{tt}\|_{L^2} \\
&\le \de\si^4\|\na \te_{tt}\|_{L^2}^2
+C(\de)\left(\si^4 \|\n^{1/2} \te_{tt}\|_{L^2}^2+\si^3\|\na u_{tt}\|_{L^2}^2  \right)+C(\de),\ea\ee
\be \la{eg8}\ba
\si^4|\tilde K_5|&\le C\si^4\|\te_{tt}\|_{L^6}
\left( \|\na u_t\|_{L^2}^{3/2}\|\na u_{t}\|_{L^6}^{1/2}+\|\na u\|_{L^3}\|\na u_{tt}\|_{L^2}\right)  \\
&\quad+C\si^4\|\te_{tt}\|_{L^6}
\left( \|\na H_t\|_{L^2}^{3/2}\|\na H_{t}\|_{L^6}^{1/2}+\|\na H\|_{L^3}\|\na H_{tt}\|_{L^2}\right)  \\
&\le \de\si^4\|\na \te_{tt}\|_{L^2}^2
+C(\de)\si^4\left(\|\na u_{tt}\|_{L^2}^2 +\|\na H_{tt}\|_{L^2}^2\right) +C(\de),\ea\ee
and
\be \la{eg6}\ba
\si^4|\tilde K_3|&\le C \si^4\|\n_t\|_{L^3} \|\te_{tt}\|_{L^6}
\left( \si^{-1/2}\|\na u_t\|_{L^2} +\|\na\te_{t}\|_{L^2} \right) \\
&\le \de\si^4\|\na \te_{tt}\|_{L^2}^2 +C(\de),\ea\ee
where in the last inequality we have used (\ref{eg12}).

Next, multiplying (\ref{eg3})  by $\si^4$ and substituting (\ref{eg4})--(\ref{eg6}) into the resulting inequality, one obtains after  choosing $\de$ suitably small that
\bnn \ba
& \frac{d}{dt}\int\si^4\n|\te_{tt}|^2dx +\frac{\ka(\ga-1)}{R}\int\si^4|\na \te_{tt}|^2dx\\
& \le  C\si^2\left(\|\n^{1/2} \te_{tt}\|_{L^2}^2
+\|\na u_{tt}\|_{L^2}^2 +\|\na H_{tt}\|_{L^2}^2 \right)+C,\ea\enn
which along with (\ref{eg10})   gives
\be\la{eg13} \sup_{ 0\le t\le T}\si^4\int  \n |\te_{tt}|^2dx
+\int_{0}^T\si^4\int_{ } |\nabla \te_{tt}|^2 dxdt\le C.\ee

Finally, applying the standard $L^2$-estimate  to (\ref{eg1}), we deduce from  Lemmas \ref{le11}--\ref{pe1}, (\ref{va1}),  (\ref{eg13}),
and (\ref{eg17}) that
\be\la{eg14}\ba
&\sup_{0\le t\le T}\si^2\|\na^2\te_t\|_{L^2}\\
&\le  C\sup_{0\le t\le T}\si^2\left(\|\n\te_{tt}\|_{L^2}
+  \|\n_t\|_{L^3}\|\te_t\|_{L^6}+\|\n_t\|_{L^6} \left(\|\na\te\|_{L^3}+1\right)\right)\\
& \quad +C\sup_{0\le t\le T}\si^2\left(\|\n^{1/2} \te_t\|_{L^2}+\|\na\te_t\|_{L^2}+ (1+\|\na^2\te\|_{L^2}) \|\na u_t\|_{L^2}\right)\\
&\quad+C\sup_{0\le t\le T}\si^2(
 \|\na u_t\|_{L^6}+\|\na H_t\|_{L^6})\\
& \le  C.\ea\ee
Moreover, it follows from  the standard $H^2$-estimate  of
$(\ref{3.29})$, (\ref{hs}), \eqref{d7},  \eqref{nq1}, and Lemma \ref{le11}   that
\bnn\ba &\|\na^2\te\|_{H^2}\\
&\le C\left(\|\n\te_t\|_{H^2}+\|\n u\cdot\na\te\|_{H^2}
+\|\n\te\div u\|_{H^2}+\||\na u|^2\|_{H^2}+\||\na H|^2\|_{H^2} \right)\\
&\le C\left((1+\|\n-1\|_{H^2} )\|\te_t\|_{H^2}
+(\|\n-1\|_{H^2}+1) \| u\|_{H^2}\|\na\te\|_{H^2}\right)\\
&\quad+C(1+\|\n-1\|_{H^2})(1+\|\te-1\|_{H^2})  \| \div u\|_{H^2}+C(\|\na u\|^2_{H^2}+\|\na H\|^2_{H^2})\\
&\le C\si^{-1}+ C\| \na^3\te \|_{L^2}+C\| \te_t\|_{H^2}.\ea\enn
Combining this with
  (\ref{eg17}),  (\ref{eg14}), and  (\ref{eg13}) concludes (\ref{egg17}).
We finish the proof of Lemma \ref{sq91}.
\end{proof}

\section{\la{se5}Proof of  Theorem  \ref{th1}}

With all the a priori estimates in Sections \ref{se3} and \ref{se4}
at hand, we are in a  position to prove the main result  of this paper in this section.

\begin{pro} \la{pro2}

 For  given numbers $M>0$ (not necessarily small),
  $\on> 2,$ and $\bt>1,$   assume that  $(\rho_0,u_0,\te_0,H_0)$ satisfies (\ref{2.1}),  (\ref{3.1}),
and   (\ref{z01}). Then    there exists a unique classical solution  $(\rho,u,\te,H) $      of problem (\ref{h0})--(\ref{ch2})
 in $\Omega\times (0,\infty)$ satisfying (\ref{mn5}) and (\ref{mn2}) with $T_0$ replaced by any $T\in (0,\infty).$
  Moreover,  (\ref{zs2}), (\ref{a2.112}), (\ref{ae3.7}), and  (\ref{vu15})  hold for any $T\in (0,\infty).$

 \end{pro}

\begin{proof}
First, with the help of the standard local existence result (Lemma \ref{th0}), there exists a small $T_0>0$ which may depend on
$\inf\limits_{x\in \Omega}\n_0(x), $  such that the  problem
 (\ref{h0})--(\ref{ch2})  with   initial data $(\n_0 ,u_0,\te_0,H_0 )$
 has   a unique classical solution $(\n,u,\te,H)$ on $\O\times(0,T_0],$  which satisfies (\ref{mn6})--(\ref{mn2}). Now we will use the a priori estimates in Proposition \ref{pr1} to extend the local classical solution to all time.
Since
\bnn A_1(0)\le M^2,\quad  A_2(0)\le  C_0^{1/4},\quad A_3(0)+A_4(0)=0, \quad  \n_0<
 \hat{\rho},\quad \te_0\le \bt,\enn  then  there exists a
$T_1\in(0,T_0]$ such that (\ref{z1}) holds for $T=T_1.$

 We set \bnn \notag T^* =\sup\left\{T\,\left|\, \sup_{t\in [0,T]}\|(\n-1,u,\te-1,H)\|_{H^3}<\infty\right\},\right.\enn  and \be \la{s1}T_*=\sup\{T\le T^* \,|\,{\rm (\ref{z1}) \
holds}\}.\ee Then $ T^*\ge T_* \geq T_1>0.$
 Next, we claim that
 \be \la{s2}  T_*=\infty.\ee  Otherwise,    $T_*<\infty.$
Hence, by (\ref{z01}), Proposition \ref{pr1} tells us  (\ref{zs2})
  holds for all $0<T<T_*,$  which  implies
    Lemmas \ref{le11}--\ref{sq91} still hold for all  $0< T< T_* .$
Note here that all  constants $C$  in  Lemmas \ref{le11}--\ref{sq91}
depend  on $T_*  $ and $\inf\limits_{x\in \Omega}\n_0(x)$, are in fact independent  of  $T$.
Then,  we claim that  there
exists a positive constant $\tilde{C}$ which may  depend  on $T_* $
and $\inf\limits_{x\in \Omega}\n_0(x)$   such that, for all  $0< T<
 T_*,$  \be\la{y12}\ba \sup_{0\le t\le T}
\| \n-1\|_{H^3}   \le \tilde{C},\ea \ee which together with Lemmas \ref{le9-1}, \ref{pe1},  \ref{sq90},  \eqref{mn2}, and (\ref{3.1}) infers
$$\|(\n(x,T_*)-1,u(x,T_*),\te(x,T_*)-1,H(x,T_*))\|_{H^3}
 \le \tilde{C},$$
 $$\inf_{x\in \Omega}\n(x,T_*)>0,\quad\inf_{x\in \Omega}\te(x,T_*)>0.$$
Lemma \ref{th0} thus implies that there exists some $T^{**}>T_*,$  such that
(\ref{z1}) holds for $T=T^{**},$   which contradicts (\ref{s1}).
Hence, (\ref{s2}) holds. This as well as  Lemmas \ref{th0}, \ref{a13.1}, \ref{le8}, and Proposition \ref{pr1}, thus concludes the proof of   Proposition \ref{pro2}.

 Finally, it remains to prove (\ref{y12}). By $(\ref{h0})_3$ and (\ref{2.1}), one can define
 \be\ba\notag
\sqrt\n\dot\theta(\cdot,0)\triangleq &- \frac{\ga-1}{R} \rho_0^{-1/2}
\left(\ka\Delta\te_0-R\rho_0 \theta_0 \div u_0\right)\\
&+ \frac{\ga-1}{R} \rho_0^{-1/2}
\left(\lambda (\div u_0)^2+2\mu |\mathfrak{D}(u_0)|^2+\nu|\curl H_0|^2\right),
 \ea\ee
 which along with  (\ref{2.1}) yields
 \be \la{ssp91}\|\sqrt\n\dot\theta(\cdot,0)\|_{L^2}\le \tilde{C}.\ee
Then it follows from \eqref{ssp91},  (\ref{3.99}), and Lemma \ref{le11} that
\be  \ba \la{a51}
\sup\limits_{0\le t\le T} \int \n|\dot\te|^2dx+\int_0^T \|\na\dot\te\|_{L^2}^2dt \le \tilde{C},
\ea\ee
which combined with  (\ref{lop4}) and Lemma \ref{le11} shows that
\be\la{sp211}
\sup\limits_{0\le t\le T}\|\na^2 \theta\|_{L^2} \le \tilde{C}.\ee

Next, by  $(\ref{h0})_2$ and  (\ref{2.1}), one can define
\be\ba\notag
u_t(\cdot,0) \triangleq &-u_0\cdot\na u_0-R\n_0^{-1}\na (\n_0\te_0)\\
&+\n_0^{-1}\left( \mu \Delta u_0 + (\mu+\lambda) \na \div u_0 +H_0\cdot\na H_0-\frac{1}{2}\na|H_0|^2\right),
\ea\ee
which as well as (\ref{2.1}) and \eqref{fa} leads to
\be\la{ssp9}
\|\na u_t(\cdot,0)\|_{L^2}+\|\na H_t(\cdot,0)\|_{L^2}\le \tilde C.
\ee
Thus, one deduces from Gr\"{o}nwall's inequality, Lemmas \ref{le11}, \ref{le9-1}, \eqref{eee8}, \eqref{va1},  (\ref{4.052}),  (\ref{s4}), \eqref{pl}, and (\ref{a51})--(\ref{ssp9})  that
\be\la{ssp1}
\sup_{0\le t\le T}(\| u_t\|_{H^1}+\|\na H_t\|_{L^2})+\int_0^T\int (\n|u_{tt}|^2+|H_{tt}|^2)dxdt\le\tilde{C},\ee
which together with  (\ref{va2}), \eqref{d3}, \eqref{sp211}, and (\ref{qq1}) implies
\be\la{sp221} \sup\limits_{0\le t\le T}(\|u\|_{H^3}+\|H\|_{H^3}) \le \tilde{C}.\ee
This combined with Lemma \ref{le11},  \eqref{ex4}, \eqref{nt4}, (\ref{a51}), (\ref{sp211}), (\ref{ssp1}),   and  (\ref{sp221}) gives
\be\la{ssp24} \ia\left(\|\na^3\te\|_{L^2}^2+ \|\nabla u_t\|_{H^1}^2\right)dt\le \tilde{C} . \ee

Now, some standard calculations lead to
 \be\la{sp134}\ba  & \left(\|\na^3 \n\|_{L^2} \right)_t \\
&\le \tilde{C}\left(\| |\na^3u| |\na \n| \|_{L^2}+ \||\na^2u||\na^2
      \n|\|_{L^2}+ \||\na u||\na^3 \n|\|_{L^2} +\| \na^4u \|_{L^2} \right)\\
&\le \tilde{C}\left(\| \na^3 u\|_{L^2}\|\na \n \|_{H^2}+ \| \na^2u\|_{L^3}\|\na^2 \n \|_{L^6}
       +\|\na u\|_{L^\infty}\|\na^3 \n\|_{L^2} +\| \na^4u \|_{L^2}\right) \\
&\le \tilde{C}(1+ \| \na^3\n \|_{L^2}+ \| \na^2 u_t\|^2_{L^2}+ \|\na^3\te\|_{L^2}), \ea\ee
where we have used (\ref{sp221}),  Lemma \ref{le11}, and the following fact:
\be \notag\ba \|\na^2 u\|_{H^2}
&\le\tilde{C}\left(\| \n\dot u \|_{H^2}+\||H||\na H|\|_{H^2}+\|\na P\|_{H^2} +1\right)\\
&\le \tilde{C}(1+\| \n-1\|_{H^2})(\|  u_t \|_{H^2}+\|  u \|_{H^2}\|  \na u \|_{H^2})+\tilde{C}\|H\|_{H^2}\|\na  H \|_{H^2}\\
&\quad+\tilde C(1+\|\n-1\|_{H^2}+\|\te-1 \|_{H^2})(\|\na\n\|_{H^2}+\|\na\te\|_{H^2})+\tilde C\\
& \le \tilde{C} (1+ \|\na^2  u_t\|_{L^2}+\|\na^3 \n \|_{L^2}+\|\na^3 \te \|_{L^2})\ea\ee
owing to (\ref{2tdu}),  \eqref{hs}, \eqref{sp211}, \eqref{ssp1}, \eqref{sp221}, and Lemma \ref{le11}.

Applying Gr\"{o}nwall's inequality to  \eqref{sp134} and using (\ref{ssp24})  show that
\bnn\la{sp26} \sup\limits_{0\le t\le
T}\|\nabla^3  \n\|_{L^2} \le \tilde{C},\enn which together with
(\ref{qq1}) gives (\ref{y12}).
The proof of Proposition \ref{pro2} is completed.
\end{proof}

With  Proposition \ref{pro2} at hand, we are ready to prove  Theorem \ref{th1}.

\begin{proof}[Proof of  Theorem   \ref{th1}]
 Let $(\n_0,u_0,\te_0,H_0)$  satisfying (\ref{co3})--(\ref{co2}) be the initial data in Theorem \ref{th1}.  Assume that  $C_0$  satisfies (\ref{co14}) with \be\notag\ve\triangleq \ve_0/2,\ee
  where  $\ve_0$  is given in Proposition \ref{pr1}.

To begin with, we  construct the approximate initial data $(\n_0^{m,\eta},u_0^{m,\eta}, \te_0^{m,\eta},H_0^{m,\eta})$. For constants
\be\notag
m \in \mathbb{Z}^+,\ \  \eta \in \left(0, \eta_0 \right),\ \  \eta_0\triangleq \min\xl\{1,\frac{1}{2}(\on-\sup\limits_{x\in \O}\n_0(x)) \xr\},
\ee
we set
\begin{align*}
\n_0^{m,\eta} = \frac{\n_0^{m}+ \eta}{1+\eta},\ \  u_0^{m,\eta}=\frac{u_0^m }{1+\eta},\ \  \te_0^{m,\eta}= \frac{\te_0^{m} + \eta}{1+\eta},\ \ H_0^{m,\eta}=\frac{H_0^m }{1+\eta},
\end{align*}
where $\n_0^{m}$, $u_0^m$, $\te_0^m$,  and $H_0^m$ satisfy that
\be\ba\notag
&0 \le \n_0^{m} \in C^{\infty},\ \  \lim_{m \to \infty} \|\n_0^{m} -\rho_0\|_{H^2 \cap W^{2,q}}=0,
\ea\ee
\be\notag
u_0^{m} \in C^{\infty},\ \ u_0^m\cdot n=0 ,\,\,\curl u_0^m\times n=0\,\,\text{on}\,\,\p\O,\ \ \lim_{m \to \infty}\| \tilde{u}_0^m -{u}_0\|_{H^2}=0,
\ee
\be\ba\notag
H_0^{m} \in C^{\infty},\ \ \div H_0^m&=0,\ \ H_0^m\cdot n=0 ,\,\,\curl H_0^m\times n=0\,\,\text{on}\,\,\p\O,\\ &\lim_{m \to \infty}\| \tilde{H}_0^m -{H}_0\|_{H^2}=0,
\ea\ee
and $\te_0^m$ is the unique smooth solution to the following Poisson equation:
\be\notag\begin{cases}
	\Delta \te_0^m=\Delta \tilde{\te}_0^m,&\text{in}\,\,\O,\\
	\na \te_0^m\cdot n=0 ,&\text{on}\,\,\p\O,\\
	 \te_0^m\rightarrow1,\,\,&\text{as}\,\,|x|\rightarrow\infty,
\end{cases}\ee
with $\tilde{\te}_0^m=\tilde\te_0\ast j_{m^{-1}}$, $\tilde\te_0-1$ is the  $H^1$-extension of $\te_0-1$,  and  $j_{m^{-1}}(x)$ is the standard mollifying kernel of width $m^{-1}$.

Then for any $\eta\in (0, \eta_0)$, there exists $m_1(\eta)\ge 0$ such that for $m \ge m_1(\eta)$, the approximate initial data
$(\n_0^{m,\eta},u_0^{m,\eta}, \te_0^{m,\eta},H_0^{m,\eta})$ satisfies
\be \la{de3}\begin{cases}(\n_0^{m,\eta}-1,u_0^{m,\eta}, \te_0^{m,\eta}-1,H_0^{m,\eta})\in C^\infty ,\quad \div H_0^{m,\eta}=0,\\
	\dis \frac{\eta}{2}\le  \n_0^{m,\eta}  <\hat\n,~~\, \frac{\eta}{4}\le \te_0^{m,\eta} \le \hat \te,~~\,\|\na u_0^{m,\eta}\|_{L^2}+\|\na H_0^{m,\eta}\|_{L^2} \le M, \\
	\dis u_0^{m,\eta}\cdot n=0,~~\,\curl u_0^{m,\eta}\times n=0,~~\,\na \te_0^{m,\eta}\cdot n=0,\,\, \text{on}\,\p\O,\\
\dis H_0^{m,\eta}\cdot n=0,~~\,\curl H_0^{m,\eta}\times n=0,\,\, \text{on}\,\p\O,	\end{cases}
\ee
and
 \be\ba \la{de03}
&\lim\limits_{\eta\rightarrow 0} \lim\limits_{m\rightarrow \infty}
\left(\| \n_0^{m,\eta} - \n_0 \|_ {H^2 \cap W^{2,q}}+\| u_0^{m,\eta}-u_0\|_{H^2}\right)=0,\\
&\lim\limits_{\eta\rightarrow 0} \lim\limits_{m\rightarrow \infty}
\left(\| \te_0^{m,\eta}- \te_0  \|_{H^1}+\|H_0^{m,\eta}-H_0\|_{H^2}\right)=0.
\ea\ee
Moreover,  the initial energy $C_0^{m,\eta}$
for $(\n_0^{m,\eta},u_0^{m,\eta}, \te_0^{m,\eta},H_0^{m,\eta}),$ which is defined by  the right-hand side of (\ref{e})
with $(\n_0,u_0,\te_0,H_0)$   replaced by
$(\n_0^{m,\eta},u_0^{m,\eta}, \te_0^{m,\eta},H_0^{m,\eta}),$
satisfies \bnn \lim\limits_{\eta\rightarrow 0} \lim\limits_{m\rightarrow \infty} C_0^{m,\eta}=C_0.\enn
Therefore, there exists  an  $\eta_1\in(0, \eta_0) $
such that, for any $\eta\in(0,\eta_1),$ we can find some $m_2(\eta)\geq m_1(\eta)$  such that   \be \la{de1} C_0^{m,\eta}\le C_0+\ve_0/2\le  \ve_0 , \ee
provided that\be  \la{de7}0<\eta<\eta_1 ,\,\, m\geq m_2(\eta).\ee

 Now, assume that $m,\eta$ satisfy (\ref{de7}),
  Proposition \ref{pro2} combined with (\ref{de1}) and (\ref{de3}) thus implies that the problem (\ref{h0})--(\ref{ch2}) with  initial data $(\n_0^{m,\eta},u_0^{m,\eta}, \te_0^{m,\eta},H_0^{m,\eta})$
has a smooth solution  $(\n^{m,\eta},u^{m,\eta}, \te^{m,\eta},H^{m,\eta}) $
   on $\Omega\times (0,T] $ for all $T>0. $
    Moreover,  (\ref{h8}), \eqref{zs2}, (\ref{a2.112}), \eqref{ae3.7}, and (\ref{vu15})   with $(\n,u,\te,H)$  being replaced by $(\n^{m,\eta},u^{m,\eta}, \te^{m,\eta},H^{m,\eta})$ all hold.

 Next, for the initial data $(\n_0^{m,\eta},u_0^{m,\eta}, \te_0^{m,\eta},H_0^{m,\eta})$, the function $\tilde g$ in (\ref{co12})  is
 \be \la{co5}\ba \tilde g & \triangleq(\n_0^{m,\eta})^{-1/2}\left(-\mu \Delta u_0^{m,\eta}-(\mu+\lambda)\na\div
 u_0^{m,\eta}+R\na (\n_0^{m,\eta}\te^{m,\eta}_0)\right)\\
 &\quad-(\n_0^{m,\eta})^{-1/2}(\na\times H_0^{m,\eta})\times H_0^{m,\eta}\\
& = (\n_0^{m,\eta})^{-1/2}\sqrt{\n_0}g+\mu(\n_0^{m,\eta})^{-1/2}\Delta(u_0-u_0^{m,\eta})\\
&\quad+(\mu+\lambda) (\n_0^{m,\eta})^{-1/2} \na \div(u_0-u_0^{m,\eta})+ R(\n_0^{m,\eta})^{-1/2} \na(\n_0^{m,\eta}\te_0^{m,\eta}-\n_0\te_0)\\
&\quad+(\n_0^{m,\eta})^{-1/2}((\na\times H_0)\times H_0-(\na\times H_0^{m,\eta})\times H_0^{m,\eta}),\ea\ee
where in the second equality we have used (\ref{co2}).
Since $g \in L^2,$ it follows from (\ref{co5}),  (\ref{de3}), (\ref{de03}), and  (\ref{co3})  that for any $\eta\in(0,\eta_1),$ there exist some $m_3(\eta)\geq m_2(\eta)$ and a positive constant $C$ independent of $m$ and $\eta$ such that
 \be\la{de4}
 	\|\tilde g\|_{L^2}\le (1+\eta)^{1/2}\|g\|_{L^2}+C\eta^{-1/2}\de(m) + C\eta^{1/2},
 	\ee
with   $0\le\de(m) \rightarrow 0$ as
$m \rightarrow \infty.$ Hence,  for any  $\eta\in(0,\eta_1),$ there exists some $m_4(\eta)\geq m_3(\eta)$ such that for any $ m\geq m_4(\eta)$,
 \be \la{de9}\de(m) <\eta.\ee  We thus obtain from (\ref{de4}) and (\ref{de9}) that
there exists some positive constant $C$ independent of $m$ and $\eta$ such that  \be\la{de14}  \|\tilde g \|_{L^2}\le \|g \|_{L^2}+C,\ee provided that\be \la{de10} 0<\eta<\eta_1,\,\,  m\geq m_4(\eta).\ee

 Now, we   assume that $m,$  $\eta$ satisfy (\ref{de10}).
 We thus deduce from (\ref{de3})-- (\ref{de1}),  (\ref{de14}), Proposition \ref{pr1}, 
 and Lemmas \ref{le8}, \ref{le11}--\ref{sq91} that for any $T>0,$
 there exists some positive constant $C$ independent of $m$ and $\eta$ such that
 (\ref{h8}), (\ref{zs2}),   (\ref{a2.112}),  \eqref{ae3.7}, (\ref{vu15}), \eqref{lee2}, \eqref{qq1},  (\ref{va5})--(\ref{vva5}), \eqref{nq1}, (\ref{y2}),  (\ref{eg17}),
  and  (\ref{egg17})  hold for  $(\n^{m,\eta},u^{m,\eta}, \te^{m,\eta},H^{m,\eta}) .$
   Then passing  to the limit first $m\rightarrow \infty,$ then $\eta\rightarrow 0,$
 along with standard arguments leads to  that there exists a solution $(\n,u,\te,H)$ of the problem (\ref{h0})--(\ref{ch2})
   on $\Omega\times (0,T]$ for all $T>0$, which   satisfies  (\ref{h8}), (\ref{a2.112}),    \eqref{ae3.7},  (\ref{vu15}),  \eqref{lee2}, \eqref{qq1}, (\ref{va5})--(\ref{vva5}), \eqref{nq1}, (\ref{y2}),  (\ref{eg17}), (\ref{egg17}),
   and  the estimates of $A_i(T)\,(i=1,\cdots,4)$ in
   (\ref{zs2}). Therefore,    $(\n,u,\te,H)$ satisfies  (\ref{h8}) and \eqref{h9}.

Finally, since the proof of the uniqueness of $(\n,u,\te,H)$ is similar to that  of \cite[Theorem 1]{choe1}, we omit it here for simplicity. To finish the proof of Theorem \ref{th1}, it remains to prove (\ref{h11}).
On the one hand, one can rewrite $\eqref{h0}_1$  as
\be\la{vvv}(\n-1)_t+\div((\n-1)u)+\div u=0.\ee
Multiplying (\ref{vvv}) by $4(\n-1)^3$
and integrating by parts, one derives that, for $t\ge 1,$
\be  \notag(\|\n-1\|_{L^4}^4)'(t) =-3\int(\n-1)^4\div u dx-4\int(\n-1)^3\div u dx,\ee
which combined with (\ref{vu15}) implies that
\be \la{hhh}
\int_{1}^{\infty}|(\|\n-1\|_{L^4}^4)'(t)| dt\le C\int_{1}^{\infty}\|\n-1\|_{L^4}^4dt +C\int_{1}^{\infty} \|\na u\|_{L^4}^4 dt \le C.\ee
On the other hand, we deduce from $A_i(T)(i=2,3,4)$ in  (\ref{zs2}) and (\ref{vu15}) that
\be\la{mmq}\ba
\int_1^\infty |(\|\na u\|_{L^2}^2)'(t)|dt &=2\int_1^\infty \left|\int \pa_j u^i\pa_j u^i_t dx\right|dt\\
&=2\int_1^\infty \left|\int \pa_j u^i\pa_j(\dot u^i-u^k\pa_k u^i ) dx\right|dt \\&=\int_1^\infty \left|\int (2\pa_j u^i\pa_j \dot u^i-2\pa_j u^i\pa_j u^k\pa_k u^i+|\na u|^2\div u ) dx\right|dt\\&\le C\int_1^\infty \left(\|\na u\|_{L^2}\|\na \dot u\|_{L^2}+\|\na u\|_{L^3}^3\right)dt\\&\le C\int_1^\infty \left(\|\na  \dot u\|_{L^2}^2+\|\na u\|^2_{L^2}+\|\na u\|_{L^4}^4\right)dt \le C,\\
\ea\ee
\be \la{vu34}\ba \int_1^\infty|\left(\|\na\te\|_{L^2}^2\right)'(t)
|dt&= 2\int_1^\infty\left| \int \na\te\cdot \na\te_tdx\right|dt\\
&\le C\int_1^\infty\left(\|\na  \te \|_{L^2}^2+\|\na \te_t\|^2_{L^2}\right)dt \le C,\ea\ee
and
\be \la{vuu}\ba \int_1^\infty|\left(\|\na H\|_{L^2}^2\right)'(t)
|dt&= 2\int_1^\infty\left| \int \na H\cdot \na H_tdx\right|dt\\
&\le C\int_1^\infty\left(\|\na  H \|_{L^2}^2+\|\na H_t\|^2_{L^2}\right)dt \le C.\ea\ee
It thus follows from \eqref{vu15}, \eqref{hhh}--\eqref{vuu},  and $A_2(T)$ in (\ref{zs2}) that
\be\notag
\lim_{t\rightarrow\infty}
\left(\|\rho-1\|_{L^4}+\|\na u\|_{L^2}+\|\na\te\|_{L^2}+\|\na H\|_{L^2}\right)=0,
\ee
which together with \eqref{a2.112}, \eqref{h8}, and (\ref{vu15}) concludes (\ref{h11}).
 The proof of Theorem \ref{th1} is completed.
\end{proof}


\end{document}